\documentclass{article}
\usepackage{graphicx} 
\usepackage{fancyhdr}
\usepackage{latexsym}
\usepackage{mathrsfs}
\usepackage{cancel}
\usepackage{color}
\usepackage{multirow}
\usepackage{eso-pic}
\usepackage{dsfont}
\usepackage{amssymb}
\usepackage{graphicx}
\usepackage{setspace}
\usepackage{geometry}
\usepackage{amsthm}
\usepackage{hyperref}
\usepackage[intlimits]{amsmath}
\usepackage{hyperref}
\usepackage{tikz-cd}
\usepackage{amssymb,amsfonts}
\usepackage[all,arc]{xy}
\usepackage{enumerate}
\usepackage{mathrsfs}
\usepackage{tikz}
\usepackage{bbm}
\usepackage{graphicx}
\usepackage{mathtools}
\usetikzlibrary{braids}

\newcommand{\wt}{\widetilde}
\newcommand{\wh}{\widehat}

\newcommand{\End}{\mathrm{End}}

\newcommand{\Hom}{\mathrm{Hom}}

\newcommand{\Mod}{\mathrm{-Mod}}

\newcommand{\Gr}{\mathrm{Gr}}

\newcommand{\lp}{\left(}
\newcommand{\rp}{\right)}

\newcommand{\QCoh}{\mathrm{QCoh}}
\newcommand{\IndCoh}{\mathrm{IndCoh}}

\newcommand{\Rep}{\mathrm{Rep}}

\newcommand{\lpp}{(\!(}
\newcommand{\rpp}{)\!)}
\newcommand{\lbb}{[\![}
\newcommand{\rbb}{]\!]}

\newcommand{\pd}{{\partial}}

\newcommand{\lag}{\langle}
\newcommand{\rag}{\rangle}

\newcommand{\C}{\mathbb C}

\newcommand{\fd}{\mathfrak{d}}
\newcommand{\fg}{\mathfrak{g}}

\newcommand{\fL}{\mathfrak{L}}
\newcommand{\fD}{\mathfrak{D}}

\newcommand{\fG}{\mathfrak{G}}

\newcommand{\CC}{{\mathcal C}}

\newcommand{\CE}{{\mathcal E}}
\newcommand{\CF}{{\mathcal F}}
\newcommand{\CG}{{\mathcal G}}
\newcommand{\CH}{{\mathcal H}}

\newcommand{\CK}{{\mathcal K}}
\newcommand{\CL}{{\mathcal L}}

\newcommand{\CN}{{\mathcal N}}
\newcommand{\CO}{{\mathcal O}}
\newcommand{\CP}{{\mathcal P}}

\newcommand{\CR}{{\mathcal R}}
\newcommand{\CS}{{\mathcal S}}
\newcommand{\CT}{{\mathcal T}}

\newcommand{\CY}{{\mathcal Y}}

\newcommand{\be}{\begin{equation}}
\newcommand{\ee}{\end{equation}}

\newcommand{\btik}{\begin{tikzcd}}
\newcommand{\etik}{\end{tikzcd}}
\newcommand{\mf}{\mathfrak}

\newcommand{\Bun}{\textnormal{Bun}_G}
\newcommand{\BunL}{\textnormal{Bun}_{{}^L\!G}}

\newcommand{\gout}{\fg_{\textnormal{out}}}

\newcommand{\rltensor}{{}^r\!\!\otimes^l}

\newcommand{\sttensor}{{}^s\!\!\otimes^t}
\newcommand{\sstensor}{{}^s\!\!\otimes^s}
\newcommand{\gkl}{\mathfrak{g}(\CK)\lbb\lambda\rbb}
\newcommand{\gol}{\mathfrak{g}(\CO)\lbb\lambda\rbb}

\newtheorem{Def}{Definition}[section]
\newtheorem{Thm}[Def]{Theorem}
\newtheorem{Prop}[Def]{Proposition}
\newtheorem{Cor}[Def]{Corollary}
\newtheorem{Lem}[Def]{Lemma}
\newtheorem{Rem}[Def]{Remark}

\newtheorem{Exp}[Def]{Example}

\numberwithin{equation}{section}

\title{Quantum groupoids from moduli spaces of $G$-bundles}
\author{Raschid Abedin \& Wenjun Niu}
\date{\today}

\begin{document}

\maketitle

\abstract{In a previous work, we have constructed the Yangian $Y_\hbar (\fd)$ of the cotangent Lie algebra $\fd=T^*\fg$ for a simple Lie algebra $\fg$, from the geometry of the equivariant affine Grassmanian associated to $G$ with $\fg=\mathrm{Lie}(G)$. In this paper, we construct a quantum groupoid $\Upsilon_\hbar^\sigma (\fd)$ associated to $\fd$ over a formal neighbourhood of the moduli space of \(G\)-bundles and show that it is a dynamical twist of \(Y_\hbar(\fd)\). Using this dynamical twist, we construct a dynamical quantum spectral $R$-matrix, which essentially controls the meromorphic braiding of $\Upsilon_\hbar^\sigma (\fd)$. 

This construction is motivated by the Hecke action of the equivariant affine Grassmanian on the moduli space of \(G\)-bundles in the setting of coherent sheaves. Heuristically speaking, the quantum groupoid $\Upsilon_\hbar^\sigma (\fd)$ controls this action at a formal neighbourhood of a regularly stable $G$-bundle. From the work of Costello-Witten-Yamazaki, it is expected that this Hecke action should give rise to a dynamical integrable system. Our result gives a mathematical confirmation of this and an explicit $R$-matrix underlying the integrability. 
}

\newpage

\tableofcontents

\newpage

\section{Introduction}

Let $G$ be a simple complex Lie group with Lie algebra $\fg$ and let $\fd\coloneqq T^*\fg=\fg\ltimes \fg^*$ be the \textit{cotangent Lie algebra} of $\fg$, i.e.\ the Lie algebra of the cotangent Lie group \(T^*G\). In \cite{abedin2024yangian}, we constructed a quantum group $Y_\hbar (\fd)$, which we called the Yangian of \(\fd\), with the following properties:

\begin{enumerate}
    \item As a Hopf algebra, it is the, up to isomorphism, unique bigraded quantization of the Lie bialgebra structure on $\fd(\CO)$ given by the 1-cocycle \(\delta\) with respect to Yang's r-matrix $\gamma (t_1,t_2)=\frac{C}{t_1-t_2}$. Here, $C$ is the quadratic Casimir element of $\fd$ with respect to the canonical evaluation pairing \(\langle,\rangle \colon \fd \times \fd \to \C\). 

    \item There is a spectral $R$-matrix $\CR_\gamma(z)\in Y_\hbar (\fd)^{\otimes 2}\lbb z^{-1}\rbb$ satisfying the quantum Yang-Baxter equation and such that
    \be
        \CR_\gamma(z)\lp \lp \tau_z\otimes 1\rp \Delta_\hbar^\gamma \rp \CR_\gamma(z)^{-1}=(\tau_z\otimes 1) \Delta_\hbar^{\gamma,\textnormal{op}}
    \ee
    holds in the dense subalgebra quantizing \((\fd[t],\delta)\). Here, \(\Delta^\gamma_\hbar\) is the comuliplication of \(Y_\hbar(\fd)\).
    
 \end{enumerate}

In this paper, we produce a dynamical version of \(Y_\hbar(\fd)\) over the formal neighborhood of a regularly stable point in the moduli space of \(G\)-bundles, by adapting the constructions from \cite{abedin2024yangian} to the language of Lie algebroids and by utilizing the classical dynamical spectral \(r\)-matrix constructed in \cite{felder_kzb} (see also \cite{abedin2024r}).

Our set-up is as follows. Let $\Bun (\Sigma)$ be the moduli stack of $G$-bundles over a smooth complex irreducible projective curve $\Sigma$. If one fixes a point $p\in \Sigma$, then the space $\Bun (\Sigma)$ can be described as a double quotient via the so-called ``loop-group uniformization":
\be
\Bun (\Sigma)= G(\CO)\!\setminus \!G(\CK)/ G(\Sigma^\circ),\qquad \Sigma^{\circ} \coloneqq\Sigma\setminus \{p\}.
\ee
Let $\CP\in \Bun (\Sigma)$ be a regularly stable point, i.e.\ a \(G\)-bundle whose automorphisms are precisely given by the center of \(G\), and let $B=\wh\Bun (\Sigma)_{\CP}$ be the formal neighbourhood around this point. Then the canonical projection $G(\CK)\to \Bun (\Sigma)$ given by loop-group uniformization admits a splitting $\sigma\colon B\to G(\CK)$ over \(B\). Let $R=\CO (B)$ be the topological ring of functions on $B$ and we chose local coordinates $\lambda=(\lambda_1,\ldots, \lambda_n)$ so that $R\cong \C\lbb\lambda\rbb$.

Combining the construction from \cite{felder_kzb,abedin2024r} and the skew-symmetrization method from \cite{abedin2024yangian}, we will show that one can assign to $(\Sigma, \CP, \sigma)$ a classical spectral dynamical $r$-matrix
\be
\rho(t_1,t_2;\lambda)\in \frac{C}{t_1-t_2} + (\fd \otimes \fd)[\![t_1,t_2]\!][\![\lambda]\!]
\ee 
i.e.\ a solution to the dynamical classical Yang-Baxter equation:
\begin{equation}\label{eq:intro_dynamicalCYBE_rho}
    \begin{split}
        &[\rho^{(12)}(t_1,t_2;\lambda),\rho^{(13)}(t_1,t_3;\lambda)] + [\rho^{(12)}(t_1,t_2;\lambda),\rho^{(23)}(t_2,t_3;\lambda)] + [\rho^{(32)}(t_3,t_2;\lambda),\rho^{(13)}(t_1,t_3;\lambda)] \\&= \sum_{\alpha = 1}^N \left(\omega^{(1)}_\alpha \partial_\alpha \rho^{(23)}(t_2,t_3;\lambda) - \omega_\alpha^{(2)}(t_2;\lambda)\partial_\alpha \rho^{(13)}(t_1,t_3;\lambda) +\omega_\alpha^{(3)}(t_3;\lambda)\partial_\alpha \rho^{(12)}(t_1,t_2;\lambda)\right).
    \end{split}
\end{equation}
This $\rho$ induces the structure of a Lie bialgebroid $\fd (\CO)\lbb\lambda\rbb$ over $B$. 

Our first result is that this Lie bialgebroid admits a canonical quantization to a Hopf algebroid.

\begin{Thm}

    Associated to the data of $(\Sigma, \CP, \sigma)$, there exists an explicit Hopf algebroid $\Upsilon_\hbar^\sigma (\fd)$ over $B$ that quantizes the Lie bialgebroid structure on $\fd (\CO)\lbb\lambda\rbb$ defined by \(\rho\). 
    \hfill \qedsymbol
\end{Thm}

Our second result is that this Hopf algebroid $\Upsilon_\hbar^\sigma (\fd)$ can be obtained from the trivial Hopf algebroid $\Upsilon_\hbar (\fd):=Y_\hbar (\fd)\otimes D_B$ via a dynamical twist, where $D_B= U_R (T_B)$ is the universal enveloping algebra of the tangent Lie algebroid $T_B$ of $B$, i.e.\ its the algebra of differential operators on \(B\). More precisely, we show:

\begin{Thm}\label{thm:intro_dynamical_twist}
    There exists a dynamical quantum twist
    \be
\CF\in\Upsilon_\hbar (\fd)\sttensor_R\Upsilon_\hbar (\fd)
    \ee
that induces a twist-equivalence of the Hopf algebroids \(\Upsilon_\hbar (\fd)\) and \(\Upsilon_\hbar^\sigma (\fd)\) over $R$, i.e.:
\begin{enumerate}
    \item \(
       \Upsilon_\hbar (\fd)\cong \Upsilon_\hbar^\sigma (\fd)\) as \(R\)-algebras;
        
    \item \(\CF^{-1}\Delta_{\hbar}\CF = \Delta^\sigma_\hbar\);

    \item \((\Delta_\hbar \otimes 1)(\CF)\CF^{(12)} = (1 \otimes \Delta_\hbar)(\CF)\CF^{(23)}\). \hfill \qedsymbol
\end{enumerate}
\end{Thm}

A more careful analysis of the twist \(\CF\) shows that it decomposes as \(\CF = F \Theta\), where
    \be
    F\in Y_\hbar (\fd)\lbb\lambda\rbb\sttensor Y_\hbar (\fd)\lbb\lambda\rbb~~  \textnormal{ and }~~  \Theta=\exp \left(\hbar \sum_{\alpha = 1}^N \pd_{\alpha}\otimes \omega_\alpha\right)
    \ee
    for \(\pd_\alpha = \pd/\pd \lambda_\alpha\) and some $\omega_\alpha\in \fg^*(\CO)\lbb\lambda\rbb$. 
Moreover, the tensor $F$ satisfies a dynamical version of the twisting equation in Theorem \ref{thm:intro_dynamical_twist}.3, namely: 
    \be
    (1\otimes \Delta_\hbar) (F)F^{(23)}(\lambda)=(\Delta_\hbar\otimes 1)(F) F^{(12)}(\lambda+\hbar \omega^3).
    \ee
Here, we used dynamical notation, which pragmatically in our setting boils down to
\begin{equation}
    F^{(12)}(\lambda+\hbar \omega^3) = \Theta^{(13),-1}\Theta^{(23),-1}F^{(12)}(\lambda)\Theta^{(23)}\Theta^{(13)}.
\end{equation}
Using this properties of the twist, we can construct a dynamical spectral $R$-matrix from the spectral $R$-matrix of $Y_\hbar (\fd)$ and \(\CF\). 

\begin{Thm}\label{thm:intro_quantum_rmatrix}
    Let $\CR(z, w;\lambda)$ be the $Y_\hbar (\fd)\otimes Y_\hbar (\fd)$-valued tensor series over $B$ defined by
    \be
\CR(z, w;\lambda):=\lp \tau_z\otimes \tau_w (F)\rp^{(21), -1} \CR_\gamma(z-w) (\tau_z\otimes \tau_w) (F). 
    \ee
This series is well-defined when acting on a tensor product of smooth modules of $Y_\hbar (\fd)$, and satisfies the following dynamical version of the quantum Yang-Baxter equation:
\be
\begin{aligned}
    \CR^{(12)}(z_1, z_2; \lambda) &\CR^{(13)}(z_1, z_3; \lambda+\hbar \omega^{(2)})\CR^{(23)} (z_2, z_3; \lambda)\\ &=\CR^{(23)} (z_2, z_3; \lambda+\hbar\omega^{(1)}) \CR^{(13)}(z_1, z_3; \lambda)\CR^{(12)}(z_1, z_2; \lambda+\hbar\omega^{(3)}). 
\end{aligned}
\ee
\hfill \qedsymbol
\end{Thm}

One should think of $z, w$ as coordinates in a neighborhood of $p$, with which we defined the projection of $G(\CK)$ to $\Bun (\Sigma)$ and its local splitting $\sigma$. 

\begin{Rem}\label{Rem:topology}
    We must remark that all algebraic operations involving formal series in \(t, \lambda = (\lambda_1,\dots,\lambda_n)\), and \(\hbar\) are taken to be topologically with respect to the corresponding limit topologies. For instance, \(U(-)\) denotes the completed universal enveloping algebra, tensor products \(-\otimes-\) denotes the completed tensor product, and \((-)^*\) denotes the continuous dual. In particular, identities like \(\fd(\CO) \otimes \fd(\CO) \cong (\fd \otimes \fd)[\![t_1,t_2]\!]\), \(\fg[\![t]\!]^* \cong t^{-1}\fg[t^{-1}]\), and \(Y_\hbar(\fd)[\![\lambda]\!] = Y_\hbar(\fd)[\![\lambda]\!]\) hold. The reader is cautioned to keep this convention in mind while reading this document.
\end{Rem}

\subsection{Mathematical motivation}\label{subsec:mathmotive}
The mathematical motivation of this work comes from two places. Firstly in understanding the algebraic structures underlying Hecke modifications, a vital part of the geometric Langlands correspondence, better and secondly in finding an \(\hbar\)-deformation of the dynamical \(r\)-matrix and the KZB connection constructed from them in \cite{felder_kzb}. Both of these motivations are closely related: the KZB connection provides a deformation of the quadratic quantized Hitchin Hamilitonians and these are integral to the Beilinson-Drinfeld approach to the geometric Langlands correspondence; see \cite{beilinson_drinfeld_quantization}.

\subsubsection{Hecke modifications and the geometric Langlands correspondence}
The geometric Langlands correspondence (GLC in short) \cite{beilinson1991quantization}, now given a full proof by \cite{GLC1, GLC2, GLC3, GLC4, GLC5}, roughly states an equivalence of derived categories (see \cite{frenkel2002geometric} for details)
\be\label{eq:GLCglobal}
\mathrm{DMod}_{\frac{1}{2}} \lp \Bun (\Sigma) \rp\simeq  \IndCoh_{\text{nilp}} \lp \mathrm{LocSys}_{{}^L \! G} (\Sigma) \rp,
\ee
which is compatible with the action of the symmetric monoidal category $\mathrm{Rep}({}^L\! G)$ (or its factorization version over $\Sigma$) on both sides. The action on the right is given by the localization functor \cite[Section 4.3]{gaitsgory2013outline}, and on the left via the Satake equivalence
\be\label{eq:Satake}
\CS: \mathrm{Perv} (G(\CO)\!\setminus\! \Gr_G)\simeq \Rep ({}^L\! G)
\ee
and the so-called \textit{Hecke modifications}.

Hecke modifications can be understood geometrically via the loop-group uniformization of $\Bun (\Sigma)$. If one represent the stack as
\be
\Bun (\Sigma)= G(\CO)\!\setminus \! G(\CK)/G(\Sigma^\circ), \qquad \Sigma^\circ=\Sigma\setminus \{p\},
\ee
then the action of $\mathrm{Perv} (G(\CO)\!\setminus\! \Gr_G)$ is given geometrically by push-pull along the diagram:
\be\label{eq:Heckemod}
\btik
 &  G(\CO)\!\setminus \! (G(\CK)\times_{G(\CO)} G(\CK))/G(\Sigma^o)\arrow[dr] \arrow[dl]& \\
G(\CO)\!\setminus \! \Gr_G\times \Bun & & \Bun 
\etik.
\ee
The original version of the geometric Langlands correspondence, as stated in \cite{beilinson1991quantization}, is to find a sheaf $\CF$ on $\Bun (\Sigma)$ with the property that it is an ``eigenvector" of the action of $ \mathrm{Perv} (G(\CO)\!\setminus\! \Gr_G)$, namely:
\be
\ell * \CF \cong \CS (\ell) \otimes \CF,\qquad \forall \ell \in \mathrm{Perv} (G(\CO)\!\setminus\! \Gr_G). 
\ee

As a classical limit of the GLC \eqref{eq:GLCglobal}, it is conjectured in \cite{Donagi:2006cr} and \cite{Ben-Zvi:2016mrh} that there should be an equivalence of the form (coherent GLC):
\be\label{eq:coherentGLC}
\QCoh (T^* \Bun (\Sigma))\simeq \QCoh (T^* \BunL (\Sigma)).
\ee
It is likely that one needs to modify the categories above (just as in the case of the original geometric Langlands conjecture), and one candidate was proposed in \cite{Padurariu:2024zff}. 

In the above generalization, the action of Hecke modifications is intuitively given by a similar diagram as in equation \eqref{eq:Heckemod}, replacing all stacks involved by their cotangent stack. However, due to infinite-dimensionality of the stacks involved, it is not easy to define the category of sheaves on $T^*\Bun (\Sigma)$ or the push-pull functors on coherent sheaves. 

Our work was motivated by the study of Hecke modifications in a slightly different setting (namely $\CN=2$ pure gauge theory). More specifically, given $p\in \Sigma$, one can consider Hecke modifications at $p$ in equation \eqref{eq:Heckemod}, applied to $\QCoh$ directly rather than $\mathrm{DMod}$. This should define a monoidal functor
\be\label{eq:CoherentHecke}
\CF_p: \QCoh_{G(\CO)} (\Gr_G)\longrightarrow \End (\QCoh (\Bun (\Sigma))),
\ee
where the RHS is the category of endo-functors of $\QCoh (\Bun (\Sigma))$. The problem that motivated this paper is the question of what algebraic structure controls this monoidal functor. Understanding this problem will help understand the action of Hecke modifications in the coherent-setting, as well as the coherent GLC conjecture \eqref{eq:coherentGLC}. 

Again due to non-smoothness and infinite-dimensionality of the stacks involved, this question is not easy to answer. In this paper, we simplify the situation by working perturbatively, namely at the formal neighborhood \(B\) of $\Bun (\Sigma)$ around a regularly stable \(G\)-bundle \(\CP\), as well as over the formal neighborhood of the identity coset in $\Gr_G$. Very roughly, the Hopf algebroid $\Upsilon_\hbar^\sigma (\fd)$ is what controls the action of $\CF_p$ in the sense that the functor $\CF_p$, restricted to the formal completion $\wh \Gr_G$, factors to modules of this Hopf algebroid. In other words, it induces a functor
\be
\wt \CF_p: \QCoh_{G(\CO)} (\wh \Gr_G)\longrightarrow \Upsilon_\hbar^\sigma (\fd)\Mod
\ee
over \(B\). Moreover, we can understand this Hopf algebroid via a dynamical twist, and construct a quantum integrable system using the Yangian algebra $Y_\hbar (\fd)$. 

In a future project, we will go beyond the perturbative setting, by considering formal completion of $\Gr_G$ along other $G(\CO)$-orbits. We hope that this will lead to applications in the coherent GLC conjecture. 

\subsubsection{Quantization of Hitchin systems and the KZB equation}
The cotangent stack $T^*\Bun (\Sigma)$, also known as the moduli space of \textit{Higgs bundles} or \emph{Hitchins moduli space}, is an important geometric object in the area of classical integrability. It supports a natural integrable system, the \emph{Hitchin system}, first constructed in \cite{Hitchin:1987mz} as follows. Consider the Hitchin fibration
\be
p \colon T^* \Bun (\Sigma)\longrightarrow \bigoplus_{1\leq i\leq n}\Gamma (\Sigma, \Omega_{\Sigma}^{d_i}) \coloneqq H_\Sigma,
\ee
which is induced by the invariant polynomials of \(\fg\). This fibration corresponds to an embedding \(\CO(H_\Sigma) \to \CO( T^* \Bun (\Sigma))\). The image of this map, which consists of the so-called \emph{Hitchin Hamiltonians}, provides a maximal subspace of independent Poisson-commuting functions. 
Here $n=\mathrm{rk}(\fg)$ and $d_i=e_i+1$, where $e_i$ are the exponents of $\fg$. Furthermore, the Hitchin fibration has generically fibers which are essentially abelian and hence this system is in fact algebraically completely integrable. 

In the Beilinson-Drinfeld approach \cite{beilinson_drinfeld_quantization} to the geometric Langlands correspondence, an integral step is the quantization of the Hitchin Hamiltonians to differential operators on \(\Bun(\Sigma)\). The spectrum of these differential operators identifies with subspace of \(({}^L\! G)_{\textnormal{ad}}\)-local systems, the so-called opers, and this fact can be used to obtain a partial correspondence in the vein of \eqref{eq:GLCglobal}.

In genus 0, i.e.\ for \(\Sigma = \mathbb{P}^1\), where one has to add punctures to make the setting interesting, the Hitchin systems are also known as Garnier systems or classical Gaudin systems, and the quantization of these are the so-called Gaudin systems. Similarly, the quadratic Hitchin Hamiltonians are also known as Gaudin Hamiltonians. These admit a natural deformation through the Knizhnik-Zamolodchikov (KZ) connection and the commutativity of these operators follows from the so-called KZ equation. The KZ connection can thereby be described explicitly using Yang's \(r\)-matrix and the KZ equation follows from the classical Yang-Baxter equation. 

In higher genus, the quantized Hitchin Hamiltonians can be deformed to the Knizhnik-Zamolodchikov-Bernard (KZB) connection. These can be described explicitly using a classical dynamical spectral \(r\)-matrix and the KZB equation follows from the dynamical classical Yang-Baxter equation; see \cite{felder_kzb}. In \cite{felder_kzb}, the question for a quantization of this dynamical classical \(r\)-matrix was posed and it is natural to ask if the arising quantization gives rise to a quantum KZB equation, i.e.\ a particular difference equation that quantizes the KZB equation.

In this work, we provide an answer to the first part of the questions posed in \cite{felder_kzb}: we construct a quantization of the skew-symmetrization \(\rho\) of the original classical dynamical \(r\)-matrix, namely \(\CR(z,w;\lambda)\) from Theorem \ref{thm:intro_quantum_rmatrix}. Observe that the \(r\)-matrix from \cite{felder_kzb} is not skew-symmetric and therefore some kind of skew-symmetrization is expected in order to make a quantization to a quantum \(R\)-matrix possible. In fact, skew-symmetry over \(\fd\) instead of over the simple Lie algebra \(\fg\) is necessary in the proper classical dynamical setting in higher genus, since the classification from \cite{etingof_varchenko} suggests that skew-symmetric dynamical \(r\)-matrices over simple Lie algebras are a genus one phenomenon. 

In future work, it would be interesting if our setting admits a quantum KZB equation and if this leads to applications in the GLC.

\subsection{Physical motivation}

The physics motivation behind our work is the study of integrable systems arising from 4d holomorphic-topological (HT) theories, and is closely related to the mathematical set-ups in Section \ref{subsec:mathmotive}. In \cite{costello2014integrable}, the author explained that a mixed holomorphic-topological field theory $\CT^{4d}$ in 4 dimensions can give rise to integrable lattice models. This was further ellaborated and generalized in \cite{costello2018gauge, costello2018gauge2, costello2019gauge}. The rough idea is as follows. 

\begin{enumerate}
    \item The theory $\CT^{4d}$ gives rise to a category, physically the category of line operators $\CL$. This $\CL$ has the structure of a chiral monoidal category. The monoidal product comes from operator collision in the topological direction, and the chiral structure comes from operator product expansion (OPE) in the holomorphic direction. 

    \item Take any smooth complex curve $\Sigma$, dimensional reduction of $\CT$ along $\Sigma$ gives rise to a 2d TQFT, usually denoted by $\int_{\Sigma}\CT^{4d}$.  According to the cobordism hypothesis \cite{Lurie:2009keu}, there is a category $\CC_\Sigma$ associated to this TQFT, which is physically the category of boundary conditions. 

    \item One can construct the state space of $\int_{\Sigma}\CT^{4d}$ by compactifying the theory further on $S^1$, and inserting various interfaces at different points of $S^1$. In particular, given a collection of objects $\{\ell_i\}_{1\leq i\leq n}$ in $\CL$, inserting them at various points $\{p_i\}\subseteq \Sigma$ one obtains such a collection of interfaces for $\int_{\Sigma}\CT^{4d}$, and therefore a state space
    \be
\CH (\ell_1, \ldots, \ell_n):= \text{ state space of } \int_{\Sigma}\CT^{4d} \text{ with insertions } \ell_i \text{ along } S^1 \text{ and at } \{p_i\}.
    \ee

    \item This space $\CH (\ell_1, \ldots, \ell_n)$ does not depend on the insertion points on $S^1$, and in fact is equal to $\CH (\ell_1* \ell_2*\cdots *\ell_n)$, where one first collides all the interfaces $\ell_i$ to form a single interface and then inserts it at any point on $S^1$. Moreover, if the category $\CC_\Sigma=\mathrm{Vect}$, then one can show that
    \be
\CH (\ell_1, \ldots, \ell_n)=\bigotimes_i \CH (\ell_i).
    \ee

\item Taking the state space $\CH (\ell_1, \ell_2)$ with two insertions, there is an $R$-matrix
\be
\CR(p_1, p_2)\in \End (\CH(\ell_1, \ell_2))=\End (\CH(\ell_1)\otimes \CH (\ell_2)), \qquad p_1\ne p_2\in \Sigma
\ee
which depends meromorphically on $(p_1,p_2)$ and satisfies the quantum Yang-Baxter equation. The transfer matrix of the integrable lattice model is obtained from this meromorphic $R$-matrix. 

\item The origin of the matrix $\CR(p_1, p_2)$ is the chiral-monoidal structure of $\CL$. More precisely, this chiral-monoidal structure should imply the existence of functorial isomorphisms
\be
(\ell_1, p_1) * (\ell_2, p_2)\cong (\ell_2, p_2) * (\ell_1, p_1), \qquad \text{ whenever } p_1\ne p_2.
\ee
The $R$-matrix $\CR(p_1, p_2)$ is the action of this natural isomorphism on the associated state spaces. The quantum Yang-Baxter equation comes from the compatibility of the above functorial isomorphisms with associativity of the monoidal structure. 
    
\end{enumerate}

This strategy is then applied to the famous 4-dimensional version of Chern-Simons theory for a simple Lie algebra $\fg$, with $\Sigma=\mathbb P^1$ and singularity at $\infty$. It was shown that in this case $\CC_\Sigma$ is indeed $\mathrm{Vect}$, and the associated integrable system is governed by the Yangian of \(\fg\), first constructed in \cite{Drinfeld:1985rx}.

The author of  \cite{kapustin2006holomorphic} observed that any 4d $\CN=2$ theory admits a holomorphic-topological twist. For a pure gauge theory with gauge group $G$, it was argued that the category of line operators $\CL$ should have the form
\be
\CL= \QCoh (\mathrm{Maps}(\mathbb B, BG))=\QCoh \lp G(\CO)\! \setminus\! \Gr_G\rp,
\ee
Here, $\mathbb B=\mathbb D\coprod_{\mathbb D^\times} \mathbb D$ is the formal bubble. Moreover, the category $\CC_\Sigma$ in this case is given by
\be
\CC_\Sigma=\QCoh (\mathrm{Maps}(\Sigma, BG))=\QCoh (\Bun (\Sigma)). 
\ee
Under this identifications, the procedure of Step 3, namely the procedure of producing interfaces (which are objects in $\End (\CC_\Sigma)$) from line operators (which are objects in $\CL$), is precisely given by Hecke modifications, namely the functor in equation \eqref{eq:CoherentHecke}. 

According to the general arguments of \cite{costello2014integrable}, whenever the category $\QCoh (\Bun (\Sigma))$ is trivial, one should obtain such an integrable lattice model. In this case, the functor $\CF_p$ from Hecke modifications simply provides a monoidal fiber functor for $\CL$, and it should be compatible with factorization structure (namely inserting at various points on $\Sigma$). This compatibility should ultimately leads to the $R$-matrix of Step 5. 

In \cite{abedin2024yangian}, based on the geometric study of $\CL$ in \cite{cautis2019cluster, niu2022local, cautis2023canonical}, we constructed a spectral $R$-matrix associated to the perturbative sector of this holomorphic-topological theory. The $R$-matrix is the universal $R$-matrix of the quantum group $Y_\hbar (\fd)$, the Yangian of $\fd=T^*\fg$ (or its twisted versions).  

However, for general $\Sigma$, the category $\CC_\Sigma=\QCoh (\Bun (\Sigma))$ is not $\mathrm{Vect}$. In this case, it was argued in \cite[Section 11]{costello2018gauge} that the associated integrable system is a dynamical system, where the dynamical parameter takes values in the moduli space of $G$-bundles. Unfortunately, when this moduli space is not smooth or is a stack, the most general definition of what this means is still mysterious (at least to us). The result of the current paper, as was the prediction of \cite[Section 11]{costello2018gauge} , is that perturbatively, as long as one is working over a neighborhood $B$ of $\Bun (\Sigma)$ where it is smooth, this integrable system is governed by a quantum groupoid, one that is twist-equivalent to the quantum group $Y_\hbar (\fd)$. In fact, over $B$, we can understand this dynamical twist explicitly, and extract solutions of dynamical quantum Yang-Baxter equations. 

It will be an interesting question to understand this integrable system beyond the perturbative sector, by studying the geometry of $\Bun (\Sigma)$ and $G(\CO)\!\setminus\!\Gr_G$. We leave this for future research. 

\subsection{Structure of the paper}

The paper is structured as follows.

\begin{itemize}

    \item In Section \ref{sec:generalities}, we recall the definition and properties of Hopf algebroids following \cite{xu2001quantum}, especially the definition of quantum twists in this setting. We also introduce Lie bialgebroids and their quantizations following \cite{liu_weinstein_xu_manin_triples}. This section serves as a tool-kit for the rest of the paper. 
    
    \item In Section \ref{sec:preliminaries}, we recall the definition of $Y_\hbar (\fd)$ and its properties that are relevant to our construction. We also recall the moduli space $\Bun (\Sigma)$, its geometry and its formal completion over regularly stable bundles. Finally, we explain, following the methods of \cite{abedin2024r}, that our setup $(\Sigma, \CP, \sigma)$ leads to Lie bialgerboids and dynamical $r$-matrices over the formal neighborhood of $\Bun (\Sigma)$ around $\CP$.  

    \item In Section \ref{sec:constructHopf}, we construct the quantization $\Upsilon_\hbar^\sigma (\fd)$ of the Lie bialgebroid from the previous section, generalizing the ideas of \cite{abedin2024yangian} to splittings of Lie algebroids. 

    \item In Section \ref{sec:dynamicalR}, we show that $\Upsilon_\hbar^\sigma (\fd)$ can be obtained from the Yangian $Y_\hbar (\fd)$ via a dynamical twist. Decomposing this twist, we construct solutions of quantum dynamical Yang-Baxter equation. 
    
\end{itemize}

\subsection{Acknowledgements}

R.A. would like to thank Meer Ashwinkumar, Giovanni Felder, and Silvan Lacroix for useful discussions related to this work. The work of R.A. was supported by the DFG grant AB 940/1--1 and as part of the NCCR SwissMAP, a National
Centre of Competence in Research, funded by the Swiss
National Science Foundation (grant number 205607).

W.N. would like to thank Kevin Costello, Tudor Dimofte, Davide Gaiotto, Faroogh Moosavian, Ben Webster and Harold Williams for discussions related to this topic. W.N. would also like to thank his friend Don Manuel for many encouragements. W.N.'s research is supported by Perimeter Institute for Theoretical Physics. Research at Perimeter Institute is supported in part by the Government of Canada through the Department of Innovation, Science and Economic Development Canada and by the Province of Ontario through the Ministry of Colleges and Universities.

\section{Generalities about Hopf algebroids}\label{sec:generalities}

In this section, we collect some preliminaries on Hopf algebroids, Lie bialgebroids, and their quantizations, with the hope of making this paper more self-contained. We in particular comment on how the quantization strategy we presented in \cite{abedin2024yangian} has a natural generalization to the case of Lie algebroids. As we have mentioned, this section serves as a tool-kit for the constructions to occur in later sections. 

Those who are familiar with Hopf algebroids and their basic properties may skip the whole section except sections \ref{subsubsec:formalgroupoid} and \ref{subsec:cotangbialg}. There we present, to our knowledge new, results on the construction of the algebra of ``functions on the formal groupoid" associated to a Lie algebroid and the construction of quantum groupoids from Lie algebroid splittings. These are important ingredients for our construction which seems to not have been done in the literature. 

We must comment that in the main application, all the Hopf algebroid structure considered are continuous with respect to the limit topologies coming from the loop algebra $\CK\coloneqq\C(\!(t)\!)$ as well as the formal neighborhood \(B = \textnormal{Spf}(\C[\![\lambda]\!])\) of the moduli space of \(G\)-bundles. Therefore, technically speaking, one must extend all the statements we present here to the topological setting. Such an extension is rather simple, and for cleanliness of statements we will omit the mentioning of topology in this section (as in Remark \ref{Rem:topology}).

\subsection{Definition and properties}

Let $A$ be an algebra. All algebras considered here are unital. We will follow the definitions given in \cite{xu2001quantum}. 

\begin{Def}\label{def:hopf_algebroid}
A Hopf algebroid $(H, s, t, m, \Delta, \epsilon)$ over $A$ consists of the following data.

\begin{enumerate}
    \item A total algebra $H$ with multiplication $m$, an algebra homomorphism $t\colon A\to H$ called \textbf{target map}, an algebra antihomomorphism $s\colon A^{\textnormal{op}}\to H$ called \textbf{source map}, such that $s(a)t(b)=t(b)s(a)$ for all $a, b\in A$. This makes $H$ into an $A$-$A$-bimodule via \(h \cdot a = s(a)h\) and \(a\cdot h = t(a)h\) for \(a \in A,h \in H\). Moreover, we denote by $H\sttensor_A H$ the tensor product which uses the right-module structure induced by \(s\) and left-module structure by \(t\), i.e.\ the tensor product over \(\C\) with the additional relation \(s(a)h_1 \otimes h_2 = h_1 \otimes t(a) h_2 \in H\sttensor_A H\) for \(h_1,h_2, a\in A\).

    \item A coproduct $\Delta\colon H\to H\sttensor_A H$, which is a map of $A$-$A$-bimodules, making $H$ a coalgebra object in the category of $A$-$A$-bimodules. In particular, it satisfies the coassociativity\be
        (1\otimes \Delta)\Delta=(\Delta \otimes 1)\Delta: H\to H\sttensor_A H\sttensor_A H.
    \ee
    and is compatible with the product in the sense that
    \be\label{eq:coprodAA}
        \Delta (h)(s(a)\otimes 1-1\otimes t(a))=0, \qquad \Delta (h_1h_2)=\Delta (h_1)\Delta (h_2).
    \ee
    holds for all \(h_1,h_2 \in H\) and \(a \in A\).
    Here, we use the right action of $H\otimes H$ on $H\sttensor_A H$ to define the multiplications above. The first equality guarantees that the second equality makes sense. Moreover, $\Delta (1)=1\otimes 1$. 

    \item A counit $\epsilon\colon H\to A$, which is an $A$-$A$-bimodule morphism such that $\epsilon (1_H)=1_A$ (implying that $\epsilon s=\epsilon t=\mathrm{Id}_A$) and
    \(
(\epsilon\otimes 1) \Delta=(1\otimes \epsilon )\Delta=\mathrm{Id}_H 
    \)
    holds.\hfill $\lozenge$
\end{enumerate}

\end{Def}

\begin{Rem}

    We warn the readers that in this paper, we used a different convention on source and target maps compared to the one used in \cite{xu2001quantum}. \hfill $\lozenge$
\end{Rem}

We now define what an anchor map is. Consider the algebra $\End (A)$, namely \(\C\)-linear endomorphisms of the algebra $A$. This is an $A$-$A$-bimodule, where $A$ acts on the left/right by left/right multiplication on the target. Namely,
\be
af a' (-)=af(-)a'\,,\qquad \forall a, a'\in A, f\in \End (A).
\ee
Suppose that we are given an $A$-$A$-bimodule map $\psi: H\to \End (A)$, which particularly makes $A$ into a module of $H$. We define two maps:
\be\label{eq:phi_st}
\varphi_t, \varphi_s: H\sttensor_A H \otimes A\to H,
\ee
by
\be\label{eq:varphits}
\varphi_t (h_1\otimes h_2 \otimes a)= t(\psi (h_2)(a))h_1, \qquad \varphi_s (h_1\otimes h_2 \otimes a)=s (\psi (h_1)(a))h_2. 
\ee
These are well-defined thanks to the fact that $\psi$ is an $A$-$A$-bimodule map. 

\begin{Def}
    An anchor map is an $A$-$A$-bimodule homomorphism $\psi: H\to \End (A)$ such that
    \begin{enumerate}
        \item $\varphi_t (\Delta h\otimes a)=ht(a)$, $\varphi_s (\Delta h\otimes a)=hs(a)$. 

        \item $\psi (h) (1_A)=\epsilon (h)$. \hfill $\lozenge$
        
    \end{enumerate}
    
\end{Def}

Note that in \cite[Proposition 3.3]{xu2001quantum}, it was shown that if the first item is true, then $\psi (h) (1_A)$ defines a counit for $H$. Therefore the second item just means that the counit comes exactly from acting on the identity via the anchor map. Given a tensor 
\begin{equation}
    h_1\otimes h_2\otimes \cdots \otimes h_n\in \underbrace{H \,\sttensor_A \dots \sttensor_A H}_{n\textnormal{-fold}}   
\end{equation}
and $a_{i_1},\ldots, a_{i_k}\in A$, the element
\be
h_1\otimes_A \cdots \otimes_A \psi (h_{i_j})(a_{i_j})\otimes_A \cdots \otimes_A h_n \in \underbrace{H \,\sttensor_A \dots \sttensor_A H}_{(n-k)\textnormal{-fold}} 
\ee
is a well-defined element, thanks to the fact that $\psi$ is a bimodule map. The following properties are consequences of the existence of an anchor map. 

\begin{Cor}
    Let \(a,b \in A\) and \(h \in H\). Then the following results are true:    
    \begin{enumerate}
        \item $\psi (t (a))(b)=ab$, $\psi (s(a))(b)=ba$. 

        \item $\psi (\Delta (h)) (a\otimes b)=\psi (h)(ab)$. 

        \item $\epsilon (hg)=\psi (h)(\epsilon (g))$. \hfill \qedsymbol
    \end{enumerate}
\end{Cor}

The definitions above ensure that the category of \(H\)-modules is monoidal; see \cite[Theorem 3.6]{xu2001quantum}.

\begin{Thm}
  If $H$ is a Hopf algebroid over $A$ equipped with an anchor map $\psi$, the category of $H$-modules is monoidal with unit $(A, \psi)$. Given two modules of $H$, say $M, N$, their monoidal product is given by $M\sttensor_A N$, where the $H$-module structure is defined by
    \be
    h\cdot (m\otimes n)=\Delta (h)(m\otimes n)\,,\qquad h \in H, m \in M, n \in N,
    \ee
where the tensor product on the RHS is over \(\C\). This is well-defined thanks to equation \eqref{eq:coprodAA}. The forgetful functor from the category of $H$-modules to the category of $A$-$A$-bimodules is an exact monoidal functor. \hfill \qedsymbol

\end{Thm}

Let us give one example of a Hopf algebroid. More examples will be given later. Let $A=R$ be a smooth ring (or its formal completion at a smooth point), let $D_R$ be the algebra of differential operators of $R$. Then $D_R$ has the structure of a Hopf algebroid. Here $s=t$ is the embedding $R\to D_R$, and the coproduct is defined by
\be
\Delta (D) (r_1, r_2)=D(r_1r_2)\,,\qquad r_1,r_2 \in R.
\ee
The anchor map is the obvious one, making $R$ into a module of $D_R$, and the counit is $\epsilon(D)=D(1)$. 

Underlying $D_R$ is a Lie algebroid over $R$, called the \textbf{tangent Lie algebroid} $T_R$. These are the linear differential operators, and one can show that $D_R$ is the universal enveloping algebra of $D_R$ over $R$. We will review the definition of this in Section \ref{subsec:Liealgebroid}.

\subsection{Twisting Hopf algebroids}\label{subsec:TwistHopf}

Just as in the case of Hopf algebras, it is important to consider twists of Hopf algebroids. In this section, we recall their definition, again following \cite{xu2001quantum}. 

Let $H$ be a Hopf algebroid over $A$ with anchor map $\psi$. Let $\CF$ be an element in $H\,\sttensor_A H$. Using the associated maps \(\varphi_s\) and \(\varphi_t\) from \eqref{eq:phi_st}, we define two maps $t_\CF, s_{\CF}\colon A\to H$ by
\be
t_\CF(a)=\varphi_t (\CF\otimes a), \qquad s_\CF (a)=\varphi_s (\CF\otimes a ). 
\ee
Moreover, define $a\cdot_\CF b$ for \(a,b\in A\) by
\be
a\cdot_\CF b=\psi(t_\CF(a))(b). 
\ee

\begin{Prop}\label{Prop:twistequations}
    Assume that $\CF$ satisfies the following conditions. 
    \begin{enumerate}
        \item Cocycle condition:
        \be
(\Delta\otimes 1)(\CF)\CF^{(12)}=(1\otimes \Delta) (\CF)\CF^{(23)}\in H\,\sttensor_A H\,\sttensor_A H.
        \ee

        \item Compatibility with counit:
        \be
(\epsilon\otimes 1)\CF=(1\otimes \epsilon) \CF=1_H.
        \ee
    \end{enumerate}
Then the following results hold true:
\begin{enumerate}
    \item $(A, \cdot_\CF)$ is an associative algebra, which we denote by $A_\CF$. 

    \item $t_\CF \colon A_\CF \to H$ and $s_\CF \colon A_\CF^{\textnormal{op}} \to H$ define new commuting algebra homomorphisms. 

    \item $\CF$ satisfies \(
    \CF (s_\CF (a)\otimes 1-1\otimes t_\CF (a))=0\) for all \(a \in A\). 
   Consequently, $\CF$ defines a map $\CF\colon H\,\sttensor_{A_\CF} H\longrightarrow H\sttensor_A H$ by
    \be\label{eq:CFmap}
\CF (h_1\otimes h_2)\coloneqq\CF\cdot (h_1\otimes h_2)\,,\qquad \forall h_1,h_2 \in H,
    \ee
where on the RHS the tensor product is over \(\C\). \hfill \qedsymbol
\end{enumerate}

\end{Prop}

The map defined in equation \eqref{eq:CFmap} above intertwines the actions of $A_\CF$ on both sides, where the action of $A_\CF$ on the right is induced by restricting the action of $H$ on $H\,\sttensor_A H$ to $A_\CF$. Let us now assume that $\CF$ is invertible as a map, whose inverse is given by $\CF^{-1}$. Define $\Delta_\CF\colon H\to H\,\sttensor_{A_\CF} H$ by
\be
\Delta_\CF(h)\coloneqq\CF^{-1}(\Delta (h)\CF)\in H\,\sttensor_{A_\CF} H
\ee
for all \(h \in H\).
It turns out that $\Delta_\CF$ is a coproduct for $H$ as a bimodule of $A_\CF$, and satisfies the requirements of a Hopf algebroid. Such an element $\CF$ is called a \textbf{twist}. This notion of twist is the generalization of the twist construction for Hopf algebras. From such a twist element one can construct a new Hopf algebroid structure on $H$ from the original one, and moreover these two Hopf algebroids gives equivalent monoidal categories. We summarize this into the following theorem (see \cite[Theorem 4.14]{xu2001quantum}).

\begin{Thm}
    Assume that $(H, s,t, m, \Delta, \epsilon)$ is a Hopf algebroid over $A$ with anchor map $\psi$, and let $\CF$ be a twist, i.e.\ it satisfies the conditions of Proposition \ref{Prop:twistequations} and is invertible. Then $(H, s_\CF, t_\CF, m, \Delta_\CF, \epsilon)$ is a Hopf algebroid over $A_\CF$ with anchor $\psi$. Moreover, there is an equivalence of monoidal categories:
    \be
(H, \Delta)\Mod\simeq (H,\Delta_\CF)\Mod
    \ee
    induced by $\CF$, which intertwines the underlying bimodule structures of $A_\CF$. We say that $A$ and $A_\CF$ are twist-equivalent. \hfill \qedsymbol
    
\end{Thm}

In the proof of Section \ref{sec:constructHopf}, we will sometimes use a twist that is actually an element in $H_1\sttensor_A H_2$ for two different Hopf algebroids with some common subalgebra. In the end, as long as the twist and its coproduct makes sense in the common subalgebra, the statements above remain true. 

\subsection{Lie algebroid and formal groupoid}\label{subsec:Liealgebroid}

We review something about Lie algebroid and the structures of its universal enveloping algebra, following \cite{moerdijk2010universal}. We will also define the formal ring of functions on the formal groupoid associated to such a Lie algebroid. In this section, we will always assume that we have a commutative ring $R$. 

\subsubsection{Lie algebroids and their universal envelopes}\label{subsubsec:LieU}

\begin{Def}
    A Lie algebroid $\mf G$ is a left $R$-module, together with the structure of a Lie algebra with Lie bracket $[,]$ over \(\C\), and a homomorphism of Lie algebras $\psi\colon \mf G\to \mathrm{Der} (R)$ into the Lie algebra of derivations of $R$, which is also a homomorphism of left $R$-modules, such that
    \be
[x, ry]=r[x, y]+\psi (x)(r) y
    \ee
    for all \(x,y\in \mf G\) and \(r \in R\). \hfill $\lozenge$
\end{Def}

\begin{Rem}
    Lie algberoids as defined above are actually global sections of Lie algebroids in the differential geometric language and are also sometimes called Lie-Rinehart algebras.\hfill $\lozenge$
\end{Rem}

The set of all Lie algebroids over $R$ forms a category, and this category has the final object $T_R = \mathrm{Der}(R)$. In the case when $R$ is a smooth ring of finite Krull dimension, the module $\mathrm{Der}(R)$ is a finitely-generated free module over $R$, whose rank is equal to the dimension of $R$. This module is the tangent Lie algebroid $T_R$ of $R$. 

Just as the case of a Lie algebra, one can associate an associative object to a Lie algebroid. This object turns out to be a Hopf algebroid over $R$. We start with a formal definition. 

\begin{Def}
    The universal enveloping algebra of a Lie algebroid $\mf G$ over $R$, denoted by $U_R (\mf G)$, is the unique algebra characterized by the following universal property:
    \begin{enumerate}
        \item There is an algebra homomorphism $s\colon R\to U_R (\mf G)$ and a Lie algebra homomorphism $\rho\colon \mf G\to U_R (\mf G)$ such that
        \be\label{eq:XrU}
            \rho (rx)=s(r)\rho (x),\qquad [\rho (x), s(r)]=s(\psi(x)(r)), \qquad \forall x\in \mf G, r\in R; 
        \ee
        \item For any other algebra $A$ equipped with an algebra homomorphism $s_A \colon R\to A$ and a Lie algebra homomorphism $\rho_A \colon \mf G \to A$ satisfying equation \eqref{eq:XrU}, there is a unique algebra homomorphism $f\colon U_R (\mf G)\to A$ such that $f\circ \rho=\rho_A$ and $f\circ s=s_A$.\hfill $\lozenge$
    \end{enumerate}  
\end{Def}

Let us now sketch the construction of $U_R (\mf G)$. The space $R\oplus \mf G$ has the structure of a Lie algebra with bracket
\be
[(r, x), (s, y)]=(\psi (x)(s)-\psi (y)(r), [x, y]).
\ee
Let $U(R\oplus \mf G)$ be its universal enveloping algebra (over $\C$). There is an embedding of Lie algebras $i\colon R\oplus \mf G\to U (R\oplus \mf G)$, and let $\overline{U} (R\oplus \mf G)$ be the subalgebra generated by $i (R\oplus \mf G)$. We can define $U_R (\mf G)$ to be the quotient of $\overline{U}(R\oplus \mf G)$ by the two-sided ideal generated by elements of the form
\be
i(s, 0) i(r, x)-i (sr, sx), \qquad \forall s, r\in R, x\in \mf G.
\ee
Both the universal definition and the actual construction have advantages in different situations. For example, the universal property implies that $U_R (\mf G)$ comes equipped with an algebra homomorphism (also a morphism of $R$-modules) $\psi\colon U_R (\mf G)\to \End (R)$. It is also most convenient to use the universal property to construct maps between universal enveloping algebras of Lie algebroids. On the other hand, the explicit construction reveals many properties of $U_R (\mf G)$. For example, we have the following result. 

\begin{Lem}[PBW theorem]\label{Lem:PBW}
    Let $\mf G$ be a Lie algebroid that is a free $R$-module, then $U_R (\mf G)$ has a natural filtration with an identification of $R$-algebras:
    \be
\mathrm{Gr} \lp U_R(\mf G)\rp \cong S_R (\mf G).
    \ee
    Here, the right hand side is the symmetric algebra over $R$ generated by $\mf G$.\hfill \qedsymbol 
\end{Lem}

This is convenient in proving isomorphisms between many algebras. For instance, one can show that $U_R (T_R)\cong D_R$, where $D_R$ is the algebra of differential operators on $R$, when $R$ is a smooth ring. From now on, we always assume that $\mf G$ is free over $R$. 

It turns out that $U_R (\mf G)$ has the structure of a Hopf algebroid over $R$. Let $U_R(\mf G)\otimes_R U_R(\mf G)$ be the tensor product, where we use the action of $R$ on $U_R (\mf G)$ by left multiplication. Define the subset $U_R(\mf G)\overline{\otimes}_R U_R(\mf G)$ to be the kernel of the map 
\begin{equation}
    \theta \colon U_R(\mf G)\otimes_R U_R(\mf G)\longrightarrow \Hom (R, U_R(\mf G)\otimes_R U_R(\mf G))    
\end{equation}
given by
\be\label{eq:thetaabr}
\theta (x\otimes_R y) (r)=xr\otimes y-x\otimes yr,\qquad x,y\in U_R (\mf G), r\in R.
\ee
This is well-defined since the tensor uses the left multiplication of $R$ on $U_R (\mf G)$. It turns out that $U_R(\mf G)\overline{\otimes}_R U_R(\mf G))$ is an algebra, whose multiplication is the evident multiplication on the two tensor copies. Define $\Delta\colon U_R (\mf G)\to U_R(\mf G)\overline{\otimes}_R U_R(\mf G)$ by
\be
\Delta (x)=x\otimes 1+1\otimes x,\qquad \Delta (r)=r\otimes 1=1\otimes r,\qquad \forall x\in \mf G, r\in R. 
\ee
Using the universal property, one can show that this gives rise to an algebra homomorphism (in fact an embedding, thanks to PBW theorem) from $U_R (\mf G)$ to $U_R(\mf G)\overline{\otimes}_R U_R(\mf G)$. Note that the condition satisfied in equation \eqref{eq:thetaabr} by $\Delta$ is precisely the condition in equation \eqref{eq:coprodAA} in the definition of a Hopf algebroid. Moreover, one can define $\epsilon\colon U_R (\mf G)\to R$ by
\be
\epsilon (a)=\psi (a)(1).
\ee
Finally, let us define source and target map to be both equal to the algebra embedding of $s\colon R\to U_R (\mf G)$. 

\begin{Prop}
    $(U_R(\mf G),s, t = s, m, \Delta, \epsilon)$ defines a cocommutative Hopf algebroid over $R$ with anchor map $\psi$. Moreover, $\mf G$ is precisely the subset of primitive elements. \hfill \qedsymbol
    
\end{Prop}

If we extend the universal envelope by a formal parameter, i.e.\ if we consider $U_R(\mf G)\lbb\hbar\rbb$, this ring admits group-like elements. More precisely, we have the following useful lemma, whose proof goes along the same lines as \cite[Lemma 3.3.1.]{abedin2024yangian}.

\begin{Lem}\label{Lem:groupoid_like_elements}
        We have
        \(\Delta(a) = a \otimes a \in U(\mf G)[\![\hbar]\!] \,\sttensor_{R[\![\hbar]\!]} U(\mf G)[\![\hbar]\!]\) if and only if \(a = e^x\) for some \(x \in \hbar\mf G[\![\hbar]\!]\).\hfill \qedsymbol
\end{Lem}

Splittings of Lie algebroids will be very important in this work. We conclude this section by examining the structures induced on the level of universal enveloping algebras from such splittings. Let $\mf G$ be a Lie algebroid, and let $\mf H$ and $\mf K$ be Lie subalgebroids, such that $\mf G=\mf H\oplus \mf K$ as $R$-module. From the universal property, we have obvious algebra embeddings $U_R (\mf K), U_R (\mf H)\to U_R (\mf G)$. Moreover, multiplication gives a map
\be
U_R (\mf H)\otimes U_R (\mf K)\to U_R (\mf G),
\ee
which factors through the tensor product over $R$:
    \be\label{eq:multirl}
m\colon U_R (\mf H)\rltensor_R U_R (\mf K)\longrightarrow  U_R(\mf G).
    \ee
Here, we use $\rltensor_R$ to mean the tensor where we use right multiplication on the first factor and left multiplication on the second (note that this is different from $\sttensor$, since in this case $s=t$). By PBW theorem \ref{Lem:PBW}, this is an isomorphism of $R$-modules. Moreover, it is easy to see that this is an isomorphism of:
\begin{enumerate}
    \item left $U_R (\mf H)$-modules;

    \item right $U_R (\mf K)$-modules;

    \item cocommutative \(R\)-coalgebras. 
    
\end{enumerate}

Given such a decomposition, we obtain a left action of $U_R (\mf K)$ on $U_R (\mf H)$, denoted by $\rhd$, via
\be
U_R (\mf G)\otimes_{U_R (\mf K)} R\cong U_R (\mf H),
\ee
and similarly a right action of $U_R (\mf H)$ on $U_R (\mf G)$, denoted by $\lhd$, via
\be
R\otimes_{U_R (\mf H)} U_R (\mf G)\cong U_R (\mf K).
\ee
The action maps are all coalgebra homomorphisms. The following is a simple generalization of a similar statement for Lie algebras from \cite{majid1990physics}, and we omit the proof here. 

\begin{Prop}\label{prop:matchedULie}
For \(h,g \in U_R(\mf K)\) and \(a,b\in U_R(\mf H)\), the following results are true.
    \be\label{eq:crossedUg}
\begin{aligned}
    &(hg)\lhd a=(h\lhd (g_{(1)}\rhd a_{(1)}))\cdot (g_{(2)}\lhd a_{(2)}), 1\lhd a=\epsilon(a)\\
    & h\rhd (ab)=(h_{(1)}\rhd a_{(1)})\cdot ((h_{(2)}\lhd a_{(2)})\rhd b), h\rhd 1=\epsilon (h)\\
    & h_{(1)}\lhd a_{(1)}\otimes h_{(2)}\rhd a_{(2)}=h_{(2)}\lhd a_{(2)}\otimes h_{(1)}\rhd a_{(1)}.
\end{aligned}
\ee
Moreover, the multiplication on the algebra $U_R (\mf G)$ can be defined by
\be\label{eq:bicrossproduct}
(a\otimes h)(b\otimes g)=a(h^{(1)}\rhd b^{(1)})\otimes (h^{(2)}\lhd b^{(2)})g.
\ee
\hfill\qedsymbol
\end{Prop}

\subsubsection{Functions on formal groupoid}\label{subsubsec:formalgroupoid}

Let $\mf G$ be a Lie algebroid over $R$, which is free as \(R\)-module, and let $U_R(\mf G)$ be its universal enveloping algebra. Since it is an algebra over $R$, it is automatically an $R$-$R$-bimodule using left and right multiplication. Note that this $R$-$R$-bimodule structure is different from the $R$-$R$-bimodule structure of $U_R(\mf G)$ as a Hopf algebroid, since in this case source and target maps are equal. 

Denote by $\mf G^\dagger$ the $R$-linear dual of $\mf G$. Using the PBW theorem \ref{Lem:PBW}, we can identify $U_R(\mf G)$ as a left $R$-module (in fact, as a cocommutative coalgebra over $R$), with $S_R(\mf G)$, and therefore:
\be
\Hom_R (U_R(\mf G), R)=\Hom_R (S_R(\mf G), R)=\hat S_R (\mf G^\dagger).
\ee
Here, $\hat S_R (\mf G^\dagger)$ is the symmetric algebra over $R$ generated by $\mf G^\dagger = \Hom_R(\mf G,R)$ and completed with respect to the symmetric degrees. Since $U_R(\mf G)$ is also a right $R$-module, we obtain an $R$-$R$-bimodule structure on $\hat S_R (\mf G^\dagger)$ by
\be
\lag rfr', a\rag=\lag f, ar\rag r'\,,\qquad r,r' \in R,f \in \hat S_R (\mf G^\dagger). 
\ee

\begin{Lem}\label{lem:source_target_formal_groupoid}
    The two actions of $R$ on $\hat S_R (\mf G^\dagger)$ preserve the commutative algebra structure on $\hat S_R (\mf G^\dagger)$. Namely, the actions give rise to two ring homomorphism from $R$ to $\hat S_R (\mf G^\dagger)$. 
    
\end{Lem}

\begin{proof}
    It is clear for the right action. Let $f, g \in \hat S_R (\mf G^\dagger)$ and $r \in R$. For any $w\in U_R(\mf G)$, we have:
    \be
\lag r(fg), w\rag=\lag fg, wr\rag=\lag f\otimes g, \Delta (wr)\rag=\lag f\otimes g, \Delta (w)\Delta (r)\rag.
    \ee
Here, $\Delta (w)\Delta (r)$ is computed in $U_R(\mf G)\overline{\otimes}_R U_R(\mf G)$, so, by definition of \(\overline{\otimes}\), we have:
\be
\Delta (w)\Delta (r)=\sum w^{(1)}r\otimes w^{(2)}=\sum w^{(1)}\otimes w^{(2)}r.
\ee
Applying $\lag f\otimes g,-\rag$ to this we get the desired result: \(r(fg) = (rf)g = f(rg)\). 
\end{proof}

    One can think of the commutative algebra $\hat S_R (\mf G^\dagger)$ as the ring of functions on the formal groupoid that is given by exponentiation of the Lie algebroid $\mf G$. In this picture, the two ring homomorphisms (embeddings) from Lemma \ref{lem:source_target_formal_groupoid} are the source and the target maps associated to this groupoid. 
    More precisely, 
    \begin{equation}\label{eq:source_target_formal_groupoid}
        s(r)(x) = \langle 1,x\rangle r = \epsilon(x)r = r\epsilon(x) \textnormal{ and } t(r)(x) = \langle 1,xr\rangle \,,\qquad r \in R,x \in U_R(\fG) 
    \end{equation}
    defines source and target maps $s\colon R\to \hat S_R (\mf G^\dagger)$ and $t\colon R\to \hat S_R (\mf G^\dagger)$ such that the aforementioned $R$-$R$-bimodule structure on $\hat S_R (\mf G^\dagger)$ coincides with the \(s\)-\(t\)-module structure in the Hopf algebroid setting. Observe that the target map \(t\) does not look very explicit yet, however we will fix this in the \(\hbar\)-adic language below.
    
    Let us now define a comultiplication and counit that completes the triple \((\hat{S}(\fG^\dagger),s,t)\) to a Hopf algebroid $\hat S_R (\mf G^\dagger)$ with anchor map.

    \begin{Prop}
        There exists algebra homomorphisms
        \be
\Delta\colon \hat S_R (\mf G^\dagger)\longrightarrow \hat S_R (\mf G^\dagger)\sttensor_R \hat S_R (\mf G^\dagger),
        \ee
        where the tensor product on the RHS is the completed tensor product, and $\epsilon\colon \hat S_R (\mf G^\dagger)\to R$ making $\hat S_R (\mf G^\dagger)$ into a Hopf algebroid over $R$. This algebra has an anchor map given by left multiplication with $\epsilon$. 

    \end{Prop}

\begin{proof}
    Multiplication defines a map of $R$-$R$-bimodules:
    \be
        m\colon U_R(\mf G)\rltensor_R U_R(\mf G)\longrightarrow  U_R(\mf G),
    \ee
where $\rltensor_R$ has the same meaning as equation \eqref{eq:multirl}. By dualizing, we get
\be
\hat S_R (\mf G^\dagger)\longrightarrow \Hom_R (U_R(\mf G)\rltensor_R U_R(\mf G), R). 
\ee
Now the tensor-hom adjunction provides
\be
\Hom_R  (U_R(\mf G)\rltensor_R U_R(\mf G), R)= \Hom_R(U_R(\mf G), \Hom_R(U_R(\mf G), R))=\hat S_R(\mf G^\dagger)\sttensor_R \hat S_R(\mf G^\dagger).
\ee
The dual of $m$ therefore defines a morphism of bimodules:
\be
\Delta\colon \hat S_R (\mf G^\dagger)\to \hat S_R (\mf G^\dagger)\sttensor_R \hat S_R (\mf G^\dagger). 
\ee
The coassociativity of $\Delta$ clearly follows from associativity of $m$, so we only need to show that $\Delta$ is a morphism of commutative algebras. This follows from the fact that $m$ is a morphism of cocommutative coalgebras. 

The map $\epsilon$ is defined to be the dual of the unit map $R\to U_R(\mf G)$. It can be thought of as mapping $f\mapsto f(1)$, evaluating the functions $f$ at identity. If we think of $\hat S_R(\mf G)$ as functions on the formal groupoid associated to $\mf G$, then this is precisely the unit map of this groupoid. The statement about anchor map is clear. 

\end{proof}

\begin{Rem}\label{rem:evalutating_sttenosors}
    Note that the identification:
    \be
\Hom_R  (U_R(\mf G)\rltensor_R U_R(\mf G), R)\cong \hat S_R (\mf G^\dagger)\sttensor_R \hat S_R (\mf G^\dagger)
    \ee
    can be explicitly given by the following formula:
    \be\label{eq:stpair}
        \lag f\sttensor_R g, u_1\rltensor u_2\rag=g(u_1 f(u_2)) =(f(u_2)g)(u_1) = g(u_1)f(u_2) + \epsilon(g)\psi(u_2)f(u_1),
    \ee
for \(f, g\in \hat S_R (\mf G^\dagger), u_1, u_2\in U_R(\mf G)\).\hfill $\lozenge$

\end{Rem}

\begin{Rem}
    We further comment that in later sections, we use different versions of tensor products, and therefore different ways of pairing them. For example, given two free left $R$ modules $M, N$, we can identify 
    \be
M\otimes N^\dagger\cong \Hom_{R}(N, M).
    \ee
  Under this isomorphism, the pairing between $m\otimes_R f$ with $n\in N$ is given by
    \be
m\otimes_R f(n)=f(n)m,
    \ee
    which is slightly different from the above pairing in equation \eqref{eq:stpair}. In later applications, we will omit all the decorations in the tensor to avoid clustering of notations. To help reduce confusion, we will comment on which version of tensor and dual pairing is being used in each specific situations. \hfill $\lozenge$

\end{Rem}

The above statements can be translated into the $\hbar$-adic language, as in \cite{abedin2024yangian}. In particular, we identify the linear dual of $U_R(\mf G)\lbb\hbar\rbb$ with $S_R (\hbar \mf G^\dagger)\lbb\hbar\rbb$, and obtain a coproduct:
\be
S_R (\hbar \mf G^\dagger)\lbb\hbar\rbb\to S_R (\hbar \mf G^\dagger)\lbb\hbar\rbb\,\sttensor_{R\lbb\hbar\rbb} S_R (\hbar \mf G^\dagger)\lbb\hbar\rbb.
\ee
Dividing the formula on $S_R (\hbar \mf G^\dagger)\lbb\hbar\rbb$ by $\hbar$, we get a coproduct \(\Delta_\hbar\) on $S_R (\mf G^\dagger)\lbb\hbar\rbb$. We obtain a Hopf algebroid over \(R[\![\hbar]\!]\) in this way. This convenience of this language becomes clear in the following statement.

\begin{Prop}\label{prop:coproduct_symmetric_algebra_explicit}
    The following results are true:
    \begin{enumerate}
        \item Let $f\in S_R (\mf G^\dagger)\lbb\hbar\rbb$. If for any $x\in \mf G$, $\lag e^{\hbar x}, f\rag=0$, then $f=0$. Here, $e^{\hbar x}$ is computed in the algebra $U_R(\mf G)\lbb\hbar\rbb$. 

        \item The source and target maps \(s\colon R[\![\hbar]\!] \to S_R (\mf G^\dagger)\lbb\hbar\rbb\) and \(t\colon R[\![\hbar]\!] \to  S_R (\mf G^\dagger)\lbb\hbar\rbb\) are uniquely determined by
        \begin{equation}
        s(r)(e^{\hbar x}) = r \textnormal{ and } t(r)(e^{\hbar x}) = e^{\hbar \psi(x)}(r) \,,\qquad r \in R,x \in \fG. 
        \end{equation}

        \item Let us choose an $R$ basis for $\mf G$, say $x_i$, with brackets $[x_i, x_j]=f_{ij}^kx_k$ for $f_{ij}^k\in R$. Let $t^i$ be the dual basis over $R$, then
        \be
            \Delta_\hbar (t^i)=t^i\otimes 1+1\otimes t^i-\hbar\left( \psi(-)t^i+\frac{1}{2}\sum_{k,j}f_{kj}^it^k\otimes t^j\right) +\CO (\hbar^2),
        \ee
        where \(\psi(-)t^i \in \fG^\dagger \otimes \fG^\dagger\) is given by \(x \otimes y \mapsto \psi(y) t^i(x)\).
    \end{enumerate}
\end{Prop}

\begin{proof}
    The statement in 1.\ simply follows from PBW theorem \ref{Lem:PBW} in the same vein as \cite[Lemma 3.4.2.]{abedin2024yangian}. Now 2.\ follows from 1., \eqref{eq:source_target_formal_groupoid}, and
    \begin{equation}
        \langle t(r),e^{\hbar x}\rangle = \langle 1,e^{\hbar x}r\rangle = \langle 1, (e^{\hbar x}re^{-\hbar x})e^{\hbar x}\rangle = e^{\hbar \psi(x)}(r) \langle 1,e^{\hbar x}\rangle = e^{\hbar \psi(x)}(r)
    \end{equation}
    where we used the left linearity and the fact that \(e^{\hbar x}re^{-\hbar x} = e^{\hbar \textnormal{ad}(x)}r = e^{\hbar \psi(x)}(r)\) holds in \(U_R(\fG)[\![\hbar]\!]\).
    
    By definition, we have
    \be
    \begin{split}
        \lag \Delta_\hbar (t^i), e^{\hbar x_j}\otimes e^{\hbar x_k}\rag&=\lag t^i, e^{\hbar x_j}e^{\hbar x_k}\rag=\lag t^i, 1+\hbar (x_j+x_k)+\hbar^2/2[x_j,x_k]\rag+\CO(\hbar^2)
        \\
        &=\lag t^i,x_j+x_k\rag+\hbar/2 f^i_{jk} + \CO(\hbar^2).
    \end{split}
    \ee
    On the other hand, if we write \(t^{i} = \sum_{(t^{i})}  t^{i,(1)} \otimes t^{i,(2)}\), we have
    \begin{equation}
        \begin{split}
            &\lag        \Delta_\hbar (t^i), e^{\hbar x_j}\otimes e^{\hbar x_k}\rag=\sum_{(t^i)}\lp\lag t^{i,(1)},e^{\hbar x_k}\rag \lag t^{i,(2)},e^{\hbar x_j}\rag + \epsilon(t^{i,(2)})\psi(e^{\hbar x_j})\lag t^{i,(1)},e^{\hbar x_k}\rag \rp\\&=\psi(e^{\hbar x_j})\lag t^{i},e^{\hbar x_k}\rag+\sum_{(t^i)}\lag t^{i,(1)},e^{\hbar x_k}\rag \lag t^{i,(2)},e^{\hbar x_j}\rag
            \\&= \langle t^i,x_j+x_k\rangle + 
            \hbar\left( \psi(x_j)\lag t^{i},x_k\rag+\sum_{(t^i)}\lag t^{i,(1)},x_k\rag \lag t^{i,(2)},x_j\rag\right) + \CO(\hbar^2)
        \end{split}
    \end{equation}
    under consideration of Remark \ref{rem:evalutating_sttenosors}. Therefore, the $\hbar$-order term comes from evaluating on $x_j\otimes x_k$ using
    \be
\sum f_{jk}^i\lag t^k\otimes t^j, x_j\otimes x_k\rag=f_{jk}^i,
    \ee
   we find that  the statement in 2.\ follows. 
\end{proof}

\begin{Rem}
    Let us point out that we will use Proposition \ref{prop:coproduct_symmetric_algebra_explicit}.1.\ frequently in order to prove identities involving elements of \(S(\fG^\dagger)[\![\hbar]\!]\) by evaluating on elements of the form \(e^{\hbar x}\) for \(x \in \fG\) only.
\end{Rem}

\subsection{Lie bialgebroids and their quantizations}\label{subsec:qLiebialgebroid}
The quantization scheme for Lie bialgebras to Hopf algebras can be generalized to the algebroid language. The analog of a Lie bialgebra is thereby a Lie bialgebroid. 

To give the definition, it is useful to adapt some notions of differential calculus to the theory of Lie algebroids. Let \(\mf G\) be a Lie algebroid over \(R\) with bracket \([,]_\fG\) anchor map \(\psi_\fG\). Furthermore, let \(f \in \bigwedge^k \fG^\dagger\), and \(x,x_1,\dots,x_{k+1} \in R\). Then we can define the exterior derivative \(d_{\fG} f \in \bigwedge^{k+1}\fG^\dagger\) by: 
\begin{equation}
    \begin{split}
       d_{\fG}f(x_1,\dots,x_{k+1}) &\coloneqq \sum_{i=1}^{k+1} (-1)^{i+1}\psi_\fG(x_i)f(x_1,\dots,\widehat{x_i},\dots,x_{k+1}) \\&+ \sum_{i < j}(-1)^{i+j}f([x_i,x_j]_\fG,x_1,\dots,\widehat{x_i},\dots,\widehat{x_j},\dots,x_{k+1}).
    \end{split}
\end{equation}
Similarly, we can define the Lie derivative \(L_{\fG,x} f \in \bigwedge^k\fG^\dagger\) and the interior derivative \(\iota_{\fG,x} f\in \bigwedge^{k-1}\fG^\dagger\) by the following formulas: 
\begin{equation}
    \begin{split}
       &\iota_{\fG,x} f(x_1,\dots,x_{k-1}) \coloneqq f(x,x_1,\dots,x_{k-1}),\\
       &L_{\fG,x} f(x_1,\dots,x_k) \coloneqq \psi_\fG(x)f(x_1,\dots,d_k) - \sum_{i = 1}^k f(x_1,\dots,[x,x_i]_\fG,\dots,x_k).
    \end{split}
\end{equation}
As in the usual differential calculus, we have \(L_{\fG,x} = \iota_{\fG,x}d_{\fG} + d_{\fG}\iota_{\fG,x}\). 

\begin{Def}
    A Lie bialgebroid \((\fG,\fG^\dagger)\) consists of Lie algebroid structures on \(\fG\) and \(\fG^\dagger\) compatible in the sense that
    \begin{equation}
        d_{\fG^{\dagger}}[x,y]_{\fG} = [d_{\fG^{\dagger}}x,y]_{\fG} + [x,d_{\fG^{\dagger}}y]_{\fG}
    \end{equation}
    holds for all \(x,y \in \fG\). Here, the embedding \(\fG \subseteq (\fG^\dagger)^\dagger\) was used.\hfill $\lozenge$
    \end{Def}

    Similar to Lie bialgebras, Lie bialgebroids can be defined by Manin triples using the notion of classical double. However, the classical double of a Lie bialgebroid is in general not a Lie algebroid anymore, but a weaker algebraic structure, called Courant algebroid.

    \begin{Def}\label{def:courant_algebroid}
        A Courant algebroid \((\fD,[,],\psi,\langle,\rangle)\) over \(R\) consists of an \(R\)-module \(\fD\) equipped with a skew-symmetric \(\C\)-bilinear operation \([,] \colon \fD \times \fD \to \fD\), a non-degenerate symmetric \(R\)-bilinear form \(\langle,\rangle \colon \fD \times \fD \to R\), and a \(\C\)-linear \(\psi \colon \fD \to \textnormal{Der}(R)\) such that for all \(x_1,x_2,x_3 \in \fD\) and \(r,s \in R\) the following identities hold: 
        \begin{enumerate}
            \item \(\sum_{\sigma \in S_3}[[x_{\sigma(1)},x_{\sigma(2)}],x_{\sigma(3)}] = \frac{1}{6}\sum_{\sigma \in S_3}D\langle [x_{\sigma(1)},x_{\sigma(2)}],x_{\sigma(3)}\rangle\); 

            \item \([x_1,rx_2] = \psi(x_1)(r)x_2 + r[x_1,x_2]-\frac{1}{2}\langle x_1,x_2\rangle Dr\);

            \item \(\psi([x_1,x_2]) = [\psi(x_1),\psi(x_2)]\);

            \item \(\psi(x_1)\langle x_2,x_3\rangle = \langle [x_1,x_2] + \frac{1}{2}D\langle x_1,x_2\rangle ,x_3\rangle + \langle x_2,[x_1,x_3]+\frac{1}{2}D\langle x_1,x_3\rangle \rangle\);

            \item \(\langle Dr,Ds \rangle = 0\).
            \end{enumerate}
            Here, \(D \colon R \to \fD\) is the derivation uniquely determined by
            \begin{equation}
                \langle Dr,x\rangle = \psi(x)r
            \end{equation}
            for all \(x \in \fD\) and \(r \in R\).\hfill $\lozenge$
    \end{Def}

    \begin{Rem}
        By convention, the bilinear form in the defining properties of a Courant algebroid here is twice the bilinear form in the defining properties of a Courant algebroid in \cite{liu_weinstein_xu_manin_triples}.\hfill $\lozenge$
    \end{Rem}
    
    The following statements hold true.

    \begin{Thm}{\cite[Theorem 2.5 \& 2.6]{liu_weinstein_xu_manin_triples}}\label{thm:double_of_lie_bialgebroid}   
        For every Lie bialgebroid \((\fG,\fG^\dagger)\), \(\fD = \fG \oplus \fG^\dagger\) has a unique structure of a Courant algebroid such that the canonical pairing is the structure bilinear form and \(\fL\) and \(\fL^\dagger\) are subalgebroids.

        Conversely, let \(\fD\) be a Courant algebroid and \(\fD = \fL_1 \oplus \fL_2\) be a splitting into Lagrangian Courant subalgebroids such that the bilinear form of \(\fD\) induces an isomorphism \(\fL_2 \cong \fL_1^\dagger\). Then \((\fL_1,\fL_2 \cong \fL_1^\dagger)\) forms a Lie bialgebroid.\hfill \qedsymbol
    \end{Thm}

    Let us outline the construction of the Courant algebroid structure on \(\fD = \fG \oplus \fG^\dagger\) for a Lie bialgebroid \((\fG,\fG^\dagger)\). We have the canonical \(R\)-bilinear form \(\langle,\rangle \colon \fD \times \fD \to R\) defined by
    \begin{equation}
        \langle x_1 + f_1,x_2+f_2 \rangle \coloneqq f_1(x_2) + f_2(x_1)\,,\qquad x_1,x_2 \in \fG,f_1,f_2 \in \fG^\dagger.
    \end{equation} Furthermore, set \(\psi \coloneqq \psi_{\fG} + \psi_{\fG^\dagger}\), and let \([,] \colon \fD \times \fD \to \fD\) be the unique skew-symmetric \(\C\)-bilinear map with the properties:
    \begin{enumerate}
        \item \(\fG,\fG^\dagger \subseteq \fD\) are \(\C\)-subalgebras;

        \item \([x,f] = L_{\fG,x} f - L_{\fG^\dagger,f} x - \frac{1}{2}(d_{\fG} - d_{\fG^\dagger}) f(x)\) for all \(x \in \fG\) and \(f \in \fG^\dagger\).
    \end{enumerate}
    Then \((\fD,[,],\langle,\rangle,\psi)\) is the double Courant algebroid of \((\fG,\fG^\dagger)\).

    \begin{Exp}\label{ex:trivial_double}
    For any Lie algebroid \(\fG\) over \(R\), we can equip \(\fG^\dagger\) with the trivial Lie algebroid structure and then \((\fG,\fG^\dagger)\) is a Lie bialgebroid.

    The Lie bracket of the double Courant algebroid  \(\fD = \fG \oplus \fG^\dagger\) of \((\fG,\fG^\dagger)\) is the unique skew-symmetric \(\C\)-bilinear form defined by:
    \begin{enumerate}
        \item \(\fG \subseteq \fD\) is a \(\C\)-subalgebra and \(\fG^\dagger \subseteq \fG\) is an abelian ideal;

        \item \([x,f] = L_x f - \frac{1}{2}d f(x)\) for all \(x \in \fG\) and \(f \in \fG^\dagger\).\hfill $\lozenge$
    \end{enumerate}
    
    \end{Exp}

    It can be shown that a Hopf algebroid \(H\) over \(\C[\![\hbar]\!]\) which evaluates to the universal envelope \(U(\fG)\) of a Lie algebroid \(\fG\) for \(\hbar = 0\) induces a Lie bialgebroid \((\fG,\fG^\dagger)\). In fact, it can be shown that the datum of a Lie bialgebroid \((\fG,\fG^\dagger)\) is completely encoded into \(\fG\) and the action of \(d_{\fG^\dagger}\) on \(R\) and \(\fG\). Now, for a Hopf algebroid that evaluates to the universal enveloping algebra of \(\fG\), the \(\hbar\)-coefficient of the difference of the source and target map define the action of \(d_{\fG^\dagger}\) on \(R\) and the \(\hbar\)-coefficient of the skew-symmetrization its comultiplication defines the action of \(d_{\fG^\dagger}\) on \(\fG\).
    
    \begin{Def}
        The quantization of a Lie bialgebroid \(\fL\) is a topologically free Hopf algebroid \(H\) over \(\C[\![\hbar]\!]\) that evaluates to \(U(\fL)\) for \(\hbar = 0\) and induces the original Lie bialgebroid structure as described above.\hfill $\lozenge$
    \end{Def}

\subsubsection{Cotangent bialgebroid and double quotient of groupoid}\label{subsec:cotangbialg}

We now explain the Lie bialgebroid structure that is relevant to our work, and the idea of our quantization. This is very similar to the Lie algebra case we discussed in the previous work \cite{abedin2024yangian}. Let $\mf G$ be a Lie algebroid over $R$ and let us equip the direct sum $\fD:=\mf G\oplus \mf G^\dagger$ with the trivial double Courant algebroid structure described in Example \ref{ex:trivial_double}.

If we have a Lie algebroid decomposition $\mf G=\mf H\oplus \mf K$, then we obtain two Lagrangian subalgebroids
\be
\fL=\mf H\oplus \mf H^\perp,\qquad \fL^\dagger=\mf K\oplus \mf K^\perp.
\ee
Here $\mf H^\perp$ denotes the subspace of $\mf G^\dagger$ that pairs trivially with $\mf H$. This induces a Lie bialgebroid structure on $\fL$, and it is this structure that we are seeking to quantize. 

The idea of the quantization is again the double quotient construction. In the case of ordinary Lie algebras $\mf h\subseteq \mf g$, what we have shown in \cite{abedin2024yangian}, is that such a quantization comes from looking at the double quotient space
\be
H\!\setminus\! G/ H, \qquad G=\exp (\mf g),~ H=\exp (\mf h),
\ee
and the associated category of quasi-coherent sheaves. Essentially, we showed that the splitting of $\mf h\to \mf g$ given by $\mf k$ gives rise to a fiber functor for $\QCoh (H\setminus\! G/ H)$, whose endomorphism algebra quantizes the Lie bialgebra $\mf h\ltimes \mf h^\perp\subseteq T^*\mf g$.

Similarly, our quantization of the Lie algebroid $\fL$ is inspired by the geometry of double quotient. More specifically, there is a double quotient groupoid
\be
\CH\!\setminus\! \CG/ \CH, \qquad \CG=\exp (\mf G),~ \CH=\exp (\mf H),
\ee
whose category of sheaves is a monoidal category. We will use the splitting of $\mf G$ into $\mf H\oplus \mf K$ to construct a Hopf algebroid that quantizes $\fL$. This algebroid should be thought of as the endomorphism algebra of the fiber functor from $\QCoh (\CH\setminus\! \CG/ \CH)$ to $R$-$R$-bimodules, which is induced from the decomposition of $\mf G$. 

In this work we will only focus on Lie algebroids and splittings that arise from the geometry of $\Bun (\Sigma)$. However, our construction is completely general (as far as constructing the Hopf algebroid and twisting is concerned). In the next sections, we will explain how the geometry of $\Bun (\Sigma)$ around a regularly stable bundle $\CP$ leads to splittings of Lie algebroids over them, and we will construct Hopf algebroids that are twist-equivalent to the Yangian $Y_\hbar(\fd)$. 

\section{Setup}\label{sec:preliminaries}

In this section, we explain the set-up of this paper. We start with recalling the construction of Yangian $Y_\hbar (\fd)$ associated to $\fd$, since we will use certain features of this algebra in Section \ref{sec:constructHopf} and \ref{sec:dynamicalR}. We then recall geometry of the moduli space $\Bun (\Sigma)$ for a smooth projective curve $\Sigma$, especially around a regularly stable bundle $\CP$.

In Section \ref{sec:lie_bialgebroid}, we adapt the skew-symmetrization idea of \cite{abedin2024yangian} to the case of Lie algebroids and the dynamical spectral classical \(r\)-matrix from \cite{felder_kzb,abedin2024r}, which results in a Lie bialgebroid over the formal neighborhood of $\CP$ in $\Bun (\Sigma)$ that is defined by a skew-symmetric dynamical spectral classical \(r\)-matrix. We conclude this section by explaining how this Lie bialgebroid can be viewed as a dynamical twist of the classical limit Lie bialgebra of the Yangian over \(\fd\).

\subsection{Yangian over \(\fd\)}
\label{sec:Yangian}
Let us start with outlining the main results from \cite{abedin2024yangian} relevant to this work. We fix a simple complex connected Lie group $G$ with Lie algebra $\fg$, and denote by $\fd$ the \textit{cotangent Lie algebra} $T^*\fg$ of $\fg$. The Lie algebra $\fd$ is of the form
\be
\fd=\fg\ltimes \fg^*
\ee
and comes equipped with a bilinear pairing \(\langle,\rangle_\fd\), which simply pairs $\fg$ with $\fg^*$:
\begin{equation}
    \langle x + f, y +g \rangle_\fd = f(y) + g(x) \,,\qquad x,y \in \fg,f,g\in\fg^*.
\end{equation}
Let $C$ be the quadratic Casimir of $\fd$, which is of the form
\be
C \coloneqq \sum_{\alpha = 1}^d (I_\alpha \otimes I^\alpha + I^\alpha \otimes I_\alpha), 
\ee
 where \(\{I_\alpha\}_{\alpha =1}^d\subseteq \fg\) and \(\{I^\alpha\}_{\alpha = 1}^d\subseteq \fg^*\) are basis dual to each other. 

Let $\CK=\C\lpp t\rpp$ be the field of Laurent series, $\CO=\C\lbb t\rbb$ be the ring of power series and write \(\fd(\CO) \coloneqq \fd \otimes \CO\) as well as \(\fd(\CK) \coloneqq \fd \otimes \CK\). Consider the expression
    \begin{equation}\label{eq:yangs_rmatrix_d}
        \gamma(t_1,t_2) \coloneqq \frac{C}{t_1-t_2}\in \fd (\CK)\otimes \fd (\CK),
    \end{equation}
    where we expand $\frac{1}{t_1-t_2}$ in $t_2 = 0$. Then 
    \begin{equation}
        \delta \colon \fd(\CO) \longrightarrow \fd(\CO) \otimes \fd(\CO)\,,\qquad x \longmapsto [x(t_1) \otimes 1 + 1 \otimes x(t_2),r(t_1,t_2)]
    \end{equation}
    defines a Lie bialgebra structure on \(\fd(\CO)\). The above $\gamma$ is precisely Yang's $r$-matrix associated to $\fd$, and satisfies classical spectral Yang-Baxter equation, i.e.\
    \begin{equation}\label{eq:CYBE}
        [r^{(12)}(t_1,t_2),r^{(13)}(t_1,t_3)] + [r^{(12)}(t_1,t_2),r^{(23)}(t_2,t_3)] + [r^{(13)}(t_1,t_3),r^{(23)}(t_2,t_3)] = 0.
    \end{equation}
    holds for \(r = \gamma\).

The first main result of \cite{abedin2024yangian} is a quantization of this Lie bialgebra. 

\begin{Thm}\label{thm:previousresult1}
    There exists a Hopf algebra \(Y_\hbar(\fd)\) over $\C\lbb\hbar\rbb$ satisfying the following properties: 
\begin{enumerate}
    \item \(Y_\hbar(\fd)\) is a quantization of the Lie bialgebra $(\fd (\CO),\delta)$ defined above. 
    
    \item There exists a dense Hopf subalgebra \(Y^\circ_\hbar(\fd)\subseteq Y_\hbar(\fd)\) quantizing the Lie bialgebra $(\fd[t],\delta)$. 

    \item The Hopf algebra \(Y_\hbar(\fd)\) is the (up to unique isomorphism) unique  quantization that is graded with respect to loop-grading of \(t\) and with respect to the grading determined by the semi-direct product \(\fd = \fg \ltimes \fg^*\). We call \(Y_\hbar(\fd)\) the Yangian of \(\fd\). \hfill \qedsymbol
\end{enumerate}
\end{Thm}

Moreover, we showed that just as the usual Yangian, this Hopf algebra \(Y_\hbar(\fd)\) admits a spectral $R$-matrix. 

\begin{Thm}\label{thm:previousresult2}
    The Yangian \(Y^\circ_\hbar(\fd)\) admits an \(\CR_\gamma(z) \in (Y^\circ_\hbar(\fd) \otimes_{\C[\![\hbar]\!]} Y^\circ_\hbar(\fd))[\![z^{-1}]\!]\) satisfying
\begin{equation}\label{eq:Rmatrix_for_yangian}
    (\tau_z \otimes 1)\Delta_\hbar^{op}(x) = \CR_\gamma(z)((\tau_z \otimes 1)\Delta_\hbar(x)) \CR_\gamma(z)^{-1}\,,\qquad x \in Y^\circ(\fd)
\end{equation}
for the comultiplication \(\Delta_\hbar\) of \(Y_\hbar(\fd)\), where \(\tau_z\) is a quantization of the shift \(t \mapsto t + z\) by a formal parameter \(z\).
Moreover, \(\CR_\gamma(z)\) is a solution to the spectral quantum Yang-Baxter equation
\begin{equation}\label{eq:quantum_Yang_Baxter}
    \CR_\gamma^{(12)}(z_1-z_2)\CR_\gamma^{(13)}(z_1-z_3)\CR_\gamma^{(23)}(z_2-z_3) = \CR_\gamma^{(23)}(z_2-z_3)\CR_\gamma^{(13)}(z_1-z_3)\CR_\gamma^{(12)}(z_1-z_2)
\end{equation}
and quantizes \(r\), i.e.\ \(\CR_\gamma(z) = 1 + \hbar (\tau_z \otimes 1)r + \dots\) holds.\hfill \qedsymbol

\end{Thm}

\begin{Rem}
    Observe that in \cite{abedin2024yangian}, we worked with \(-\delta\) instead of \(\delta\). It turns out that working with \(\delta\) is more convenient for us in this work, and this convention change leads to minor changes in the following formulas compared to \cite{abedin2024yangian}.\hfill $\lozenge$
\end{Rem}

Let us outline the arguments behind Theorem \ref{thm:previousresult1} and Theorem \ref{thm:previousresult2}.

\subsubsection{A sketch of the construction of $Y_\hbar (\fd)$}

The construction of $Y_\hbar (\fd)$ took inspiration from geometry of equivariant affine-Grassmannian. Roughly speaking, this algebra is built to mimic the structure of $G(\CO)\!\setminus\!\wh \Gr_G$, where $\wh \Gr_G$ is the formal completion of $\Gr_G$ at identity. 

More specifically, the construction of the quantization \(Y_\hbar(\fd)\) of \((\fd[\![t]\!],\delta)\) relies on the following two facts:
\begin{enumerate}
    \item The Lie algebra decomposition \(\fg(\CK) = \fg(\CO) \oplus t^{-1}\fg[t^{-1}]\) induces a tensor product factorization \(U(\fg(\CK)) = U(t^{-1}\fg[t^{-1}]) \otimes U(\fg(\CO))\). The latter provides a right-action \(\rhd\) of \(U(\fg(\CO))\) on \(U(t^{-1}\fg[t^{-1}])\) and a left-action \(\lhd\) of \(U(t^{-1}\fg[t^{-1}])\) on \(U(\fg(\CO))\).
    
    \item The pairing 
    \begin{equation}
        \langle \hbar f,x\rangle = \textnormal{res}_{t = 0}\langle f,x\rangle_\fd
    \end{equation}
    for \(f \in \fg^*(\CO)\) and \(x \in t^{-1}\fg[t^{-1}]\) induces an isomorphism \(U(t^{-1}\fg[t^{-1}])^*[\![\hbar]\!] \cong S(\hbar\fg^*(\CO))[\![\hbar]\!]\).
    
\end{enumerate}

The algebra \(Y_\hbar(\fd)\) is defined as the semi-direct product, or smash product,
\be
Y_\hbar(\fd)\coloneqq U (\fg (\CO))\lbb\hbar\rbb\#_{\C[\![\hbar]\!]} S(\fg^*(\CO))[\![\hbar]\!],
\ee
where the action of \(U(\fg(\CO))[\![\hbar]\!]\) on the algebra \(S(\fg^*(\CO))[\![\hbar]\!]\) is given by
\begin{equation}
    \langle f \circ x,y\rangle := \langle f,x\rhd y\rangle. 
\end{equation}

In order to define the comultiplication \(\Delta_\hbar\) of \(Y_\hbar(\fd)\), we have to define it on its two factors \(U(\fg(\CO))[\![\hbar]\!]\) and \(S(\fg^*(\CO))[\![\hbar]\!]\) in a compatible way.
On \(S(\fg^*(\CO))[\![\hbar]\!]\) we define \(\Delta_\hbar\) dual to the multiplication map of \(U(t^{-1}\fg[t^{-1}])[\![\hbar]\!]\):
\begin{equation}
    \langle \Delta_\hbar(f),y_1 \otimes y_2\rangle \coloneqq \frac{1}{\hbar}\langle \Delta_\hbar(\hbar f),y_1 \otimes y_2\rangle \coloneqq \frac{1}{\hbar}\langle \hbar f,y_2y_1\rangle
\end{equation}
for all \(f \in \fg^*[\![t]\!]\) and \(y_1,y_2 \in U(t^{-1}\fg[t^{-1}])[\![\hbar]\!]\).

In order to define \(\Delta_\hbar\) on \(U(\fg(\CO))\) in a compatible way, we need to introduce a tensor series
\begin{equation}\label{eq:CEgamma}
    \CE_\gamma := \exp\left(\frac{\hbar}{t_1-t_2}\sum_{\alpha = 1}^d I_\alpha \otimes I^\alpha\right) \in U(t^{-1}\fg[t^{-1}])[\![\hbar]\!] \otimes_{\C[\![\hbar]\!]} S(\fg^*(\CO))[\![\hbar]\!].
\end{equation}
This tensor series $\CE_\gamma$ is uniquely characterized by the property that $\CE_\gamma (x)=x$ for any $x\in U(t^{-1}\fg[t^{-1}])$, where the evaluation of a tensor \(y \otimes f \in U(t^{-1}\fg[t^{-1}])[\![\hbar]\!] \otimes_{\C[\![\hbar]\!]} S(\fg^*(\CO))[\![\hbar]\!]\) on \(x \in U(\fg(\CO))[\![\hbar]\!]\) is \(f(x)y\). We define 
\begin{equation}
    \Delta_\hbar(x) = \CE^{-1}_\gamma(x \otimes 1 + 1 \otimes x)\CE_\gamma,
\end{equation}
where we think of $x \otimes 1 + 1 \otimes x$ and $\CE_\gamma$ as elements in $U(\fg(\CK))[\![\hbar]\!]\otimes_{\C[\![\hbar]\!]} Y_\hbar (\fd)$ and compute the multiplication in this algebra. It turns out that $\Delta_\hbar|_{U(\fg(\CO))}$ is valued in the subalgebra $U(\fg(\CO))[\![\hbar]\!]\otimes Y_\hbar (\fd)$ and, combined with \(\Delta_\hbar|_{S(\fg^*(\CO))}\) defined before, we obtain a well-defined coassociative algebra homomorphism \(\Delta_\hbar\colon Y_\hbar(\fd) \to Y_\hbar(\fd) \otimes_{\C[\![\hbar]\!]} Y_\hbar(\fd)\). This is the desired comultiplication.

The counit is defined on $S(\fg^*(\CO))[\![\hbar]\!]$ by 
\be
\epsilon (f)=\lag f, 1\rag, \qquad 1\in U(t^{-1}\fg[t^{-1}]),
\ee
and on $U(\fg(\CO))$ by $\epsilon(x)=0$ for all $x\in \fg(\CO)$. The antipode \(S\) is defined on $S(\fg^*(\CO))[\![\hbar]\!]$ by
\be
\lag S(f), x\rag=\lag f, S(x)\rag, \qquad x\in U(t^{-1}\fg[t^{-1}]),
\ee
and on $U(\fg(\CO))[\![\hbar]\!]$ essentially by the requirement that $\nabla ((S\otimes 1)\Delta_\hbar(x))=0$ for all \(x \in \fg(\CO)\). Here, \(\nabla \colon Y_\hbar(\fd) \otimes_{\C[\![\hbar]\!]} Y_\hbar(\fd) \to Y_\hbar(\fd)\) is the multiplication map. 

One can then check that these structures are compatible and defines a Hopf algbera structure on \(Y_\hbar(\fd)\), and that it indeed provides a quantization of \((\fd(\CO),\delta)\), proving Theorem \ref{thm:previousresult1}.1. The subalgebra of \(Y_\hbar(\fd)\) generated by \(\fd[t]\) is then the Hopf subalgebra \(Y^\circ_\hbar(\fd)\) desired in Theorem \ref{thm:previousresult1}.2. The uniqueness in Theorem \ref{thm:previousresult1}.3.\ is proved using Lie bialgebra cohomology in the sense of \cite[Section 9]{drinfeld1986quantum}.

\subsubsection{A sketch of constructing the spectral $R$-matrix}

In order to define the comultiplication \(\Delta_\hbar\) of \(Y_\hbar(\fd)\), we used the fact that \(U(t^{-1}\fg[t^{-1}])\) is an algebra. However, \(U(t^{-1}\fg[t^{-1}])=V_0(\fg)\) is also naturally also a vertex algebra, the vacuum vertex algebra of $\fg(\CK)$ at level $0$. This can be used to define a second meromorphic comultiplication \(\Delta_z\) on \(Y_\hbar(\fd)\), which is central to the construction of \(\CR_\gamma(z)\).

Let \(\CY \colon V_0(\fg) \to \textnormal{End}(V_0(\fg))[\![z,z^{-1}]\!]\) be the state-to-field map of \(V_0(\fg)\). We define \(\Delta_z\) on \(S(\fg^*(\CO))\) via 
\begin{equation}
    \langle \Delta_z(f),y_1 \otimes y_2\rangle := \frac{1}{\hbar}\langle \Delta_z(\hbar f),y_1 \otimes y_2\rangle := \frac{1}{\hbar}\langle \hbar f,\CY (y_2,z)y_1\rangle
\end{equation}
for all \(f \in \fg^*(\CO)\) and \(y_1,y_2 \in U(t^{-1}\fg[t^{-1}])[\![\hbar]\!]\).

In order to define \(\Delta_z\) on \(U(\fg(\CO))\), observe that the derivative \(\partial_t \colon \fg(\!(t)\!) \to \fg(\!(t)\!)\) induces a Hopf algebra derivation \(T\) of \(Y_\hbar(\fd)\). The map \(\tau_z \coloneqq e^{zT}\colon Y_\hbar(\fd) \to Y_\hbar(\fd)[\![z]\!]\) can be thought of as a quantization of \(t \mapsto t + z\). Then we define \(\Delta_z\) on \(U(\fg(\CO))\) as
\begin{equation}
    \Delta_z(x) = x \otimes 1 + 1 \otimes \tau_z(x)
\end{equation}
for all \(x \in \fg(\CO)\). It turns out that \(\Delta_z\) is a well-defined weakly coassociative and weakly cocommutative meromorphic comultiplication on \(Y_\hbar(\fd)\). The ``weakly'' here refers to the fact that one has to adjust the respective axioms by re-expanding Laurent series in appropriate domains.

It turns out that \(\Delta_z\) and \(\Delta_\hbar\) are related through a tensor series 
\begin{equation}\label{eq:R_s}
    R_s(z) = \exp\lp \frac{\hbar}{t_1-t_2+z}\sum_{\alpha = 1}^NI_\alpha \otimes I^\alpha\rp \in (Y_\hbar(\fd) \otimes_{\C[\![\hbar]\!]} Y_\hbar(\fd))[\![z^{-1}]\!]    
\end{equation}
via \((\tau_z\otimes 1)\Delta_\hbar(x)=R_s(z)\Delta_z(x)R_s(z)^{-1}\) for all \(x \in Y_\hbar(\fd)\). Here, we used the expansion
\begin{equation}\label{eq:expansion_in_Rs}
    \frac{1}{t_1-t_2+z} = \sum_{k = 0}^\infty \frac{(t_2-t_1)^k}{z^{k+1}} \in \C[t_1,t_2][\![z^{-1}]\!].
\end{equation}
Furthermore, \(R_s(z)\) satisfies
\begin{equation}
    (\Delta_{z_1}\otimes 1)(R_s(z_2)^{-1})R_{s}^{(12)}(z_1)^{-1}=(1\otimes \Delta_{z_2})(R_s(z_1+z_2)^{-1}) R_{s}^{(23)}(z_2)^{-1}.
\end{equation}
Using these properties of \(R_s(z)\) as well as the weak cocommutativity of \(\Delta_z\), one can show that
\begin{equation}\label{eq:definition_R}
    \CR_\gamma(z) = R_s(-z)^{(21)}R_s(z)^{-1}
\end{equation}
satisfies the claims of Theorem \ref{thm:previousresult2}. 

\subsection{Moduli space of \(G\)-bundles and its formal completion}\label{sec:moduli_space}
As before, let \(G\) be a simple complex connected Lie group $G$ with Lie algebra $\fg$. We denote by \(\textnormal{Bun}_G(\Sigma)\) the moduli space of \(G\)-bundles on \(\Sigma\). It is well-known that this is a smooth algebraic stack of pure dimension
\(N \coloneqq (g-1) \dim(\fg)\). Moreover, the open substack of regularly stable \(G\)-bundles, i.e.\ \(G\)-bundles whose automorphisms coincide with the center \(C(G)\) of \(G\), is a \(C(G)\)-gerbe over a smooth quasi-projective algebraic variety, which we denote by \(\textnormal{Bun}_G^0(\Sigma)\).    

Since \(G\) is simple, for every point \(p\in \Sigma\), \(\Bun(\Sigma)\) has a the following very useful description as a double quotient of loop groups: 
\begin{equation}\label{eq:loop_group_uniformization}
    \textnormal{Bun}_G(\Sigma) = G(\CO) \setminus G(\CK) \,/\, G(\Sigma^\circ)\,,\qquad \Sigma^\circ \coloneqq \Sigma \setminus \{p\}.  
\end{equation}
This identification can be understood set-theoretically as follows:
\begin{itemize}
    \item Using the open covering \(\Sigma = \Sigma^\circ \cup \mathbb{D}\), where \(\mathbb{D} \cong \textnormal{Spec}(\CO)\) is the formal disc around \(p\), any element \(g \in G(\mathbb{D}\setminus \{p\}) \cong G(\CK)\) can be used to glue the trivial \(G\)-bundles on \(\mathbb{D}\) and \(\Sigma^\circ\) together in order to obtain a \(G\)-bundle \(\CP_g\);
    
    \item Every \(G\)-bundle \(\CP\) is \'etale trivializable on the affine curve \(\Sigma^\circ\) and the formal disc \(\mathbb{D}\) and therefore \(\CP \cong \CP_g\) for some \(g \in G(\CK)\);

    \item Two elements \(g,g'\in G(\CK)\) satisfy \(\CP_g \cong \CP_{g'}\) if and only if they are related by a change of trivializations \(h_+ \in G(\CO)\) and \(h_- \in G(\Sigma^\circ)\) on \(\mathbb{D}\) and \(\Sigma^\circ\) such that \(g' = h_+ g h_-\).
\end{itemize}
Importantly, the identification \eqref{eq:loop_group_uniformization} is not only set-theoretic but also algebro-geometric in the sense that it is an isomorphism of stacks. Thereby, recall that \(G(\CK)\) is an ind-affine algebraic group of infinite type, \(G(\Sigma^\circ) \subseteq G(\CK)\) is an ind-algebraic subgroup, \(G(\CO) \subseteq G(\CK)\) is an affine algebraic subgroup, the quotient \(G(\CK)\,/\, G(\Sigma^\circ)\) is a scheme (of infinite type), and therefore the quotient stack on the RHS in \eqref{eq:loop_group_uniformization} is indeed a stack.

Let \((G(\CO)\setminus G(\CK))^0\) and \(G(\CK)^0\) be the preimage of the open substack of regularly stable \(G\)-bundles under the canonical projection \(G(\CO)\setminus G(\CK) \to \Bun(\Sigma)\) and \(G(\CK) \to \Bun(\Sigma)\) respectively. 
The projection \((G(\CO)\setminus G(\CK))^0 \to \Bun^0(\Sigma)\) admits local \'etale quasi-sections; see e.g.\ \cite{laszlo_hitchin_wzw}. Moreover, the projection \(G(\CK) \to G(\CO)\setminus G(\CK)\) admits a section in the Zariski topology. Combined, we can see that \(G(\CK)^0 \to \Bun^0(\Sigma)\) admits a local \'etale quasi-section. 

This means that for every \(\CP \in  \Bun^0(\Sigma)\), where we recall that \(\Bun^0(\Sigma)\) is the coarse moduli space of regularly stable \(G\)-bundles, there exists an \'etale morphism \(P \colon U \to \Bun^0(\Sigma)\) and a morphism \(\sigma \colon U \to G(\CK)\) such that the diagram
\begin{equation}\label{eq:etale_quasisection}
    \xymatrix{U \ar[r]^\sigma\ar[d]_{P} & G(\CK) \\  \Bun^0(\Sigma)  & G(\CK)_0 \ar[l]\ar[u]_\subseteq},
\end{equation}
commutes and \(\CP \in P(U)\). In other words, for every \(u \in U\), \(\sigma(u) \in G(\CK)\) defines the \(G\)-bundle \(P(u) \in  \Bun^0(\Sigma)\) in the sense that \(P(u) = \CP_{\sigma(u)}\). 

Now \(U\) is smooth, so for any \(u \in U\) such that \(P(u) = \CP\), we can pass to the formal completion \(B\) of \(U\) around \(u\) and obtain a formal section \(\sigma \colon B \to G(\CK)\) of the projection \(G(\CK) \to \Bun(\Sigma)\) around \(\CP\). Since \(U\) is of dimension \(N\), we can chose coordinates \(\lambda = (\lambda_1,\dots,\lambda_N)\) on \(B\), i.e.\ we fix an isomorphism \(B \cong \textnormal{Spf}(R)\), where \(R = \C[\![\lambda]\!] \coloneqq \C[\![\lambda_1,\dots,\lambda_N]\!]\). Let us summarize:

\begin{Prop}\label{prop:formal_section_of_Gbundles}
    For every regularly stable \(G\)-bundle \(\CP\), there exists a regular map
    \begin{equation}
        \sigma \colon B \coloneqq \textnormal{Spf}(R) \longrightarrow G(\CK)    
    \end{equation}
    and an embedding \(B \to \Bun(\Sigma)\) such that \(\CP = \CP_{\sigma(0)}\) and
    \begin{equation}\label{eq:etale_quasisection}
        \xymatrix{B \ar[rr]^\sigma\ar[dr] && G(\CK) \ar[dl] \\ & \Bun(\Sigma) &}
    \end{equation}
    commutes. Here, \(R = \C[\![\lambda]\!] = \C[\![\lambda_1,\dots,\lambda_N]\!]\) and \(0 \in B\) is the maximal ideal of \(R\).\hfill \qed
\end{Prop}

\subsection{Lie bialgebroids from the moduli space of $G$-bundles}\label{sec:LieHitchin}

In this section, we construct a Lie bialgebroid structure to any formal trivialization \(\sigma\colon B \to G(\CK)\) of the moduli space of \(G\)-bundles as in Proposition \ref{prop:formal_section_of_Gbundles}. More precisely, consider the Lie algebroid $\mf G \coloneqq T_R \ltimes \fg (\CK)\lbb\lambda\rbb$ and its subalgebroid \(\fG_+ \coloneqq T_R \ltimes \fg(\CO) \subseteq \fG\), where \(T_R  \coloneqq \textnormal{Der}(R)\) is the Lie algebroid of continuous derivations of \(R\). Following \cite{felder_kzb,abedin2024r}, we will associate another subalgebroid \(\fG_-^\sigma \subseteq \fG\) to \(\sigma\) such that \(\fG = \fG_+ \oplus \fG_-^\sigma\). This will give rise to a Lie bialgebroid structure according to the general scheme from Section \ref{subsec:cotangbialg}.

Furthermore, we will explain in this section how this Lie bialgebroid is essentially controlled by the classical dynamical $r$-matrix from \cite{felder_kzb,abedin2024r} and show that this Lie bialgebroid structure can be obtained from dynamically twisting the Lie bialgebra structure induced from Yang's $r$-matrix $\gamma$. 

\subsubsection{Construction of the Lie bialgebroid structure}\label{sec:lie_bialgebroid}

As explained in Section \ref{sec:moduli_space}, for the formal neighbourhood \(B\) of a regularly stable \(G\)-bundle in the moduli space \(\textnormal{Bun}_G\) of \(G\)-bundles, there exists a section \(\sigma \colon B \to G(\CK)\) of the canonical projection \(G(\CK) \to \textnormal{Bun}_G(\Sigma) = G(\CO) \setminus G(\CK) \, /\, G(\Sigma^\circ)\) given by loop group uniformization. We fix coordinates on \(B\), in particular, we identify the ring of functions \(R\) on \(B\) with \(\C[\![\lambda]\!]= \C[\![\lambda_1,\dots,\lambda_N]\!]\). 

The map \(\sigma\) defines a complete formal family of \(G\)-bundles around the chosen regularly stable \(G\)-bundle. Moreover, the Lie algebra 
\begin{equation}
    \gout^\sigma \coloneqq \textnormal{Ad}(\sigma)\fg(\Sigma^\circ)[\![\lambda]\!]    
\end{equation}
realizes the sections of the adjoint bundle associated to this family of \(G\)-bundles.

Consider \(\xi_\alpha \coloneqq \sigma^{-1}\partial_\alpha \sigma \in \fg(\CK)[\![\lambda]\!]\) for every \(\alpha \in \{1,\dots, N\}\), where \(\partial_\alpha \coloneqq \partial/\partial \lambda_\alpha\).  It is easy to calculate that
\begin{equation}
    \partial_\alpha \xi_\beta - \partial_\beta \xi_\alpha + [\xi_\alpha,\xi_\beta] = 0.
\end{equation}
Consider the Lie algebroid \(\fG \coloneqq T_R \ltimes \fg(\CK)[\![\lambda]\!]\) with the anchor map given by the canonical projection \(\psi \colon \fG \to T_R\) and the semi-direct product bracket given by the action of \(T_R\) on \(\fg(\CK)[\![\lambda]\!]\) by derivations in \(\lambda = (\lambda_1,\dots,\lambda_N)\).

Then the space
\begin{equation}
    T^\sigma_R \coloneqq \textnormal{Span}_R\{\nabla_\alpha \coloneqq \partial_\alpha + \xi_\alpha\} \subseteq \fG
\end{equation}
is a subalgebroid isomorphic to the tangent space \(T_R \coloneqq \bigoplus_{\alpha = 1}^N R\partial_\alpha\) of \(B\). Furthermore, it is easy to see that \([T^\sigma_R,\gout^\sigma] \subseteq \gout^\sigma\) holds. In particular, the isomorphism \(T_R \to T_R^\sigma\) defines a connection on \(\gout^\sigma\).

On the other hand, \(\partial_\alpha \mapsto \xi_\alpha\) defines an isomorphism from \(T_R\)  to the double quotient \(\fg(\CO)[\![\lambda]\!] \setminus \fg(\CK)[\![\lambda]\!]/\gout^\sigma\). Therefore, the spaces
\begin{equation}
    \begin{split}
        &\widetilde{\fG}^\sigma_- \coloneqq \textnormal{Span}_R\{\xi_\alpha\}_{\alpha = 1}^N \oplus \gout^\sigma \subseteq \fg(\CK)[\![\lambda]\!],\\ &\fG^\sigma_- \coloneqq T^\sigma_R \ltimes \gout^\sigma \subseteq \fG =  T_R \ltimes \fg(\CK)[\![\lambda]\!]        
    \end{split}
\end{equation}
satisfy
\begin{equation}
    \fg(\CK)[\![\lambda]\!] = \fg(\CO)[\![\lambda]\!] \oplus \widetilde{\fG}_-^\sigma \textnormal{ and } \fG = \fG_+ \oplus \fG_-^\sigma
\end{equation}
respectively, where \(\fG_+ \coloneqq T_R \ltimes \fg(\CO)[\![\lambda]\!]\).

Let us note that \(\fG^\sigma_- \subseteq \fG\) is a subalgebroid, but \(\widetilde{\fG}^\sigma_- \subseteq \fg(\CK)[\![\lambda]\!]\) is not a subalgebra. Furthermore, \(\widetilde{\fG}^\sigma_-\) is simply the image of \(\fG^\sigma_-\) under the canonical projection 
\begin{equation}
    \fG = T_R \ltimes \fg(\CK) \to \fg(\CK).
\end{equation}
The isomorphism of Lie algebroids \(T_R \to T_R^\sigma \subseteq \fG_-^\sigma\) now defines a connection on \(\fG_-^\sigma\).

We want to described the Lie bialgebroid structure associated to the splitting \(\fG = \fG_+ \oplus \fG_-^\sigma\) in the vein of Section \ref{subsec:cotangbialg}.
Let \(\fD \coloneqq \fG \ltimes \fG^\dagger = (T_R \ltimes \fg(\CK)) \ltimes (\fg^*(\CK) \oplus T_R^\dagger)\) be the trivial double of \(\fG\) described in Example \ref{ex:trivial_double}. Furthermore, let \(\fD_+ \coloneqq \fG_+ \ltimes \fg^*(\CO)[\![\lambda]\!] = T_R \ltimes \fd(\CO)[\![\lambda]\!]\) and \(\fD_-^\sigma \coloneqq \fG_-^\sigma \ltimes \fG_-^{\sigma,\bot}\). 
Observe that \(\fD_+ = \fG_+ \oplus \fG_+^\bot\) and \(\fD_-^\sigma = \fG_-^\sigma \oplus \fG_-^{\sigma,\bot} \subseteq \fD\) are subalgebroids, since \([\fG_+, \fG_+^{\bot}] \subseteq \fG_+^{\bot}\) and \([\fG_-^\sigma, \fG_-^{\sigma,\bot}] \subseteq \fG_-^{\sigma,\bot}\) follows immeditaly from Axiom 4.\ in Definition \ref{def:courant_algebroid}. In particular, \(\fD = \fD_+ \oplus \fD_-^\sigma\) is a decomposition into Lagrangian subalgebroids and thus defines a Lie bialgebroid structure on \(\fD_+\).

For sake of completeness, let us describe \(\fD_-^\sigma\) in more detail. Let \(\{d\lambda_\alpha\}_{\alpha = 1}^N \subseteq T_R^\dagger\) and \(\{\omega_\alpha\}_{\alpha = 1}^N \subseteq \gout^{\sigma,\bot} \cap \fg^*(\CO)[\![\lambda]\!]\) be the dual bases of \(\{\partial_\alpha\}_{\alpha = 1}^N \subseteq T_R\) and  \(\{\xi_\alpha\}_{\alpha = 1}^N \subseteq \fg(\CK)\) respectively. Then 
\begin{equation}\label{eq:description_Gsigma-}
    \fG^{\sigma,\bot}_- = \widetilde{\fG}^{\sigma,\bot}_- \oplus \textnormal{Span}_R\{\omega_\alpha - d\lambda_\alpha\}_{\alpha = 1}^N 
\end{equation}
holds. Here, \(\widetilde{\fG}^{\sigma,\bot}_- \subseteq \fg^*(\CK)\) is the orthogonal complement of \(\widetilde{\fG}^{\sigma}_- \subseteq \fg(\CK)\).

Recall that the Lie bialgebroid structure on \(\fD_+\) is completely determined by its Lie algebroid structure, the dual anchor map, and the dual differential \(d^\sigma_\dagger \coloneqq d_{\fD_-^\sigma}\colon \fD_+ \to \fD_+ \wedge \fD_+\). The Lie algebroid and dual anchor structure is immediate from the Lie algebroid structure of \(\fD\). We would therefore like to described \(d_\dagger\).
To do so, we will use a classical dynamical \(r\)-matrix.

Let us consider the unique element
\begin{equation}
    r(t_1,t_2;\lambda) \in \frac{1}{t_1-t_2}\sum_{\alpha = 1}^n I_\alpha \otimes I^\alpha + (\fg \otimes \fg^*)[\![t_1,t_2]\!][\![\lambda]\!]
\end{equation}
such that the Taylor expansion in \(t_2 = 0\) provides an element of \(\widetilde{\fG}^{\sigma}_- \otimes \fg^*[\![t_2]\!][\![\lambda]\!]\). This element satisfies the following version of the spectral dynamical classical Yang-Baxter equation:
\begin{equation}\label{eq:dynamicalCYBE_r}
    \begin{split}
        [r^{(12)}(t_1,t_2;\lambda),& r^{(13)}(t_1,t_3;\lambda)] + [r^{(12)}(t_1,t_2;\lambda),r^{(23)}(t_2,t_3;\lambda)] + [r^{(32)}(t_3,t_2;\lambda),r^{(13)}(t_1,t_3;\lambda)] \\&= \sum_{\alpha = 1}^N \left(\omega_\alpha^{(3)}(t_3;\lambda)\partial_\alpha r^{(12)}(t_1,t_2;\lambda) - \omega_\alpha^{(2)}(t_2;\lambda)\partial_\alpha r^{(13)}(t_1,t_3;\lambda)\right);
    \end{split}
\end{equation}
see \cite{abedin2024r}.
Here, under consideration of \(\fg \subseteq U(\fg)\), the notations \((\cdot)^{(i)},(\cdot)^{(ij)}\) can be understood coefficient-wise as e.g.\ \(a^{(2)} = 1 \otimes a \otimes 1, (a\otimes b)^{(13)} = a \otimes 1 \otimes b \in U(\fd)^{\otimes 3}\) and the Lie brackets and multiplications are understood coefficient-wise in \(U(\fd)^{\otimes 3}\) as well. Let us remark that \(r\) here is \(\overline{r}\) in \cite{abedin2024r}. 

\begin{Prop}\label{prop:dual_differential}
    The skew-symmetrization 
    \begin{equation}\label{eq:def_rho}
        \rho(t_1,t_2;\lambda) \coloneqq r(t_1,t_2;\lambda) - r^{(21)}(t_2,t_1;\lambda) \in \gamma + (\fd \otimes \fd)[\![t_1,t_2]\!][\![\lambda]\!]
    \end{equation}satisfies the spectral classical dynamical Yang-Baxter equation
    \begin{equation}\label{eq:dynamicalCYBE_rho_body}
    \begin{split}
        &[\rho^{(12)}(t_1,t_2;\lambda),\rho^{(13)}(t_1,t_3;\lambda)] + [\rho^{(12)}(t_1,t_2;\lambda),\rho^{(23)}(t_2,t_3;\lambda)] + [\rho^{(32)}(t_3,t_2;\lambda),\rho^{(13)}(t_1,t_3;\lambda)] \\&= \sum_{\alpha = 1}^N \left(\omega^{(1)}_\alpha \partial_\alpha \rho^{(23)}(t_2,t_3;\lambda) - \omega_\alpha^{(2)}(t_2;\lambda)\partial_\alpha \rho^{(13)}(t_1,t_3;\lambda) +\omega_\alpha^{(3)}(t_3;\lambda)\partial_\alpha \rho^{(12)}(t_1,t_2;\lambda)\right).
    \end{split}
    \end{equation}
    Furthermore, the dual differential \(d^\sigma_\dagger\) is given by:
    \begin{equation}\label{eq:differential_via_rho}
    \begin{split}
        d^\sigma_\dagger x =  \left[x(t_1;\lambda) \otimes 1 + 1 \otimes x(t_2;\lambda),\sum_{\alpha = 1}^N \partial_\alpha \wedge \omega_\alpha + \rho(t_1,t_2;\lambda)\right].
    \end{split}
\end{equation}
\end{Prop}
\begin{proof}
First of all, using \([\fg^*,\fg^*]=\{0\}\), we have 
{\footnotesize\begin{equation}
    \begin{split}
        &[\rho^{(12)}(t_1,t_2;\lambda),\rho^{(13)}(t_1,t_3;\lambda)] + [\rho^{(12)}(t_1,t_2;\lambda),\rho^{(23)}(t_2,t_3;\lambda)] + [\rho^{(32)}(t_3,t_2;\lambda),\rho^{(13)}(t_1,t_3;\lambda)]
        \\& = [r^{(12)}(t_1,t_2;\lambda),r^{(13)}(t_1,t_3;\lambda)] + [r^{(12)}(t_1,t_2;\lambda),r^{(23)}(t_2,t_3;\lambda)] + [r^{(32)}(t_3,t_2;\lambda),r^{(13)}(t_1,t_3;\lambda)]
        \\& - [r^{(32)}(t_3,t_2;\lambda),r^{(31)}(t_3,t_1;\lambda)] + [r^{(21)}(t_2,t_1;\lambda),r^{(32)}(t_3,t_2;\lambda)] - [r^{(12)}(t_1,t_2;\lambda),r^{(31)}(t_3,t_1;\lambda)]
        \\&- [r^{(21)}(t_2,t_1;\lambda),r^{(23)}(t_3,t_2;\lambda)]  - [r^{(21)}(t_2,t_1;\lambda),r^{(13)}(t_1,t_3;\lambda)]  + [r^{(23)}(t_2,t_3;\lambda),r^{(31)}(t_3,t_1;\lambda)]
        \\&= \sum_{\alpha = 1}^N \left(\omega_\alpha^{(3)}(t_3;\lambda)\partial_\alpha r^{(12)}(t_1,t_2;\lambda) - \omega_\alpha^{(2)}(t_2;\lambda)\partial_\alpha r^{(13)}(t_1,t_3;\lambda)\right)
        \\& - \sum_{\alpha = 1}^N \left(\omega_\alpha^{(3)}(t_1;\lambda)\partial_\alpha r^{(12)}(t_3,t_2;\lambda) - \omega_\alpha^{(2)}(t_2;\lambda)\partial_\alpha r^{(13)}(t_3,t_1;\lambda)\right)^{(13)}
        \\& - \sum_{\alpha = 1}^N \left(\omega_\alpha^{(3)}(t_3;\lambda)\partial_\alpha r^{(12)}(t_2,t_1;\lambda) - \omega_\alpha^{(2)}(t_1;\lambda)\partial_\alpha r^{(13)}(t_3,t_2;\lambda)\right)^{(12)}
        \\& = \sum_{\alpha = 1}^N \left(\omega^{(1)}_\alpha \partial_\alpha \rho^{(23)}(t_2,t_3;\lambda) - \omega_\alpha^{(2)}(t_2;\lambda)\partial_\alpha \rho^{(13)}(t_1,t_3;\lambda) +\omega_\alpha^{(3)}(t_3;\lambda)\partial_\alpha \rho^{(12)}(t_1,t_2;\lambda)\right).
    \end{split}
\end{equation}}
Here, we used \eqref{eq:dynamicalCYBE_r} three times. This proves \eqref{eq:dynamicalCYBE_rho_body}.

Let us turn to the proof of \eqref{eq:differential_via_rho}. For \(x \in \fD_+\) and \(w_1,w_2 \in \fD_-^\sigma\) we have:
\begin{equation}
    \begin{split}
        &\langle d^\sigma_{\dagger} x,w_1 \otimes w_2\rangle = \psi(w_1)w_2(x) -\psi(w_2)w_1(x)-\langle x,[w_1,w_2]\rangle 
        \\& =\psi(w_1)w_2(x) -\psi(w_2)w_1(x)-\langle x,\psi(w_1)w_2\rangle + \langle x,\psi(w_2)w_1\rangle \\&-\langle x,[w_1-\psi(w_1),w_2-\psi(w_2)]\rangle
        \\&=\psi(w_1)w_2(x) -\psi(w_2)w_1(x)-\langle x,\psi(w_1)w_2\rangle + \langle x,\psi(w_2)w_1\rangle
        \\&+\langle [x,w_2-\psi(w_2)],w_1-\psi(w_1)\rangle
        \\&=\langle \left[x(t_1;\lambda) \otimes 1 + 1 \otimes x(t_2;\lambda),\sum_{\alpha = 1}^N \partial_\alpha \wedge \omega_\alpha +  \rho(t_1,t_2;\lambda)\right],w_1 \otimes w_2\rangle.
    \end{split}
\end{equation}
Here, we used that 
\begin{equation}
        \begin{split}
        &\sum_{\alpha = 1}^N\langle \partial_\alpha x \otimes \omega_\alpha + \partial_\alpha \otimes [\omega_\alpha,x],w_1 \otimes w_2\rangle
        \\&=  
        \langle \psi(w_2)x,w_1\rangle +\langle [\pi_{\fg^*(\CO)}(w_1),x],w_2\rangle
        \\&=\psi(w_2)w_1(x) - \langle x,\psi(w_2)w_1\rangle +\langle [\pi_{\fg^*(\CO)}(w_1),x],w_2\rangle,
        \end{split}
\end{equation}
where \(\pi_{\fg^*(\CO)} \colon \fD \to \fg^*(\CO)\) is the canonical projection induced by the description \eqref{eq:description_Gsigma-}, and the fact that
\begin{equation}
    \begin{split}
        &\langle [x(t_1;\lambda) \otimes 1 + 1 \otimes x(t_2;\lambda), \rho(t_1,t_2;\lambda)],w_1 \otimes w_2\rangle \\&= \langle [x,w_2 - \psi(w_2) - \pi_{\fg^*(\CO)}(w_2)],w_1 - \psi(w_1) - \pi_{\fg^*(\CO)}(w_1)\rangle.        
    \end{split}
\end{equation}
holds. The latter identity follows from the fact that \(\rho \in \gamma + (\fd \otimes \fd)[\![t_1,t_2]\!][\![\lambda]\!]\) and that the Taylor expansion of \(\rho\) in \(t_2 = 0\) is an element of \((\wt\fG_-^\sigma \oplus \widetilde{\fG}^{\sigma,\bot}_-)\otimes \fd(\CO)\).
\end{proof}

\subsubsection{Realization by twisting}\label{subsec:trivialdyn}

Recall the Lie bialgebra structure \(\delta\) on \(\fd(\CO)\) from Section \ref{sec:Yangian} defined by 
\begin{equation}
    \delta(x) = [x(t_1) \otimes 1 + 1 \otimes x(t_2),\gamma(t_1,t_2)]
\end{equation}
for Yang's \(r\)-matrix:
\begin{equation}
    \gamma(t_1,t_2) = \frac{1}{t_1-t_2}\sum_{\alpha = 1}^{\dim(\fg)}\left(I^\alpha \otimes I_\alpha + I_\alpha \otimes I^\alpha\right).
\end{equation}
This Lie bialgebra structure is determined by the Manin triple 
\begin{equation}
    \fd(\!(t)\!) = \fd[\![t]\!] \oplus (t^{-1}\fg[t^{-1}] \oplus t^{-1}\fg^*[t^{-1}]).    
\end{equation}
Consider \(\fD_- \coloneqq t^{-1}\fg[t^{-1}][\![\lambda]\!] \oplus t^{-1}\fg^*[t^{-1}][\![\lambda]\!] \oplus T_R^\dagger\). Then \(\fD = \fD_+ \oplus \fD_-\) is a splitting
of the Courant algebroid \(\fD\) into Lagrangian subalgebroids and thus defines the structure of a Lie bialgebroid on \(\fD_+\). This Lie bialgebroid can be viewed as a natural extension of the Lie bialgebra \((\fd(\CO),\delta)\) to the algebroid setting. It is still determined by Yang's \(r\)-matrix \(\gamma\) in the sense that
\begin{equation}
    \begin{split}
        d_\dagger x = [\gamma(t_1,t_2),x(t_1;\lambda)\otimes 1 + 1 \otimes x(t_2;\lambda)],
    \end{split}
\end{equation}
where \(d_\dagger\) is the dual differential of this Lie bialgebroid. Indeed, it is easy to calculate
\begin{equation}
    \begin{split}
        \langle d_\dagger x,w_1 \otimes w_2\rangle &= \underbrace{\psi(w_1)}_{= 0} w_2(x) - \underbrace{\psi(w_2)}_{= 0}w_1(x) - \langle x, [w_1,w_2]\rangle = \langle [x,w_2],w_1\rangle 
        \\&= \langle [x(t_1;\lambda) \otimes 1 + 1 \otimes x(t_2,\lambda),\gamma(t_1,t_2)],w_1 \otimes w_2\rangle.
    \end{split}
\end{equation}
Here, we used that \([w_1,w_2] \in t^{-1}\fd[t^{-1}][\![\lambda]\!]\) and therefore \(d_\dagger \partial_\alpha = 0\).

In the theory of Lie bialgebras, two Lie bialgebra structures \(\delta_1\) and \(\delta_2\) on the same Lie algebra \(\mf b\) are called twisted to each other if their classical double \(D(\mf b,\delta_1)\) and \(D(\mf b,\delta_2)\) are isomorphic in a way that the bilinear form is preserved and \(\mf b\) is mapped to itself; see \cite{drinfeld1989quasi,karolinsky_Stolin}. In this sense, the Lie bialgebroids defined by the Manin triples \(\fD = \fD_+ \oplus \fD_-\) and \(\fD = \fD_+\oplus \fD_-^\sigma\) are dynamically twisted to each other.

\section{Quantization of the Lie bialgebroid}\label{sec:constructHopf}

In this section, we construct the quantization $\Upsilon_\hbar^\sigma (\fd)$ of the Lie bialgebroid structure on $\fD_+=T_B\ltimes \fd (\CO)\lbb\lambda\rbb$ defined in Section \ref{sec:lie_bialgebroid}. We will also define a Hopf subalgebroid $\wt \Upsilon_\hbar^\sigma (\fd)$ generated by $\fd (\CO)$ and $R$. The essence of this construction follows from the idea of our previous work \cite{abedin2024yangian}, as well as the foundations on formal groupoids and Lie algebroids that we discussed in Section \ref{sec:generalities}. However, it is less straightforward compared to the usual Lie algebra case, as one needs to always keep track of the source and target maps. 

\subsection{Defining the algebra}

The starting point of this construction is the decomposition
\be
\mf G=\mf G_+\oplus \mf G_-^\sigma
\ee
from Section \ref{sec:lie_bialgebroid}.
Using this decomposition, the multiplication map gives an isomorphism
\be
U_R(\mathfrak G)=U_R(\mf G_-^\sigma)\rltensor_R U_R(\mathfrak G_+)
\ee
of cocommutative coalgebras over $R$. As discussed in Section \ref{subsubsec:LieU}, this induces a left action $\rhd$ of $U_R(\mathfrak G_+)$ on $U_R(\mf G_-^\sigma)$ and a right action $\lhd$ of $U_R(\mf G_-^\sigma)$ on $U_R(\mathfrak G_+)$, satisfying the compatibility conditions from Proposition \ref{prop:matchedULie}. 

Note that as a module over $R$, $\mf G_-^\sigma$ is canonically isomorphic to $\fg(\CK)[\![\lambda]\!]/\fg(\CO)\lbb\lambda\rbb$, and therefore we can identify the dual of $\mf G_-^\sigma$ with $\mf G_-^{\sigma,\dagger}=\fg^*(\CO)\lbb\lambda\rbb$. Consequently, we find an identification \(U_R(\fG_-^\sigma)^\dagger \cong \hat S_R(\fG_-^{\sigma,\dagger}) = \hat S_R(\fg^*(\CO)[\![\lambda]\!])\) of commutative \(R\)-algebras. In the \(\hbar\)-adic language, under consideration of \(\hat S_R(\hbar \fG_-^{\sigma,\dagger})[\![\hbar]\!] = S_R(\hbar \fG_-^{\sigma,\dagger})[\![\hbar]\!]\), we get the following identification of commutative \(R[\![\hbar]\!]\)-algebras:
\be
U_R(\fG_-^\sigma)^\dagger[\![\hbar]\!] = S_R (\hbar \mf G_-^{\sigma,\dagger})[\![\hbar]\!]=S_R(\hbar\fg^*(\CO)\lbb\lambda\rbb)[\![\hbar]\!]. 
\ee
Here, we recall that the canonical pairing is extended by \(\C[\![\hbar]\!]\)-bilinearly from \(\langle \hbar f, x \rangle = f(x)\) for \(f \in \fG^\dagger\) and \(x \in \fG\). 

Let $\psi$ be the anchor map for $U_R(\mathfrak G)$, which we recall that on primitive elements it is simply given by the projection:
\be
\psi\colon  T_R\ltimes \gkl\longrightarrow T_R.
\ee
Note that for any $x\in  \mf G_+$, the dual action \((x,f) \mapsto f(x\rhd- )\)
is no longer an element that is $R$-linear on the left, and therefore does not give a well-defined action of \(U_R(\fG_+)\) on $S_R (\mf G_-^{\sigma,\dagger})[\![\hbar]\!]$. Instead, we have to define the following:
\be\label{eq:dual_action_to_rhd}
\langle x\circ f, a\rangle \coloneqq \psi(x)\langle f,a\rangle-\langle f, x\rhd a\rangle \,,\qquad f \in S_R (\mf G_-^{\sigma,\dagger})[\![\hbar]\!], x \in \fG_+, a \in U(\fG_-^\sigma)[\![\hbar]\!].
\ee

\begin{Lem}
    Equation \eqref{eq:dual_action_to_rhd} gives a well-defined action of $\mathfrak G_+$ on $S_R (\mf G_-^{\sigma,\dagger})[\![\hbar]\!]$ by algebra derivations. This action extends to an action of $U_R(\mathfrak G_+)[\![\hbar]\!]$ on $S_R (\mf G_-^{\sigma,\dagger})[\![\hbar]\!]$.
    
\end{Lem}

\begin{proof}
For any \(x \in \fG_+\), \(f \in S_R (\mf G_-^{\sigma,\dagger})[\![\hbar]\!]\), and \(a \in U_R(\fG_-^\sigma)[\![\hbar]\!]\), the calculation
\be
\begin{split}
    \langle x\circ f,ra\rangle&=\psi(x)(r\langle f,a\rangle )-\langle f,x\rhd (r a)\rangle 
    \\&=r\langle x\circ f,a\rangle +\psi(x)(r) \langle f,a\rangle -\langle f,\psi(x)(r) a\rangle =r\langle x\circ f,a\rangle .
\end{split}
\ee
shows that \(x\circ f\) is \(R\)-linear on the left and therefore the action is well-defined.

Furthermore, for another \(y \in \fG_+\), we see that $\circ$ satisfies
   \be
        \begin{split}
            &\langle x\circ (y\circ f)-y\circ( x\circ f),a\rangle \\&= \psi(x)\psi(y)\langle f,a\rangle - \psi(x)\langle f,y\rhd a\rangle - \psi(y) \langle f,x \rhd a\rangle + \langle f, y \rhd (x \rhd a)\rangle 
            \\&- \psi(y)\psi(x)\langle f,a\rangle + \psi(y)\langle f,x\rhd a\rangle + \psi(x)\langle f, y\rhd a\rangle - \langle f, x \rhd (y \rhd a)\rangle
            \\& = \psi([x,y])\langle f,a\rangle - \langle f,[x,y]\rhd a\rangle = \langle [x, y]\circ f,a\rangle,
        \end{split}
   \ee
so \(x \circ (y \circ f) - y \circ (x \circ f) = [x,y]\circ f\). In order to see that \(\circ\) provides algebra derivations, consider \(y \in \fG_-^\sigma\) and calculate:
\be
    \begin{split}
        &\langle x\circ (fg),e^{\hbar y}\rangle = \psi(x)\langle fg,e^{\hbar y}\rangle -\langle fg,x \rhd e^{\hbar y}\rangle \\&= \psi(x)\langle f \otimes g, \Delta (e^{\hbar y})\rangle-\langle f \otimes g, \Delta (x \rhd e^{\hbar y})\rangle \\&= \psi(x)(\langle f,e^{\hbar y}\rangle\langle g,e^{\hbar y}\rangle) - \langle f \otimes g, \Delta(x) \rhd \Delta(e^{\hbar y})\rangle
        \\&=\psi(x)(\langle f,e^{\hbar y}\rangle\langle g,e^{\hbar y}\rangle) - \langle f \otimes g,(x \rhd e^{\hbar y}) \otimes e^{\hbar y} + e^{\hbar y} \otimes (x \rhd e^{\hbar y})\rangle
        \\&=\langle g,e^{\hbar y}\rangle \psi(x)\langle f,e^{\hbar y}\rangle+ \langle f,e^{\hbar y}\rangle\psi(x)\langle g,e^{\hbar y}\rangle - \langle f,e^{\hbar y}\rangle\langle g,x \rhd e^{\hbar y} \rangle - \langle g,e^{\hbar y}\rangle\langle f, x \rhd e^{\hbar y}\rangle 
        \\&= \langle (x \circ f)g + f(x \circ g),e^{\hbar y}\rangle  
    \end{split}
\ee
Here we used that \(e^{\hbar y}\) is group-like, i.e.\ \(\Delta(e^{\hbar y}) = e^{\hbar y} \otimes e^{\hbar y}\).

The action \(\circ\) extends to an action of $U_R(\mathfrak G_+)[\![\hbar]\!]$ by the universal property of the universal envelope if additionally the action of $[x, s(r)]$ and $s(\psi(x)r)$ on $U_R(\mathfrak G_+)[\![\hbar]\!]$ coincide for every \(r \in R\). But this follows from the fact that
\be
    \begin{split}
        \langle x\circ (fr), a\rangle - \langle (x\circ f)r,a\rangle &=\psi(x) (\langle f, a\rangle r) -\langle f,x\rhd a \rangle r-\psi(x)(\langle f,a\rangle )r+\langle f,x\rhd a\rangle r\\&=\langle f,a\rangle \psi(x)(r)       
    \end{split}
\ee
holds.
\end{proof}

\begin{Rem}
    Note that
\be
\langle x\circ (rf), a\rangle -\langle r(x\circ f), a\rangle =-\langle f,(x\rhd a)r\rangle +\langle f,x\rhd (ar)\rangle =\langle f, a^{(1)}\psi (x\lhd a^{(2)})(r)\rangle ,
\ee
and therefore the action of $[x, t(r)]$ and $ t(\psi (x)r)$ do not coincide.
\end{Rem}

We would like to now define the smash-product of $U_R(\mathfrak G_+)[\![\hbar]\!]$ with $S_R (\mf G_-^{\sigma,\dagger})[\![\hbar]\!]$ over \(R[\![\hbar]\!]\). In other words, we would like to define an algebra structure on
\be
S_R (\mf G_-^{\sigma,\dagger})[\![\hbar]\!]\,\, \rltensor_{R[\![\hbar]\!]} U_R(\mathfrak G_+)[\![\hbar]\!]
\ee
where the right action of $R[\![\hbar]\!]$ on $S_R (\mf G_-^{\sigma,\dagger})[\![\hbar]\!]$ is by \(\C[\![\hbar]\!]\)-linear extension of the source map $s$. It is essentially defined by the relation $(1\otimes f) (x\otimes 1)=x\otimes f-1\otimes (x\circ f)$ for every \(f \in S_R (\mf G_-^{\sigma,\dagger})[\![\hbar]\!]\) and \(x \in \fG_+\). In particular, this relation already completely determines a Lie algebroid structure over $S_R (\mf G_-^{\sigma,\dagger})[\![\hbar]\!]$ on $ \fG_{+,\textnormal{ext}} \coloneqq S_R (\mf G_-^{\sigma,\dagger})[\![\hbar]\!]\, {}^s\!\!\otimes_{R[\![\hbar]\!]} \mathfrak G_+[\![\hbar]\!]$, where the anchor map is simply \(1 \otimes \psi\). We then define our algebra as the universal envelope 
\be
\Upsilon_\hbar^\sigma(\fd)\coloneqq U_{S_R (\mf G_-^{\sigma,\dagger})[\![\hbar]\!]}(\mathfrak G_{+, \textnormal{ext}}).
\ee
of \(\fG_{+,\textnormal{ext}}\) over \(S_R (\mf G_-^{\sigma,\dagger})[\![\hbar]\!]\).

\begin{Lem}
    There is an isomorphism of $S_R (\mf G_-^{\sigma,\dagger})[\![\hbar]\!]$-modules:
    \be
\Upsilon_\hbar^\sigma(\fd)\cong S_R (\mf G_-^{\sigma,\dagger})[\![\hbar]\!]\,{}^s\!\!\otimes_R U_R(\mathfrak G_+). 
    \ee
\end{Lem}

\begin{proof}
  The isomorphism is simply given by multiplication. That multiplication indeed provides an isomorphism follows by considering the associated graded object
    \be
    \begin{split}
        &\mathrm{Gr}\Upsilon_\hbar^\sigma(\fd)\cong S_{S_R (\mf G_-^{\sigma,\dagger})[\![\hbar]\!]}(\mathfrak G_{+, \textnormal{ext}})\cong S_R (\mf G_-^{\sigma,\dagger})[\![\hbar]\!] \,{}^s\!\!\otimes_R S_R(\mathfrak G_+)\\&\cong \mathrm{Gr}\left(S_R (\mf G_-^{\sigma,\dagger})[\![\hbar]\!]\,{}^s\!\!\otimes_{R[\![\hbar]\!]} U_R(\mathfrak G_+)[\![\hbar]\!]\right),        
    \end{split}
    \ee
    where the respective standard filtration were considered.
\end{proof}

Let us observe that on the other hand, the action of $T_R$ on $S_R (\mf G_-^{\sigma,\dagger})[\![\hbar]\!]$ is simply given by derivation action on $R$. We find, therefore, a canonical isomorphism of algebras:
\be
\Upsilon_\hbar^\sigma(\fd)\cong D_R\otimes Y_\hbar(\fd)[\![\lambda]\!]\,, \qquad D_R=U_R (T_R). 
\ee

\subsection{Defining the Hopf algebroid structure}

We will divide this construction into multiple steps. 

\subsubsection{The source and target maps}

First step, we need to construct the source and target maps. Note that in Section \ref{subsubsec:formalgroupoid}, we showed that $S_R (\mf G_-^{\sigma,\dagger})[\![\hbar]\!]$ admits two embeddings of $R[\![\hbar]\!]$, denoted by $s, t\colon R[\![\hbar]\!]\to S_R (\mf G_-^{\sigma,\dagger})[\![\hbar]\!]$, which are given by:
\be
\langle s(r)f,a\rangle =\langle f,a\rangle r\textnormal{ and } \langle t(r)f ,a\rangle =\langle f,ar\rangle \,,\qquad f \in S_R (\mf G_-^{\sigma,\dagger})[\![\hbar]\!], r \in R[\![\hbar]\!],a\in U(\fG_-^\sigma)[\![\hbar]\!]. 
\ee
Furthermore, explicit formulas for \(s\) and \(t\) were given in Proposition \ref{prop:coproduct_symmetric_algebra_explicit}.2.
By definition, $\Upsilon_\hbar^\sigma (\fd)$ contains $S_R (\mf G_-^{\sigma,\dagger})[\![\hbar]\!]$ as a subalgebra, therefore $s, t\colon R[\![\hbar]\!] \to \Upsilon_\hbar^\sigma (\fd)$ is defined to be the inclusion $s, t \colon R[\![\hbar]\!]\to S_R (\mf G_-^{\sigma,\dagger})[\![\hbar]\!]\to \Upsilon_\hbar^\sigma (\fd)$. Note that $s$ satisfies
\be
[x, s(r)]=s(\psi (x)r)\,,\qquad \forall x\in \mathfrak G_+, r\in R. 
\ee

\subsubsection{The coproduct}

The second step is to define a coproduct
\be
\Delta_{\hbar}^\sigma\colon \Upsilon_\hbar^\sigma (\fd)\to \Upsilon_\hbar^\sigma (\fd)\sttensor_{R[\![\hbar]\!]} \Upsilon_\hbar^\sigma (\fd),
\ee
that satisfies coassociativity,
\be
\Delta_{\hbar}^\sigma(a)(s(r)-t(r))=0, \qquad \forall a\in \Upsilon_\hbar^\sigma (\fd), r\in R[\![\hbar]\!],
\ee
and $\Delta_{\hbar}^\sigma(ab)=\Delta_{\hbar}^\sigma(a)\Delta_{\hbar}^\sigma(b)$. 

Note that on the subalgebra $S_R (\mf G_-^{\sigma,\dagger})[\![\hbar]\!]$, this is basically done in Section \ref{subsubsec:formalgroupoid}: we have a coassociative map of commutative algebras:
\be
\Delta_{\hbar}^\sigma\colon  S_R (\mf G_-^{\sigma,\dagger})[\![\hbar]\!]\to S_R (\mf G_-^{\sigma,\dagger})[\![\hbar]\!] \,\sttensor_{R[\![\hbar]\!]} S_R (\mf G_-^{\sigma,\dagger})[\![\hbar]\!].
\ee
Therefore, it remains to define \(\Delta_\hbar^\sigma\) on \(U(\fG_+)[\![\hbar]\!]\) in a compatible way. 

Recall from Section \ref{sec:lie_bialgebroid} that there is a unique element
\begin{equation}
    r(t_1,t_2;\lambda) \in \frac{1}{t_1-t_2}\sum_{\alpha = 1}^n I_\alpha \otimes I^\alpha + (\fg \otimes \fg^*)[\![t_1,t_2]\!][\![\lambda]\!]
\end{equation}
whose Taylor expansion in \(t_2 = 0\) provides an element of \(\widetilde{\fG}^{\sigma}_- \otimes \fg^*(\CO)\), where \(\widetilde{\fG}_-^\sigma\) is the image of \(\fG_-^\sigma\) under the canonical projection \(\fG = T_R \ltimes \fg(\CK) \to \fg(\CK)\). This element is a dynamical \(r\)-matrix in the sense that it satisfies \eqref{eq:dynamicalCYBE_r}.

We now want to utilize this \(r\)-matrix to construct \(\Delta_\hbar\) on \(U(\fG_+)[\![\hbar]\!]\) in a similar way as it was done in \cite{abedin2024yangian}. To do so, we observe that for any tensor
\begin{equation}
    \mathcal{T}= \sum_{i \in I} \mathcal{T}_i^{(1)} \otimes \mathcal{T}_i^{(2)} \in U_R(\fG)[\![\hbar]\!] {}^s\!\!\otimes_{R[\![\hbar]\!]}^s S_R(\fG^\dagger)[\![\hbar]\!].
\end{equation}
we can associate a \(\C[\![\hbar]\!]\)-linear endomorphism of \(U(\fG)[\![\hbar]\!]\) via
\begin{equation}\label{eq:evaluating_sstensor}
    x \longmapsto  \mathcal{T}(x)\coloneqq \sum_{i \in I} \langle\mathcal{T}_i^{(2)},x\rangle \mathcal{T}_i^{(1)} \,,\qquad x\in U_R(\fG)[\![\hbar]\!].
\end{equation}

\begin{Prop}\label{Prop:rembedTheta}
    The following results are true:
    \begin{enumerate}
        \item The exponential 
    \begin{equation}
        E \coloneqq \textnormal{exp}\left(\hbar r\right) \in U_R(\wt\fG_-^\sigma)[\![\hbar]\!] {}^s\!\!\otimes_{R[\![\hbar]\!]}^s S_R (\mf G_-^{\sigma,\dagger})[\![\hbar]\!].
    \end{equation}
    represents the unique \(R[\![\hbar]\!]\)-linear map \(U(\fG)[\![\hbar]\!] \to U(\fG)[\![\hbar]\!]\) such that \(e^{\hbar x} \mapsto e^{\hbar \pi(x)}\) for \(x \in \fG\), where \(\pi \colon \fG \to \fG\) is the projection onto \(\widetilde{\fG}^\sigma_-\).
    
    \item The exponential
    \begin{equation}
        \Theta \coloneqq \exp\left(\hbar \sum_{\alpha = 1}^N \pd_\alpha \otimes \omega_\alpha\right) \in D_R[\![\hbar]\!] \,{}^s\!\!\otimes_{R[\![\hbar]\!]}^s S_R (\mf G_-^{\sigma,\dagger})[\![\hbar]\!].
    \end{equation}
    represents the unique \(R[\![\hbar]\!]\)-linear map \(U(\fG)[\![\hbar]\!] \to U(\fG)[\![\hbar]\!]\) such that \(e^{\hbar x} \mapsto e^{\hbar \phi(x) }\) for \(x \in \fG\), where \(\phi \colon \fG \to \fG\) is \(R[\![\hbar]\!]\)-linear map defined by \(\xi_\alpha \to \pd_\alpha\) and \(\phi(\fG_+ \oplus \gout^\sigma) = \{0\}\).
    
    \item The product
    \be
        \CE \coloneqq E \Theta \in U_R(\mathfrak G)[\![\hbar]\!]\,{}^s\!\!\otimes_{R[\![\hbar]\!]}^s S_R (\mf G_-^{\sigma,\dagger})[\![\hbar]\!]\subseteq U_R(\mathfrak G)[\![\hbar]\!]\,{}^s\!\!\otimes_{R[\![\hbar]\!]}^s \Upsilon_\hbar^\sigma (\fd).
    \ee
    represents the extension of the embedding \(U(\fG_-^\sigma)[\![\hbar]\!]\to U(\fG)[\![\hbar]\!]\) by 1.
    \end{enumerate}
\end{Prop}
\begin{proof}
    The statements of 1.\ \& 2.\ are proven precisely as in \cite[Lemma 3.4.4.]{abedin2024yangian} using the fact that the identification of tensors with maps from \eqref{eq:evaluating_sstensor} is multiplicative since
    \begin{equation}\label{eq:evaluating_tensors_multiplicative}
        \langle f_1f_2, e^{\hbar x} \rangle a_1a_2 = \langle f_1 \otimes f_2, \Delta(e^{\hbar x})\rangle a_1 a_2 = \langle f_1,e^{\hbar x}\rangle \langle f_2,e^{\hbar x} \rangle a_1a_2
    \end{equation}
    holds for all \(f_1,f_2 \in S\) and \(a_1,a_2 \in U(\fG)[\![\hbar]\!]\).
    
    Result 3.\ then follows from the fact that \(U(\fG)[\![\hbar]\!] \cong U(\fg(\CO))[\![\hbar]\!] \sttensor_{R[\![\hbar]\!]} D_R[\![\hbar]\!]\) implies that an \(R[\![\hbar]\!]\)-linear map \(U(\fG)[\![\hbar]\!]\to U(\fG)[\![\hbar]\!]\) is completely determined by its values on elements \(e^{\hbar x}e^{\hbar y}\) for \(x \in T_R\) and \(y \in \fg(\CK)\). Now
    \begin{equation}
        \begin{split}
            E\Theta(e^{\hbar x}e^{\hbar y}) = E(e^{\hbar x}e^{\hbar y})\Theta(e^{\hbar x}e^{\hbar y}) = E(e^{\hbar y})\Theta(e^{\hbar y}) = e^{\hbar \pi(y)}e^{\hbar \phi(y)}
        \end{split}
    \end{equation}
    holds, where the first equality uses \eqref{eq:evaluating_tensors_multiplicative} again. In particular, we see that \(\CE\) maps \(e^{\hbar y}e^{\hbar \phi(y)}\) to itself for \(y \in \wt\fG^\sigma_-\), which concludes the proof by using \(U(\fG)[\![\hbar]\!] \cong U(\fg(\CO))[\![\hbar]\!] \sttensor_{R[\![\hbar]\!]} D_R[\![\hbar]\!]\) and \(\fG_-^\sigma = \{y + \phi(y) \mid y \in \wt\fG^\sigma_-\}\).
\end{proof}

Let us write \(\CE = \sum_{i \in I} \CE^{(1)}_i \otimes \CE^{(2)}_i\) for some index set \(I\).
Define $t_{\CE}\colon R[\![\hbar]\!]\to \Upsilon_\hbar^\sigma (\fd)$ and $s_{\CE}\colon R[\![\hbar]\!]\to \Upsilon_\hbar^\sigma(\fd)$ by
    \be
s_{\CE}(r)= \sum_{i\in I}\psi(\CE^{(2)}_i)(r)\CE^{(1)}_i \textnormal{ and }t_{\CE}(r)=\sum_{i\in I}\psi(\CE^{(1)}_i)(r)\CE^{(2)}_i\,,\qquad r \in R.
    \ee
    Here, $\psi$ is defined separately on $S_R (\mf G_-^{\sigma,\dagger})[\![\hbar]\!]$ and $U_R (\mathfrak G)[\![\hbar]\!]$ as in Section \ref{sec:generalities}; see Section \ref{sec:counit} for details. Jumping ahead of ourselves, we will show that this $\psi$ assembles into an anchor map for $\Upsilon_\hbar^\sigma (\fd)$. We now prove the following statement. 

\begin{Lem}\label{LemCEst}
For every \(r \in R[\![\hbar]\!]\), we have:
    \be\label{eq:CEst}
s_{\CE}(r)=s(r),\,t_{\CE}(r)=t(r), \textnormal{ and }\CE(s(r)\otimes 1-1\otimes t(r))=0.
     \ee
\end{Lem}

\begin{proof}
    The statement for $s_\CE$ is clear. For $t_\CE$, observe that this map takes values in \(S_R (\mf G_-^{\sigma,\dagger})[\![\hbar]\!]\) and evaluating in any \(a \in U(\fG_-^\sigma)[\![\hbar]\!]\) gives:
    \be
\langle t_{\CE}(r),a \rangle =\sum_{i \in I} \psi(\CE^{(1)}_i)(r)\langle a ,\CE^{(2)}_i\rangle = \psi(a)(r)=t(r)(a).
    \ee
    Here, the fact that \(\CE\) represents the embedding of \(U_R(\fG_-^\sigma)[\![\hbar]\!] \to U_R(\fG)[\![\hbar]\!]\) was used in the second equality. 
    
    To finish the proof, we calculate:
    \be
    \begin{split}
        \CE(1\otimes t(r))(e^{\hbar y})&=\sum_{i \in I}(\CE^{(1)}_i\otimes \CE_i^{(2)}t_{\CE}(r))(e^{\hbar y}) = \sum_{i\in I} \langle t_\CE(r),e^{\hbar y}\rangle\langle \CE_i^{(2)},e^{\hbar y}\rangle \CE_i^{(1)} 
        \\&=\sum_{i \in I}\psi(e^{\hbar y})(r)\CE_i^{(2)}(e^{\hbar y})\CE_i^{(1)} =e^{\hbar y}r \\& = \CE(s_{\CE}(r)\otimes 1) (e^{\hbar y}). 
    \end{split}
    \ee
    Here, \(y \in \fG_-^\sigma\) and we used the fact that \(e^{\hbar y}\) is group-like.
\end{proof}

Because of equation \eqref{eq:CEst}, left multiplication with $\CE$ defines a map
\be
\CE\colon U_R(\mathfrak G)[\![\hbar]\!]\, \sttensor_{R[\![\hbar]\!]} \Upsilon_\hbar^\sigma (\fd)\longrightarrow U_R(\mathfrak G)[\![\hbar]\!]\,{}^s\!\!\otimes_{R[\![\hbar]\!]}^s \Upsilon_\hbar^\sigma (\fd).
\ee
It is clearly invertible since $\CE$ is of the form $1+\CO(\hbar)$. Note that for a tensor
\begin{equation}
    \mathcal{T}= \sum_{i \in I} \mathcal{T}_i^{(1)} \otimes \mathcal{T}_i^{(2)} \in U_R(\mathfrak G)[\![\hbar]\!]\,\sttensor_{R[\![\hbar]\!]}S
\end{equation}
the evaluation on elements $e^{\hbar y}\in U_R(\mf G_-^\sigma)$ by
\be
\mathcal{T}(e^{\hbar y})=\sum_{i \in I}\psi(e^{-\hbar y})(\langle \CT^{(2)}_i,e^{\hbar Y}\rangle)\CT_i^{(1)}. 
\ee
Observe the difference to \eqref{eq:evaluating_sstensor} coming from the difference between \(\sttensor_{R[\![\hbar]\!]}\) and \(\sstensor_{R[\![\hbar]\!]}\).

Let $\Delta$ denote the symmetric coproduct on $U_R(\mathfrak D)[\![\hbar]\!]$. 

\begin{Lem}\label{CEdelta}
    For $\mathcal{T} \in \{\Theta,\CE\}$, we have 
    \be\label{eq:coproduct_of_E}
        (\Delta\otimes 1) \mathcal{T}=\mathcal{T}^{(13)}\mathcal{T}^{(23)} =\mathcal{T}^{(23)}\mathcal{T}^{(13)}\,,\qquad (1\otimes \Delta^\sigma_\hbar) \mathcal{T}=\mathcal{T}^{(13)}\mathcal{T}^{(12)}
    \ee
    and \((\epsilon\otimes 1)\mathcal{T}=1 = (1\otimes \epsilon) \mathcal{T}\). 
\end{Lem}

\begin{proof}
    We will prove the statement for \(\mathcal{T} = \CE\). The proof for \(\mathcal{T} = \Theta\) is similar, under consideration that the projection \(U_R(\fG) \to U_R(T_R) = D_R\) is a Hopf algebroid homomorphism.
    
    Observe that the both sides of the equalities in \eqref{eq:coproduct_of_E} can be understood as \(R[\![\hbar]\!]\)-linear maps \(U(\fG_-^\sigma)[\![\hbar]\!] \to (U(\fG) \overline{\otimes}_{R[\![\hbar]\!]} U(\fG))\) and \(U(\fG_-^\sigma)[\![\hbar]\!] \sttensor_{R[\![\hbar]\!]} U(\fG_-^\sigma)[\![\hbar]\!] \to U(\fG)[\![\hbar]\!]\) respectively. We prove the equalities by evaluating at \(e^{\hbar y}\) and \(e^{\hbar y_1} \otimes e^{\hbar y_2}\) respectively for \(y,y_1,y_2 \in \fG_-^\sigma\). 
    
    For the first equality, we calculate the LHS by
    \be
((\Delta\otimes 1 )\CE)(e^{\hbar y})=\sum_{i \in I}\langle \CE_i^{(2)},e^{\hbar y}\rangle\Delta(\CE_i^{(1)})=\Delta(e^{\hbar y}) = e^{\hbar y}\otimes e^{\hbar y},
    \ee
    and this coincides with the RHS because of:
    \be
        \begin{split}
            (\CE^{(13)}\CE^{(23)})(e^{\hbar y})&=\sum_{i,j \in I}(\CE_i^{(1)}\otimes \CE_j^{(1)}\otimes \CE_i^{(2)}\CE_j^{(2)})(e^{\hbar Y})
            =\sum_{i,j \in I}\langle \CE_i^{(2)}\CE_j^{(2)},e^{\hbar y}\rangle (\CE_i^{(1)} \otimes \CE_j^{(1)}) \\&= \sum_{i,j \in I}\langle \CE_i^{(2)},e^{\hbar y} \rangle \langle \CE_j^{(2)},e^{\hbar y}\rangle(\CE_i^{(1)} \otimes \CE_j^{(1)}) =e^{\hbar y}\otimes e^{\hbar y}.
        \end{split}
    \ee
Similarly, for the second equality we get on the RHS
\be
\begin{split}
    \CE^{(13)}\CE^{(12)}(e^{\hbar y_1} \otimes e^{\hbar y_2})&=\sum_{i,j \in I}(\CE^{(1)}_i\CE^{(1)}_j\otimes \CE^{(2)}_j\otimes \CE^{(2)}_i)(e^{\hbar y_1} \otimes e^{\hbar y_2}) \\&= 
    \sum_{i,j \in I}\langle \CE^{(2)}_i,e^{\hbar y_1}\langle \CE^{(2)}_j,e^{\hbar y_2}\rangle \rangle \CE^{(1)}_i\CE^{(1)}_j \\&= \sum_{j \in I}e^{\hbar y_1} \langle \CE_j^{(2)},e^{\hbar y_2}\rangle \CE_j^{(1)} = e^{\hbar y_1}e^{\hbar y_2}
\end{split}
\ee
and on the LHS:
    \be
    \begin{split}
        ((1\otimes \Delta_\hbar^\sigma) \CE) (e^{\hbar y_1}\otimes e^{\hbar y_2}) &= \sum_{i \in I} \langle \Delta_\hbar^\sigma(\CE_i^{(2)}),e^{\hbar y_1}\otimes e^{\hbar y_2} \rangle    \CE_i^{(1)} \\&= \sum_{i \in I}\langle \CE_i^{(2)},e^{\hbar y_1}e^{\hbar y_2}\rangle \CE_i^{(1)} = e^{\hbar y_1}e^{\hbar y_2}.
    \end{split}
    \ee
    The statement about counit is obvious from the definition of \(\CE\) as exponential of elements from \(\fG \otimes \fG^\dagger\). 
\end{proof}

Using this, we can define a map
\be
\Delta_{\hbar}^\sigma\colon U_R(\mathfrak G_+)\to U_R(\mathfrak G)[\![\hbar]\!]\sttensor_{R[\![\hbar]\!]} \Upsilon_\hbar^\sigma (\fd),
\ee
given by $\Delta_{\hbar}^\sigma(x)=\CE^{-1}\Delta(x)\CE$ for \(x \in U(\fG_+)[\![\hbar]\!]\). We claim:

\begin{Lem}\label{lem:Delta_on_U_welldefined}
For any $x\in \mathfrak G_+$ we have
\begin{equation}
    \CE^{-1}\Delta(x)\CE = 1 \otimes x + \phi(x)
\end{equation}
where \(\phi(x) \in U_R(\mathfrak G_+)\,{}^s\!\!\otimes_{R[\![\hbar]\!]}^t S_R(\fG_-^{\sigma,\dagger})[\![\hbar]\!]\) is the unique element satisfying \(\phi(x)(e^{\hbar y}) = x \lhd e^{\hbar y}\) for all \(y \in \fG_-^\sigma\).

In particular, we see that
\begin{equation}
     \CE^{-1}\Delta(x)\CE \in U_R(\mathfrak G_+)[\![\hbar]\!]\,\sttensor_{R[\![\hbar]\!]} \Upsilon_\hbar^\sigma (\fd) \subseteq \Upsilon_\hbar^\sigma (\fd)\,\sttensor_{R[\![\hbar]\!]}\Upsilon_\hbar^\sigma (\fd)
\end{equation}
holds for all \(x \in U_R(\fG_+)[\![\hbar]\!]\).
\end{Lem}

\begin{proof}
We can multiply the equality 
\begin{equation}
    \CE^{-1} \Delta(x) \CE = \CE^{-1} (x \otimes 1 + 1 \otimes x) \CE    
\end{equation}
with \(\CE\) from the right and subtract \(\CE(1 \otimes x)\) from both sides to obtain:
\be
        \Delta(x)\CE-\CE(1\otimes X)= [1\otimes x, \CE]+(x\otimes 1)\CE.
 \ee
 Here, by definition, $(x\otimes 1)\CE = \sum_{i\in I}x\CE^{(1)}_i\otimes \CE^{(2)}_i$ holds and although this is not necessarily well-defined in the correct tensor product, the sum 
 \be
    \sum_{i \in I}\left(\CE^{(1)}_i\otimes [x,\CE^{(2)}_i]+x\CE^{(1)}_i\otimes \CE^{(2)}_i\right) \in U_R(\fG)[\![\hbar]\!] \,\sstensor_{R[\![\hbar]\!]} S_R(\fG_-^{\sigma,\dagger})[\![\hbar]\!]
 \ee
is well-defined. Therefore, we can evaluate this on $e^{\hbar y}$ for \(y \in \fG_-^\sigma\) to get:
\be
\begin{split}
    &\sum_{i \in I}\left(\langle \CE^{(2)}_i,e^{\hbar y}\rangle x\CE^{(1)}+\langle [x, \CE^{(2)}_i],e^{\hbar y}\rangle \CE^{(1)}_i\right)
    \\&=\sum_{i \in I}\left(\langle \CE^{(2)}_i,e^{\hbar y}\rangle x\CE^{(1)}+\langle \CE^{(2)}_i,x \circ e^{\hbar y}\rangle \CE^{(1)}_i\right)
    \\&=\sum_{i \in I}\left(\langle \CE^{(2)}_i,e^{\hbar y}\rangle x\CE^{(1)}_i + \psi(x)(\langle \CE_i^{(2)},e^{\hbar y}\rangle) \CE_i^{(1)} -\langle\CE^{(2)}_i,x\rhd e^{\hbar y}\rangle \CE^{(1)}_i\right)
    \\& =\sum_{i \in I}\left(x \langle \CE^{(2)}_i,e^{\hbar y}\rangle \CE^{(1)}_i  -\langle\CE^{(2)}_i,x\rhd e^{\hbar y}\rangle \CE^{(1)}_i\right)
    \\&= xe^{\hbar y}-x\rhd e^{\hbar y}=e^{\hbar y}(x\lhd e^{\hbar y}).     
\end{split}
\ee
Here, the last equality follows used Proposition \ref{prop:matchedULie}.
We conclude that 
\begin{equation}
    \Phi(x)=\Delta(x)\CE-\CE(1\otimes x) \in U_R(\mathfrak G)\,{}^s\!\!\otimes_{R[\![\hbar]\!]}^s S_R(\fG_-^{\sigma,\dagger})[\![\hbar]\!]
\end{equation}
is the unique element whose evaluation on $e^{\hbar y}$ gives $e^{\hbar y}(x\lhd e^{\hbar y})$ for any \(y \in \fG_-^\sigma\).

Let us now consider $\phi(x)\coloneqq \CE^{-1}\Phi(x) \coloneqq \sum_{i \in J}x^{(1)}_i\otimes x^{(2)}_i$ for some index set \(J\) and \(R[\![\hbar]\!]\)-linearly independent elements $x^{(1)}_i\in U_R(\mathfrak G), x^{(2)}_i \in S_R(\fG_-^{\sigma,\dagger})[\![\hbar]\!]$. Our goal is to prove that $x^{(1)}_i\in \mathfrak G_+$ for all \(i \in J\). Let us compute
\be
\begin{split}
    \Phi(x) &= \sum_{i \in I,j\in J}(\CE^{(1)}_ix^{(1)}_j\otimes \CE^{(2)}_ix^{(2)}_j) (e^{\hbar y})=\sum_{i\in I,j \in J}\langle \CE^{(2)}_i,e^{\hbar y}\rangle \langle x_j^{(2)}, e^{\hbar y}\rangle \CE^{(1)}_ix^{(1)}_j\\&=\sum_{j \in J} \langle x_j^{(2)}, e^{\hbar y}\rangle e^{\hbar y}x^{(1)}_j= e^{\hbar y}\sum_{j \in J} \psi(e^{-\hbar y})(\langle x_j^{(2)},e^{\hbar y}\rangle )x^{(1)}_j,
\end{split}
\ee
where \(e^{-\hbar y}  r e^{\hbar y} = e^{-\hbar\textnormal{ad}(y)}r = \psi(e^{-\hbar y})r\) for any \(r \in R[\![\hbar]\!]\) was used. Therefore, 
\be
x\lhd e^{\hbar y}=\phi(x)(e^{\hbar y})=\sum_{j \in J}\psi(e^{-\hbar y})(\langle x_j^{(2)},e^{\hbar y}\rangle) x_j^{(1)},
\ee
showing that $x_i^{(1)}\in \mathfrak G_+$. Therefore, 
$\Delta_{\hbar}^\sigma(x) = 1 \otimes x + \phi(x) \in U_R(\mathfrak G_+)[\![\hbar]\!]\,\sttensor_{R[\![\hbar]\!]} \Upsilon_\hbar^\sigma (\fd)$ for every \(x \in \fG_+\).
\end{proof}

In order for the maps \(\Delta_\hbar^\sigma \colon U(\fG_+)[\![\hbar]\!] \to U_R(\mathfrak G_+)[\![\hbar]\!]\,\sttensor_{R[\![\hbar]\!]} \Upsilon_\hbar^\sigma (\fd)\) and \[\Delta_\hbar^\sigma \colon S_R(\fG_-^{\sigma,\dagger})[\![\hbar]\!] \to S_R(\fG_-^{\sigma,\dagger})[\![\hbar]\!] \,\sttensor_{R[\![\hbar]\!]} S_R(\fG_-^{\sigma,\dagger})[\![\hbar]\!]\] 
to glue together to a well-defined algebra homomorphism \(\Delta_\hbar^\sigma \colon \Upsilon_\hbar^\sigma(\fd) \to \Upsilon_\hbar^\sigma(\fd)\,\sttensor_{R[\![\hbar]\!]}\Upsilon_\hbar(\fd)\), we need the following:

\begin{Prop}
    The identities
    \be
[\Delta_{\hbar}^\sigma(x_1), \Delta_{\hbar}^\sigma(x_2)]=\Delta_{\hbar}^\sigma([x_1, x_2]) \textnormal{ and } [\Delta_{\hbar}^\sigma(x),\Delta_{\hbar}^\sigma(f)]=\Delta_{\hbar}^\sigma([x, f])
    \ee
    hold for all $x, x_1, x_2\in \mathfrak G_+$ and $f\in S_R(\fG_-^{\sigma,\dagger})[\![\hbar]\!]$. 
    
\end{Prop}

\begin{proof}
    The first identity is clear. Let us prove the second one. The RHS is an element in $S_R(\fG_-^{\sigma,\dagger})[\![\hbar]\!]\,\sttensor_{R[\![\hbar]\!]} S_R(\fG_-^{\sigma,\dagger})[\![\hbar]\!]$ we can evaluate it at \(e^{\hbar y_1} \otimes e^{\hbar y_2}\) for \(y_1,y_2 \in \fG_-^\sigma\) to obtain:
    \be\label{eq:mixed_Delta_LHS}
        \begin{split}
            \langle \Delta_{\hbar}^\sigma([x, f]),e^{\hbar y_1}\otimes e^{\hbar y_2}\rangle &= \langle [x, f],e^{\hbar y_1}e^{\hbar y_2}\rangle = \langle f,x \circ (e^{\hbar y_1}e^{\hbar y_2})\rangle \\&=\psi(x) \langle f, e^{\hbar y_1}e^{\hbar y_2}\rangle -\langle f,x\rhd (e^{\hbar y_1}e^{\hbar y_2})\rangle.
        \end{split}
    \ee
Now write $\Delta_{\hbar}^\sigma(x)=1\otimes x+\phi(x)$ as in Lemma \ref{lem:Delta_on_U_welldefined}, where \(\phi(x) = \sum_{j \in J}x_j^{(1)}\otimes x_j^{(2)}\) satisfies 
\begin{equation}\label{eq:phiXY}
    e^{\hbar y}(x \lhd e^{\hbar y}) = \sum_{i \in J}\langle x_i^{(2)},e^{\hbar y}\rangle e^{\hbar y}x_i^{(1)}
\end{equation}
for all \(y \in \fG_-^\sigma\); see the proof of Lemma \ref{lem:Delta_on_U_welldefined}. Write $\Delta_{\hbar}^\sigma(f)=\sum_{(f)}f^{(1)}\otimes f^{(2)}$ in Sweedler notation, then
\be\label{eq:mixed_Delta_for_welldefinitness}
 [\Delta_{\hbar}^\sigma(x), \Delta_{\hbar}^\sigma(f)] =\sum_{(f)}\left(f^{(1)}\otimes [x, f^{(2)}]+ \sum_{i \in J}\,[x_i^{(1)}, f^{(1)}]\otimes x^{(2)}_i f^{(2)}\right).
\ee
Evaluating this on $e^{\hbar y_1}\otimes e^{\hbar y_2}$ under consideration of Remark \ref{rem:evalutating_sttenosors}, the first term gives
\be\label{eq:XfRHST1}
\begin{aligned}
    &\sum_{(f)}\langle [x, f^{(2)}],e^{\hbar y_1} \langle f^{(1)},e^{\hbar y_2}\rangle \rangle = \sum_{(f)} \langle f^{(2)},x \circ \left(e^{\hbar y_1}\langle f^{(1)},e^{\hbar y_2}\rangle \right)\rangle  
    \\&= \sum_{(f)}\lp\psi(x)\langle f^{(2)},e^{\hbar y_1} \langle f^{(1)},e^{\hbar y_2}\rangle\rangle -\langle f^{(2)},x\rhd \left(e^{\hbar y_1} \langle f^{(1)},e^{\hbar y_2}\rangle\right)\rangle\rp
    \\&= \psi(x)\langle f,e^{\hbar y_1}e^{\hbar y_2}\rangle- \sum_{(f)}\langle f^{(2)}, (x\rhd e^{\hbar y_1} + e^{\hbar y_1} \psi(x\lhd e^{\hbar y_1}))\langle f^{(1)},e^{\hbar y_2}\rangle \rangle\\ &= \psi(x)\langle f,e^{\hbar y_1}e^{\hbar y_2}\rangle - \langle f, (x\rhd e^{\hbar y_1})e^{\hbar y_2}\rangle - \sum_{(f)} \langle f^{(2)}, e^{\hbar y_1} \psi(x\lhd e^{\hbar y_1})\langle f^{(1)},e^{\hbar y_2}\rangle \rangle.
\end{aligned}
\ee
The second term in \eqref{eq:mixed_Delta_for_welldefinitness} gives:
\be\label{eq:XfRHST2}
\begin{aligned}
    &\sum_{(f),i \in J}\langle x_i^{(2)} f^{(2)}, e^{\hbar y_1} \langle[x_i^{(1)}, f^{(1)}],e^{\hbar y_2}\rangle\rangle = \sum_{(f),i \in J}\langle x_i^{(2)},e^{\hbar y_1}\rangle \langle f^{(2)}, e^{\hbar y_1} \langle [x_i^{(1)}, f^{(1)}],e^{\hbar y_2}\rangle\rangle \\&
    =\sum_{(f),i \in J}\langle x_i^{(2)},e^{\hbar y_1}\rangle \langle f^{(2)}, e^{\hbar y_1} \langle f^{(1)},x_i^{(1)}\circ e^{\hbar y_2}\rangle\rangle
    \\&=\sum_{(f),i \in J}\langle x_i^{(2)},e^{\hbar y_1}\rangle \langle f^{(2)}, e^{\hbar y_1} \psi(x_i^{(1)})\langle f^{(1)},e^{\hbar y_2}\rangle - e^{\hbar y_1} \langle f^{(1)},x_i^{(1)}\rhd e^{\hbar y_2}\rangle\rangle.
\end{aligned}
\ee
Now using equation \eqref{eq:phiXY}, this becomes
\be
\sum_{(f)}\left(\langle f^{(2)}, e^{\hbar y_1} \psi(x\lhd e^{\hbar y_1})\langle f^{(1)},e^{\hbar y_2}\rangle \rangle - \langle f^{(2)}, e^{\hbar y_1} \langle f^{(1)},(x\lhd e^{\hbar y_1})\rhd e^{\hbar y_2}\rangle \rangle\right).
\ee
Combining equation \eqref{eq:XfRHST1} and \eqref{eq:XfRHST2}, we find that 
\be
\begin{split}&\langle[\Delta_{\hbar}^\sigma(x), \Delta_{\hbar}^\sigma(f)],e^{\hbar y_1}\otimes e^{\hbar y_2}\rangle \\&= \psi(x) \langle f,e^{\hbar y_1}e^{\hbar y_2}\rangle - \langle f,(x\rhd e^{\hbar y_1})e^{\hbar y_2}+ e^{\hbar y_1}((x\lhd e^{\hbar y_1})\rhd e^{\hbar y_2})\rangle.    
\end{split}
\ee
Using the identity
\be
x\rhd (e^{\hbar y_1}e^{\hbar y_2})=(x\rhd e^{\hbar y_1})e^{\hbar y_2}+ e^{\hbar y_1}((x\lhd e^{\hbar y_1})\rhd e^{\hbar y_2})
\ee
and \eqref{eq:mixed_Delta_LHS} we conclude \(\Delta^\sigma_\hbar([x,f]) = [\Delta_\hbar^\sigma(x),\Delta_\hbar^\sigma(f)]\).
\end{proof}

We have seen so far that \(\Delta_\hbar^\sigma\), which was defined separately on \(U_R(\fG_+)[\![\hbar]\!]\) and \(S_R(\fG_-^{\sigma,\dagger})[\![\hbar]\!]\), give rise to a well-defined algebra homomorphism:
\be
\Delta_{\hbar}^\sigma: \Upsilon_\hbar^\sigma (\fd)\longrightarrow \Upsilon_\hbar^\sigma (\fd) \sttensor_{R[\![\hbar]\!]} \Upsilon_\hbar^\sigma (\fd).
\ee
In order for \(\Delta_\hbar^\sigma\) to be a coproduct, it therefore remains to prove the coassociativity and compatibility with the counit, where the latter will be postponed to the next subsection.

\begin{Prop}
The algebra homomorphism $\Delta_{\hbar}^\sigma$ is coassociative. 
\end{Prop}

\begin{proof}
    The coassociativity for elements from $S_R(\fG_-^{\sigma,\dagger})[\![\hbar]\!]$ follows from associativity of multiplication on $U_R (\mf G_-^\sigma)$. To obtain coassociativity on $U_R(\mathfrak G_+)[\![\hbar]\!]$, we compute for every $x\in \mathfrak G_+$:
    \be
(\Delta_{\hbar}^\sigma\otimes 1 )\Delta_{\hbar}^\sigma (x)=(\Delta_{\hbar}^\sigma\otimes 1 )(\CE^{-1}\Delta(x)\CE)=\CE^{(12),-1}(\Delta \otimes 1)(\CE^{-1}\Delta(x)\CE)\CE^{(12)}.
    \ee
Lemma \ref{CEdelta} shows that the above is equal to
\be
    \begin{split}
        &\CE^{(12),-1}\CE^{(13),-1}\CE^{(23),-1} (\Delta \otimes 1)(\Delta (x)) \CE^{(23)}\CE^{(13)}\CE^{(12)}\\&=\CE^{(12),-1}\CE^{(13),-1}\CE^{(23),-1} (1\otimes \Delta )(\Delta (x)) \CE^{(23)}\CE^{(13)}\CE^{(12)}.
    \end{split}
\ee
On the other hand, Lemma \ref{CEdelta} and the coassociativity of \(\Delta\) imply
\begin{equation}
    \begin{split}
        (1 \otimes \Delta_\hbar^\sigma)\Delta_\hbar^\sigma(x) &= (1 \otimes \Delta_\hbar^\sigma)(\CE^{-1}\Delta(x)\CE) = \CE^{(12),-1}\CE^{(13),-1}(1 \otimes \Delta_\hbar^{\sigma})(\Delta(x))\CE^{(13)}\CE^{(12)} \\&=  \CE^{(12),-1}\CE^{(13),-1}\CE^{(23),-1} (1\otimes \Delta )(\Delta (x)) \CE^{(23)}\CE^{(13)}\CE^{(12)},
    \end{split}
\end{equation}
where in the last equality \(\Delta(x) \in U_R(\fG_+)[\![\hbar]\!] \overline{\otimes}_{R[\![\hbar]\!]}U_R(\fG_+)[\![\hbar]\!]\) was used. This concludes the proof.
\end{proof}

\subsubsection{The counit map}\label{sec:counit}

We now construct the counit map $\epsilon\colon \Upsilon_\hbar^\sigma (\fd)\to R[\![\hbar]\!]$. In fact, we construct the counit map together with the anchor map $\psi$. Note that both $U_R (\mathfrak G_+)[\![\hbar]\!]$ and $S_R(\fG_-^{\sigma,\dagger})[\![\hbar]\!]$ come equipped with an anchor map $\psi$. Namely, \(\psi\colon U_R(\fG_+)[\![\hbar]\!] \to \textnormal{End}_{\C[\![\hbar]\!]}(R[\![\hbar]\!])\) is simply the \(\C[\![\hbar]\!]\)-linear extension of the map \(U_R(\fG_+) \to \End(R)\) obtained from the anchor map \(\psi\) of \(\fG_+\) by the universal property of the universal envelope. On the other hand, the anchor map of \(S_R(\fG_-^{\sigma,\dagger})[\![\hbar]\!]\) is simply left multiplication by the evaluation in 1. We now just need to show that these anchor maps assemble into an anchor map for $\Upsilon_\hbar^\sigma (\fd)$, using the universal property of $\Upsilon_\hbar^\sigma (\fd)$. 

\begin{Lem}
   The anchor maps \(\psi\) on \(U_R(\fG_+)[\![\hbar]\!]\) and \(S\) give rise to a well-defined map $\psi \colon \Upsilon_\hbar^\sigma (\fd) \to \End_{\C[\![\hbar]\!]}(R[\![\hbar]\!])$ of  bimodules.
\end{Lem}

\begin{proof}
Because of the universal property of $U_{S_R(\fG_-^{\sigma,\dagger})[\![\hbar]\!]}(\mathfrak G_+\otimes_{R[\![\hbar]\!]} S_R(\fG_-^{\sigma,\dagger})[\![\hbar]\!])$, we only need to show
\be
[\psi (x), \psi (f)]=\psi( [x, f]) \,,\qquad \forall x \in \fG_+, f\in S_R(\fG_-^{\sigma,\dagger})[\![\hbar]\!].
\ee
For all \(r \in R\)
\be
\begin{split}
    [\psi (x), \psi (f)]( r)&=\psi(x)(\psi (f) (r))-\psi(f)(\psi(x)(r))
    \\&=\psi(x)(\langle f,1\rangle r) - \langle f,1\rangle \psi(x)(r)=\psi (x)(\langle f,1\rangle )r,    
\end{split}
\ee
whereas
\be
    \begin{split}
        \psi([x, f])(r)=\langle [x, f],1\rangle r= \langle f,x\circ 1\rangle r =
        (\psi (x) \langle f,1\rangle-\langle f, x\rhd 1\rangle)r=\psi (x) (\langle f,1\rangle )r.
    \end{split}
\ee
This provides well-definitness of \(\psi\).
Now, \(
\psi (s(r))=\psi (t(r))=r\) implies that \(\psi\) is a bimodule morphism, which concludes the proof.
\end{proof}

To finish the construction, we just need to show that $\psi$ satisfies
\be
\varphi_t (\Delta_{\hbar}^\sigma (h)\otimes r)=ht(r) \textnormal{ and } \varphi_s (\Delta_{\hbar}^\sigma (h)\otimes r)=hs(r)\,,\qquad h \in \Upsilon_\hbar^\sigma(\fd),r \in R[\![\hbar]\!], 
\ee
where $\varphi_t, \varphi_s$ are defined as in  equation \eqref{eq:varphits}. Let us verify this for $t$, since the statement for $s$ is similar. Let us first take $f\in S_R(\fG_-^{\sigma,\dagger})[\![\hbar]\!]$, and write $\Delta_{\hbar}^\sigma (f)=\sum_{(f)}f^{(1)}\otimes f^{(2)}$ in Sweedler notation. Then
\be
\varphi_t (\Delta_{\hbar}^\sigma f\otimes r)=\sum_{(f)}t(\psi (f^{(1)})r)f^{(2)}= \sum_{(f)} t(\langle f^{(1)},1\rangle )t(r) f^{(2)}.
\ee
To show that this is equal to $t(r)f$, we evaluate on $e^{\hbar y}$ for some $y\in \fG_-^\sigma$, which gives:
\be
\begin{split}
    &\sum_{(f)}\langle t(\langle f^{(1)},1\rangle)t(r) f^{(2)}, e^{\hbar y}\rangle =\sum_{(f)}\langle f^{(2)},e^{\hbar y} \langle f^{(1)},1\rangle r\rangle \\&= \langle \Delta_{\hbar}^\sigma(f) ,1\otimes e^{\hbar y}r\rangle =\langle f,e^{\hbar y}r\rangle =\langle t(r)f ,e^{\hbar y}\rangle.    
\end{split}
\ee
Let us now show that if $h\in \Upsilon_\hbar^\sigma (\fd)$ satisfies this equation, and $x\in \mathfrak G_+$, then $hx$ satisfies this equation. By induction, the above equation is then true for all $h\in \Upsilon_\hbar^\sigma (\fd)$, since \(\fG_+\) and \(S_R(\fG_-^{\sigma,\dagger})[\![\hbar]\!]\) topologically generate \(\Upsilon_{\hbar}^\sigma(\fd)\). Write $\Delta_{\hbar}^\sigma (x)=1\otimes x+ \sum_{i \in J}x_i^{(1)}\otimes x_i^{(2)}$, we have
\be
    \begin{split}
        \varphi_t (\Delta_{\hbar}^\sigma (hx)\otimes r)&=\sum_{(h)}\left(t(\psi (h^{(1)})(r))h^{(2)}x+\sum_{i\in J}t(\psi (h^{(1)}x_i^{(1)})(r)) h^{(2)}x^{(2)}_i\right)
        \\&=
        h t(r) x+  h \sum_{i\in J}t(\psi (x_i^{(1)})(r)x_i^{(2)}.
    \end{split}
\ee
It remains to show that
\be\label{eq:remains_to_prove_counit_section}
[x, t(r)]=\sum_{i\in J}t(\psi (x^{(1)}_i) r)x_i^{(2)}
\ee
holds for all $r\in R[\![\hbar]\!]$. To this end, we pair with $e^{\hbar y}$ again. Then the LHS equals
\be
\begin{split}
\langle [x,t(r)],e^{\hbar y}\rangle &= \langle t(r),x \circ e^{\hbar y}\rangle 
=\psi (x) \langle t(r),e^{\hbar y}\rangle-\langle t(r),x\rhd e^{\hbar y}\rangle \\&=\psi (x)(\psi (e^{\hbar y})(r))-\psi (x\rhd e^{\hbar y})(r).
\end{split}
\ee
Using Lemma \ref{lem:Delta_on_U_welldefined}, the evaluation of the RHS of \eqref{eq:remains_to_prove_counit_section} is
\be
\sum_{i\in J}\psi (e^{\hbar y})\psi (x_i^{(1)}) (r) \langle x_i^{(2)},e^{\hbar y}\rangle =\psi (e^{\hbar y} (x\lhd e^{\hbar y}))(r).
\ee
Now the desired equality follows from $xe^{\hbar y}=x\rhd e^{\hbar y}+e^{\hbar y} (x\lhd e^{\hbar y})$. 

Therefore, we can define the counit of \(\Upsilon_\hbar^\sigma(\fd)\) by $\epsilon(h)=\psi(h)(1_R)$. This finishes the construction of the Hopf algebroid structure on $\Upsilon_\hbar^\sigma (\fd)$. 

\subsubsection{The Hopf subalgebroid $\wt \Upsilon_\hbar^\sigma (\fd)$}

Since $U_R(\gol)$ is a subalgebra of $U_R(\mathfrak G_+)$, we can define the subalgebra $\wt \Upsilon_\hbar^\sigma (\fd)$ to be the subalgebra of $\Upsilon_\hbar^\sigma (\fd)$ generated by $U_R (\gol)[\![\hbar]\!]$ and $S_R(\fG_-^{\sigma,\dagger})[\![\hbar]\!]$. As an algebra, this is simply isomorphic to $Y_\hbar (\fd)[\![\lambda]\!]$, and the algebra $\Upsilon_\hbar^\sigma (\fd)$ is generated by $\wt\Upsilon_\hbar^\sigma (\fd)$ together with the algebroid of differential operators on $R=\C\lbb\lambda\rbb$. It is clear that $\wt\Upsilon_\hbar^\sigma (\fd)$ is in fact a Hopf subalgebroid of $\Upsilon_\hbar^\sigma (\fd)$. Indeed, the only nontrivial statement is to verify that the coproduct indeed restricts to this subalgebroid. This is because for any $x\in \gol$ and any $y\in \fG_-^\sigma$, the element $x\lhd e^{\hbar y}$ is still an element in $\gol$, due to the fact that $\gkl$ is an ideal inside the Lie algebroid $\mathfrak G=T_R\ltimes \gkl$. 

In fact, it is not difficult to see that under the identification of $\Upsilon_\hbar^\sigma (\fd)=D_R\otimes Y_\hbar (\fd)$, the commutation relation between $\pd_\alpha\in T_R$ with an element in $\wt\Upsilon_\hbar^\sigma (\fd)=Y_\hbar (\fd)\lbb\lambda\rbb$ is given by the action of $\pd_\alpha$ on $R=\C\lbb\lambda\rbb$. We introduce this subalgebra since in the next section, the twisting matrix and dynamical $R$-matrix will be valued in this subalgebra. 

\subsection{Classical limit}

The goal of this section, is to identify the classical limit of \(\Upsilon_\hbar^\sigma(\fd)\). We have:

\begin{Prop}
    The classical limit Lie bialgebroid of \(\Upsilon_\hbar^\sigma(\fd)\) is the Lie bialgebroid introduced in Section \ref{sec:lie_bialgebroid}.
\end{Prop}

\begin{proof}
    We first have to recognize the classical limit as an algebra. However, we have already seen that \(\Upsilon_\hbar^\sigma(\fd) \cong D_R\otimes_{R}Y_\hbar(\fd)[\![\lambda]\!]\) holds as algebras. But
    \begin{equation}
        Y_\hbar(\fd)[\![\lambda]\!]/\hbar Y_\hbar(\fd)[\![\lambda]\!] = U(\fd(\CO)) \otimes R    
    \end{equation}
    holds by virtue of \cite[Theorem 3.5]{abedin2024yangian}. So 
    \begin{equation}
        \Upsilon_\hbar^\sigma(\fd)/\hbar \Upsilon_\hbar^\sigma(\fd) \cong D_R \otimes_R U(\fd(\CO))[\![\lambda]\!] \cong U_R(\fD_+)    
    \end{equation} 
    holds as Hopf algebroids.

    The difference of the source and target map of \(\Upsilon_\hbar^\sigma(\fd)\) gives
    \begin{equation}
        \begin{split}
            \langle t(r)-s(r),e^{\hbar y}\rangle &= \langle 1,e^{\hbar y}r\rangle - \langle 1,e^{\hbar y}\rangle r = \langle 1, \psi(e^{\hbar y})(r)\rangle +\langle 1,e^{\hbar y}\rangle r - \langle 1,e^{\hbar y}\rangle r  \\&= \hbar \psi(y)(r_1) + \CO(\hbar^2)
        \end{split}
    \end{equation}
    for \(y \in \fG_-^\sigma\), \(f \in \fg^*(\CO)[\![\lambda]\!]\), and \(r = r_0 + \hbar r_1 + \CO(\hbar^2) \in R[\![\hbar]\!]\). The \(\hbar\)-order is therefore precisely the dual of the anchor map of \(\fG_-^\sigma\). Since the anchor map of \(\fD_-^\sigma = \fG_-^\sigma \oplus \fG_-^{\sigma,\dagger}\) factors over the projection \(\fD_-^\sigma \to \fG_-^\sigma\) and \(s,t\colon R[\![\hbar]\!]\to\Upsilon_\hbar^\sigma(\fd)\) factors over the embedding \(S_R(\fG_-^{\sigma,\dagger})[\![\hbar]\!] \to \Upsilon_\hbar^\sigma(\fd)\), this difference actually describes the dual of the anchor map of \(\fD_-^\sigma\).

    Next, combining propositions \ref{prop:coproduct_symmetric_algebra_explicit}, \ref{Prop:rembedTheta}, and \ref{prop:dual_differential}, we see that
    \begin{equation}
        \Delta_\hbar^\sigma - \Delta_\hbar^{\sigma,21} = \hbar d_\dagger^\sigma+\CO(\hbar^2).
    \end{equation}
    The dual of the anchor of \(\fD_-^\sigma\) and the dual differential \(d_\dagger^\sigma\) now equip \(\fD_+\) with the unique structure of a Lie bialgebroid, namely the one given in Section \ref{sec:lie_bialgebroid}. This concludes the proof.
\end{proof}

\subsection{Summary}

Let us summarize what has been done in this section. 

\begin{enumerate}
    \item We constructed an algebra $\Upsilon_\hbar^\sigma (\fd)$, which is essentially the \(R[\![\hbar]\!]\)-smash product of the Hopf algebroid $U_R (\mathfrak G_+)[\![\hbar]\!]$ with the algebra $S_R(\fG^{\sigma,\dagger}_-)[\![\hbar]\!] \supseteq S_R(\hbar \fG^{\sigma,\dagger}_-)[\![\hbar]\!] = U_R (\mf G_-^\sigma)^\dagger[\![\hbar]\!]$ over $R[\![\hbar]\!]$. As an algebra, it is identified with $D_R\otimes_{R} Y_\hbar (\fd)[\![\lambda]\!]$.

\item We then constructed a Hopf algebroid structure on this \(\Upsilon_\hbar^\sigma(\fd)\). The source and target map are constructed from the source and target map of $S_R(\fG^{\sigma,\dagger}_-)[\![\hbar]\!]$, which mimics the source and target map of the formal groupoid $\exp \lp \mf G_-^\sigma\rp$. The coproduct is constructed to mimic the convolution monoidal structure on the double quotient of groupoids 
\[\exp(\mathfrak G_+)\setminus\!\exp(\mathfrak G)/\exp(\mathfrak G_+).\]
The anchor map is constructed from the anchor map of $U_R (\mathfrak G_+)[\![\hbar]\!]$, together with the identity section map $S_R(\fG^{\sigma,\dagger}_-)[\![\hbar]\!]\to R[\![\hbar]\!]$. 

\item There is a canonical Hopf subalgebroid $\wt \Upsilon_\hbar^\sigma (\fd)$, which as is isomorphic to $Y_\hbar (\fd) \otimes R$ as an algebra. 

\item We showed that the Hopf algebroid $\Upsilon_\hbar^\sigma (\fd)$ is a quantization of the Lie bialgebroid structure on $\mf D_+$ introduced in Section \ref{sec:lie_bialgebroid}.  

\end{enumerate}

\section{Dynamical twist and the $R$-matrix}\label{sec:dynamicalR}

We have shown in the previous section that the algebra $\Upsilon_\hbar^\sigma (\fd)$ has the structure of a Hopf algebroid over $R[\![\hbar]\!]$ and, as an algebra, it is simply $D_R\otimes_{R} Y_\hbar (\fd)[\![\lambda]\!]$. This algebra can be given yet another Hopf algebroid structure, by combining the Hopf algebroid structure of $D_R[\![\hbar]\!]$ with the Hopf algebra structure of $Y_\hbar (\fd)$ described in Section \ref{sec:Yangian}. In this section, we show that the two are related by a dynamical twist in the sense we discussed in Section \ref{subsec:TwistHopf}. We then decompose the twist and construct a dynamical version of the meromorphic $R$-matrix. 

\subsection{The dynamical twist \(\CF\)}

Consider $\Upsilon_\hbar(\fd)\coloneqq D_R\otimes_{R} Y_\hbar (\fd)[\![\lambda]\!]$, given by the natural Hopf algebroid structure over $R[\![\hbar]\!]$. This Hopf algebroid is simply the tensor product of the Hopf algebroid structure on $D_R=U_R (T_R)$ with the Hopf algebra structure on $Y_\hbar (\fd)[\![\lambda]\!]$ over \(R[\![\hbar]\!]\). Its classical limit is the Lie bialgebroid obtained via the trivial dynamicalification discussed in Section \ref{subsec:trivialdyn}. There is an isomorphism of algebras $\Upsilon_\hbar(\fd)\cong \Upsilon_\hbar^\sigma (\fd)$. To distinguish them as Hopf algebroids, we will denote the Hopf algebroid structure of $\Upsilon_\hbar(\fd)$ by the tuple
\be
\lp \Upsilon_\hbar (\fd), s_\gamma, t_\gamma, m, \Delta_\hbar^\gamma, \psi\rp.
\ee
We will use the same notation as in Section \ref{sec:constructHopf} for the Hopf algebroid structure of $\Upsilon_\hbar^\sigma (\fd)$. Note that $s_\gamma=t_\gamma=s$, and the two Hopf algebroids have the same anchor map $\psi$. In this section, we construct a dynamical twist between $\Upsilon_\hbar(\fd)$ and $\Upsilon_\hbar^\sigma (\fd)$, and show that it restricts to a well-defined twist on $\wt \Upsilon_\hbar(\fd) \coloneqq Y_\hbar(\fd)[\![\lambda]\!] \cong \wt \Upsilon_\hbar^\sigma (\fd)$, where the isomorphism is again only on the level of algebras. 

Let $\CE_\gamma$ be the element introduced in \eqref{eq:CEgamma}, which obviously satisfies $\CE_\gamma (s(r)\otimes 1-1\otimes s(r))=0$ for all \(r \in R[\![\hbar]\!]\).
In the rest of the section, we prove the following theorem.  

\begin{Thm}\label{Thm:quantumtwist}
    
    The following statements are true.

    \begin{enumerate}

    \item $\CF\coloneqq \CE_\gamma^{-1}\CE\in U_R (\mathfrak G)[\![\hbar]\!]\,{}^s\!\!\otimes_{R[\![\hbar]\!]}^s S_R(\fG^{\sigma,\dagger}_-)[\![\hbar]\!]$ is an invertible element in the tensor product $U_R (\mathfrak G_+)[\![\hbar]\!]\,{}^s\!\!\otimes_{R[\![\hbar]\!]}^s S_R(\fG^{\sigma,\dagger}_-)[\![\hbar]\!]$, and can be viewed as an element in $\Upsilon_\hbar(\fd)\sttensor_{R[\![\hbar]\!]} \Upsilon_\hbar(\fd)$. 
    
    \item \(\CF^{-1}\Delta_\hbar^\gamma \CF=\Delta_{\hbar}^\sigma\), \((\Delta_\hbar^\gamma\otimes 1) (\CF)\CF^{(12)}=(1\otimes \Delta_\hbar^\gamma) (\CF)\CF^{(23)}\) and $(\epsilon\otimes 1)\CF=1 =(1\otimes \epsilon)\CF$ holds. 

    \item $R_\CF=R$, with $(s_\gamma)_\CF=s$ and $(t_\gamma)_\CF=t$ in the notation of Section \ref{subsec:TwistHopf}.
        
    \end{enumerate}
    In particular, $\CF$ is a twist between $\Upsilon_\hbar(\fd)$ and $\Upsilon_\hbar^\sigma (\fd)$. 
\end{Thm}

Let us start by proving Theorem \ref{Thm:quantumtwist}.1. Note that we can canonically identify $U_R (\mathfrak G_-^\sigma)$ and $U_R (t^{-1}\fg[t^{-1}]\lbb\lambda\rbb)$ as $U_R (\mathfrak G_+)$-modules, and both are identified with
\be
S_R (\mathfrak G/ \mathfrak G_+)\cong U_R (\mathfrak G)\otimes_{U_R (\mathfrak G_+)} R
\ee
as cocommutative coalgebras over $R$. From this perspective, we can think of $\CE$ and $\CE_\gamma$ as giving two embeddings of $S_R (\mathfrak G/ \mathfrak G_+)$ into $U_R (\mathfrak G)$, which induces two different algebra structures. Moreover, we have two splittings of the cocommutative coalgebra $U_R (\mathfrak G)$:
\be
\CE (S_R (\mf G/\mf G_+))\rltensor_R U_R (\mathfrak G_+) \cong U_R (\mathfrak G)\cong \CE_\gamma\lp S_R\lp \fg(\CK)[\![\lambda]\!]/\fg(\CO)[\![\lambda]\!]\rp\rp\rltensor_R U_R (\mathfrak G_+).
\ee

For any $x\in S_R (\mathfrak G/ \mathfrak G_+)$, we can obtain an element in $U_R (\fg_{<0}\lbb\lambda\rbb)$ by first embedding $x$ into $U_R (\mathfrak G)$ using $\CE$, then apply the counit on the first factor of the decomposition using $\CE_\gamma$. This is clearly an $R$-linear map producing an element in $U_R (\mathfrak G_+)$, which we can view as a tensor
\be
\CF'\in U_R (\mathfrak G_+)\,{}^s\!\!\otimes_{R[\![\hbar]\!]}^s S_R(\fG^{\sigma,\dagger}_-)[\![\hbar]\!].
\ee
The following lemma concludes the proof of Theorem \ref{Thm:quantumtwist}.1. 

\begin{Lem}\label{Lem:CFCF'}
    We can identify $\CF'=\CF$ and this tensor is invertible.  
\end{Lem}

\begin{proof}

 Let $y\in \mathfrak G_-^\sigma$. Since $e^{\hbar y}$ is an element in $U_R (\mathfrak G)$, there is a unique decomposition
    \be
e^{\hbar y}=e^{\hbar x_1}e^{\hbar x_2}, \qquad x_1\in t^{-1}\fg [t^{-1}][\![\lambda]\!]\lbb\hbar\rbb,  x_2\in \mathfrak G_+\lbb\hbar\rbb 
    \ee
    according to Lemma \ref{Lem:groupoid_like_elements}.
    Of course, the isomorphism between $U_R (\mathfrak G_-^\sigma)$ and $U_R (t^{-1}\fg [t^{-1}][\![\lambda]\!])$ is defined by $e^{\hbar y}\mapsto e^{\hbar x_2}$, and therefore the above implies
    \be
\CE=\CE_\gamma \CF'.
    \ee
    Now $\CF$ is invertible since $\CE_\gamma$ and $\CE$ are. The proof is complete. 
    
\end{proof}

The proof of Theorem \ref{Thm:quantumtwist}.2 is identical to the proof of \cite[Proposition 3.7]{abedin2024yangian}. Although the proof there is in the context of splitting of Lie algebras, the same argument goes through for the Lie algebroid splitting $\mf G=\mf G_+\oplus \mf G_-^\sigma$, using a similar statement about the universal enveloping algebra as in \ref{prop:matchedULie}. 

It remains to prove Theorem \ref{Thm:quantumtwist}.3. Let us write $\CF=\sum_{i \in I} \CF^{(1)}\otimes \CF^{(2)}$. We first show that $(s_\gamma)_\CF=s_\gamma=s$. For any $r\in R$, we have
\be
(s_\gamma)_\CF(r)=\sum_{i \in I}\psi (\CF_i^{(2)})(r)\CF_i^{(1)}=r(1\otimes \epsilon)(\CF)=r=s(r). 
\ee
On the other hand, for any $y\in \mf G_-^\sigma$, we have
\be
(t_\gamma)_\CF(r)(e^{\hbar y})=\sum_{i \in I}\psi (\CF_i^{(1)})(r)\langle\CF_i^{(2)},e^{\hbar y}\rangle =\psi (\CF (e^{\hbar y}))(r).
\ee
By Lemma \ref{Lem:CFCF'}, if we decompose $e^{\hbar y}=e^{\hbar x_1}e^{\hbar x_2}$ for $x_1\in t^{-1}\fg [t^{-1}][\![\lambda]\!]\lbb\hbar\rbb$ and $x_2\in \mathfrak G_+\lbb\hbar\rbb$, then $\CF (e^{\hbar y})=e^{\hbar x_2}$. Note however, that $\psi$ satisfies
\be
\psi(e^{\hbar y}) = \psi (e^{\hbar x_1}e^{\hbar x_2})=\psi (e^{\hbar x_1})\psi (e^{\hbar x_2})=\psi (e^{\hbar x_2}),
\ee
where we used the fact that $\psi (e^{\hbar x_1})=1$. Therefore
\be
\psi (\CF (e^{\hbar y}))(r)=\psi (e^{\hbar x_2})(r)=\psi (e^{\hbar y})(r)=t (r),
\ee
proving that $(t_\gamma)_\CF=t$. Finally, we have
\be
r \cdot_\CF r'=\psi (t(r'))r=\epsilon (t(r'))r=r'r
\ee
as desired. This finishes the proof of Theorem \ref{Thm:quantumtwist}.3, and concludes the proof of the whole theorem.

\subsection{Decomposing the twist}
Using the decomposition \(\CE = E\Theta\), where \(\Theta = \exp\left(\sum_{\alpha = 1}^N \partial_\alpha \otimes \omega_\alpha\right)\), from Proposition \ref{Prop:rembedTheta}, we have \(\CF = F\Theta\), where \(F \coloneqq \CE_\gamma^{-1}E\).

\begin{Prop}\label{prop:cocycle_F}
    The tensor $F$ satisfies
    \be\label{eq:cocycle_for_F}
        (1\otimes \Delta_\hbar^\gamma) (F)(\lambda) F^{(23)}(\lambda)=(\Delta_\hbar^\gamma\otimes 1 )(F)(\lambda) F^{(12)}(\lambda+\hbar \omega^{(3)}),
    \ee
    where the element $F^{(12)}(\lambda+\hbar \omega^{(3)})$ is given by
    \be\label{eq:dynamical_notation}
        F^{(12)}(\lambda+\hbar \omega^{(3)})\coloneqq (\Delta\otimes 1 )(\Theta) F^{(12)} (\Delta \otimes 1) (\Theta)^{-1} = \sum_{k = 0}^\infty \frac{\hbar^k}{k!} \sum_{i_1,\ldots, i_k}  \frac{\pd^k F}{\pd_{i_1}\cdots \pd_{i_k}} \otimes \omega_{i_1} \cdots \omega_{i_k}.
    \ee
\end{Prop}

\begin{proof}
For a proof of Equation \ref{eq:dynamical_notation}, see \cite[Lemma 7.2, Lemma 7.3]{xu2001quantum}. We turn to proving \eqref{eq:cocycle_for_F}. Theorem \ref{Thm:quantumtwist} states that
\be\label{eq:proof_of_cocycle_F1}
(\Delta_\hbar^\gamma\otimes 1)(F)(\Delta_\hbar^\gamma \otimes 1)(\Theta) \CF^{(12)}=(1\otimes \Delta_\hbar^\gamma) (F)(1\otimes \Delta_\hbar^\gamma) (\Theta) \CF^{(23)}.
\ee
Using $\CF^{-1}\Delta_\hbar^\gamma \CF=\Delta_{\hbar}^\sigma$ and Lemma \ref{CEdelta}, the RHS is equal to
\be
(1\otimes \Delta_\hbar^\gamma) (F)\CF^{(23)} (1\otimes \Delta_{\hbar}^\sigma)(\Theta)=(1\otimes \Delta_\hbar^\gamma) (F) F^{(23)}\Theta^{(23)} \Theta^{(13)}\Theta^{(12)}. 
\ee
On the other hand, since $(\Delta_\hbar^\gamma\otimes 1)(\Theta)=(\Delta\otimes 1)(\Theta)$, the LHS of \eqref{eq:proof_of_cocycle_F1} is equal to
\be
\begin{split}
    &(\Delta_\hbar^\gamma\otimes 1)(F)(\lambda) F^{(12)} (\lambda+\hbar \omega^{(3)}) (\Delta\otimes 1)(\Theta) \Theta^{(12)}\\&=(\Delta_\hbar^\gamma\otimes 1)(F)(\lambda) F^{(12)} (\lambda+\hbar \omega^{(3)}) \Theta^{(23)}\Theta^{(13)}\Theta^{(12)}.
\end{split}
\ee
Equating LHS with RHS we get the desired result. 

\end{proof}

Since $\CF$ is not valued in the subalgebroid $\wt \Upsilon_\hbar(\fd)$, in principle it does not make sense to conjugate by this $\CF$. The tensor element $F$ does, therefore conjugating by $F$ makes sense in this subalgebroid. Unfortunately, it does not induce a twist equivalence of Hopf algebroid. How do we deal with the fact that the twist $\CF$ is not inner for $\wt \Upsilon_\hbar(\fd)$?

We note that even though $\Theta$ is not an element in the above algebroid, conjugation by $\Theta$ does still make sense, namely it leaves the subspace $\wt\Upsilon_\hbar (\fd)$ invariant. In fact, conjugating by $\Theta$ is an automorphism of the algebra $\wt\Upsilon_\hbar (\fd)\otimes \wt\Upsilon_\hbar (\fd)$, which can be used to twist any module. 

We can therefore interpret conjugation by $\CF=F\Theta$ in the following two steps:
\begin{enumerate}
    \item Pulling back the tensor product modules along the automorphism $\Theta$.

    \item Apply the linear map $F$. 
    
\end{enumerate}
This process intertwines the monoidal structure of the two categories, and is clearly invertible. We therefore have the following result. 

\begin{Thm}
    There is an equivalence of monoidal categories
    \be
\wt\Upsilon_\hbar (\fd)\Mod\simeq \wt \Upsilon_\hbar^\sigma (\fd) \Mod,
    \ee
    realized by the twisting $\CF=F\Theta$. 
    
\end{Thm}

\subsection{The $R$-matrix}

We now construct a solution of the dynamical Yang-Baxter equation. Let
\be\label{eq:universal_R}
\CR(z, w; \lambda) \coloneqq \lp (\tau_z\otimes \tau_w) (F(\lambda))\rp^{(21), -1} \CR_\gamma(z-w) (\tau_z\otimes \tau_w)(F(\lambda)). 
\ee
Note that since $F$ is an element in $Y_\hbar(\fd)[\![\lambda]\!]\,{}^s\!\otimes_{R[\![\hbar]\!]}^sY_\hbar (\fd)\lbb\lambda\rbb$, the series $(\tau_z\otimes \tau_w) F$ is a formal power series in the formal variables $z$ and $w$. The multiplication $\CR_\gamma(z-w) (\tau_z\otimes \tau_w) (F(\lambda))$ is therefore ill-defined since $\CR_\gamma(z-w)$ is a power series in $(z-w)^{-1}$. To make sense of this, as in \cite[Section 4.3.1]{abedin2024yangian}, we can restrict ourselves to the action of these on a product of smooth modules that are finitely-generated and flat over $R\lbb\hbar\rbb$.

More precisely, we consider modules $M, N$ of $Y_\hbar (\fd)\lbb\lambda\rbb$ that are of the form $M=V[\![\lambda;\hbar\rbb$ and $N=U[\![\lambda;\hbar]\!]$ for some finite-dimensional vector spaces $V$ and $U$ and so that there is some $N>0$ such that $t^k I_\alpha$ and $t^kI^\alpha$ acts trivially on them for $k>N$ and \(\alpha \in \{1,\dots,N\}\). The tensor product $M\sttensor_{R\lbb\hbar\rbb} N$ can be naturally identified with $(V\otimes U)\lbb\lambda;\hbar\rbb$. Now $F(z,w)\coloneqq (\tau_z\otimes \tau_w )F$ acts on this space via some
\be
F_{M,N}(z,w;\lambda) \in \End (V\otimes U)\lbb\lambda;\hbar\rbb \lbb z,w\rbb=\End (V\otimes U) \lbb z,w\rbb\lbb\lambda;\hbar\rbb, 
\ee
while the element $\CR_\gamma(z-w)$ acts on this space via some
\be
\CR_{\gamma;M,N}(z-w)\in \End (V\otimes U) [(z-w)^{-1}][\![\hbar\rbb
\ee
by virtue of the defining equations \eqref{eq:R_s},\eqref{eq:expansion_in_Rs}, and \eqref{eq:definition_R}.
Both of these elements can be embedded into
\be\label{eq:Endsmooth}
\End (V\otimes U) \lbb z, w\rbb [(z-w)^{-1}] \lbb\lambda;\hbar\rbb.
\ee
Therefore, what we mean by \eqref{eq:universal_R} is actually the collection of elements
\begin{equation}\label{eq:dynamical_Rmatrix}
    \CR_{M,N}(z, w; \lambda) \coloneqq F_{M,N}(z,w;\lambda)^{(21),-1}R_{M,N}(z-w)F_{M,N}(z,w;\lambda) 
\end{equation}
of the algebra \(\End_{R[\![\hbar]\!]}(M\,\sttensor_{R[\![\hbar]\!]} N) = \End (V\otimes U) \lbb z, w\rbb [(z-w)^{-1}][\![\lambda;\hbar]\!]\) for every pair of smooth \(\Upsilon_\hbar^\sigma(\fd)\)-modules \(M\) and \(N\) that are finitely-generated and flat over \(R[\![\hbar]\!]\).

This matrix $\CR_{M,N}$, together with $\Theta (z, w)\coloneqq(\tau_z\otimes \tau_w)\Theta=(1\otimes\tau_w)\Theta$, defines an isomorphism
\be
M\otimes_{\Delta_{\hbar, z, w}^\sigma} N\simeq N\otimes_{\Delta_{\hbar, z, w}^\sigma}M,
\ee
where $\otimes_{\Delta_{\hbar, z, w}^\sigma}$ is the tensor product defined by $\Delta_{\hbar, z, w}^\sigma\coloneqq (\tau_z\otimes \tau_w) \Delta_\hbar^\sigma$.  Indeed, the two equations
\be
\Delta_\hbar^\gamma=\CF\Delta_\hbar^\sigma \CF^{-1}, \qquad \CR_\gamma(z-w)((\tau_z\otimes \tau_w)\Delta_\hbar^\gamma) \CR_\gamma(z-w)^{-1}=\lp (\tau_z\otimes \tau_w)\Delta_\hbar^\gamma\rp^{\textnormal{op}}
\ee
imply that
\be
\CF^{(21), -1}(z, w)\CR_\gamma(z-w) \CF(z,w)\Delta_{\hbar, z, w}^\sigma \CF(z,w)^{-1} \CR_\gamma(z-w)^{-1} \CF^{(21)}= \Delta_{\hbar, z, w}^{\sigma,\textnormal{op}}
\ee
holds. Here, $\CF(z,w)\coloneqq (\tau_z\otimes \tau_w )\CF = F(z.w)\Theta(z,w)$. The matrix we are using to conjugate is precisely
\be
\CF^{(21), -1}(z, w)\CR_\gamma(z-w) \CF(z,w)=\Theta^{(21), -1} (z, w)\CR (z, w;\lambda)\Theta (z,w), 
\ee
where actually all these equalities are understood on tensor products of finitely-generated smooth \(R[\![\hbar]\!]\)-modules using \eqref{eq:dynamical_Rmatrix}.

\begin{Thm}\label{thm:QDYBE}
    The family of tensors $\CR(z, w;\lambda)$ satisfies spectral dynamical quantum Yang-Baxter equation:
    \be\label{eq:dynamYB}
    \begin{split}
        &\CR^{(12)}(z_1, z_2;\lambda) \CR^{(13)}(z_1, z_3;\lambda+\hbar \omega^{(2)})\CR^{(23)} (z_2, z_3;\lambda)\\&=\CR^{(23)} (z_2, z_3;\lambda+\hbar\omega^{(1)}) \CR^{(13)}(z_1, z_3;\lambda)\CR^{(12)}(z_1, z_2;\lambda+\hbar\omega^{(3)}). 
    \end{split}
    \ee
    Here, the precise meaning of the notation, as well as its well-definiteness, is discussed in Remark \ref{rem:QDYBE} below.
\end{Thm}

\begin{Rem}\label{rem:QDYBE}
    Actually, by \eqref{eq:dynamYB}, we mean the family of equations depending on three smooth \(Y_\hbar(\fd)[\![\lambda]\!]\)-modules of the form $M=U\lbb\lambda;\hbar\rbb, N=V\lbb\lambda;\hbar\rbb, P=W\lbb\lambda;\hbar\rbb$ 
    for three vector spaces $U, V, W$ inside
    \be\label{eq:triple_product_smooth_modules}
        \End (U\otimes V\otimes W)\lbb z_1, z_2, z_3\rbb\left[\lp(z_1-z_2)(z_1-z_3)(z_2-z_3)\rp^{-1}\right]\lbb\lambda;\hbar\rbb,
    \ee
   where e.g.\ \(\CR^{(12)}(z_1,z_2;\lambda) = \CR_{M,N}(z_1,z_2;\lambda) \otimes 1\). The dynamical notations are defined analogous to \eqref{eq:dynamical_notation}.
        
\end{Rem}

\begin{proof}[Proof of Theorem \ref{thm:QDYBE}]
  All equalities in the following proof should be understood as equalities as in Remark \ref{rem:QDYBE} inside \eqref{eq:triple_product_smooth_modules} after acting on a tensor product on smooth modules which are finitely-generated and flat over $R\lbb\hbar\rbb$. Let 
  \be
    \Delta_{\hbar, z_1, z_2}^\gamma \coloneqq (\tau_{z_1}\otimes  \tau_{z_2})\Delta_\hbar^\gamma=(\tau_{z_1-z_2}\otimes 1 )\Delta_\hbar^\gamma \tau_{z_2}.
  \ee
  The cocycle conditions \cite[Theorem 4.22.2. \& 3.]{abedin2024yangian} take the form:
    \be
    \begin{split}
        &\left(\Delta_{\hbar, z_1-z_2,0}^\gamma \otimes 1\right)(\CR_\gamma(z_2-z_3))=\CR_\gamma^{(13)}(z_1-z_3) \CR_\gamma^{(23)}(z_2-z_3), \\&\left(1\otimes \Delta_{\hbar, z_1-z_2,0}^\gamma\right) (\CR_\gamma(z_1-z_3))=\CR_\gamma^{(13)}(z_1-z_3)\CR_\gamma^{(12)}(z_1-z_2). 
    \end{split}
    \ee
    Furthermore, \ref{eq:cocycle_for_F} can be written as:
    \be\label{eq:cocycle_F_adjusted}
        \begin{split}
            &\left(\Delta_{\hbar, z_1-z_2,0}^\gamma\otimes 1\right)(F(z_2, z_3;\lambda))F^{(12)}(z_1, z_2;\lambda+\hbar \omega^3)\\&=\left(1\otimes \Delta_{\hbar, z_2-z_3,0}^\gamma\right)(F(z_1, z_3;\lambda)) F^{(23)}(z_2, z_3;\lambda).
        \end{split}
    \ee
    
Let us apply the general argument of \cite{babelon1996quasi}, with spectral parameter inserted. Namely, we can think of $F(z_1, z_2;\lambda)$ as a twist for the coproduct $\Delta_{\hbar,z_1, z_2}^\gamma$, resulting in a quasi-Hopf algebra (in the presence of spectral parameters) with associativity isomorphism $\Phi = \Phi (z_1, z_2, z_3;\lambda)$ given by
\be
\left(\left(1\otimes \Delta_{\hbar,z_2-z_3, 0}^\gamma\right) (F(z_1, z_3;\lambda)) F^{(23)}(z_2, z_3;\lambda)\right)^{-1}\left(\Delta_{\hbar,z_1-z_2, 0}^\gamma\otimes 1\right)(F(z_2, z_3;\lambda))F^{(12)}(z_1, z_2;\lambda) 
\ee
for which $\CR$ is precisely the twisted spectral $R$-matrix. These \(\Phi\) must satisfy the quasi Yang-Baxter equation:
\be\label{eq:quasiYB}
\Phi^{(321), -1} \CR^{(12)} \Phi^{(312)} \CR^{(13)}\Phi^{(132), -1} \CR^{(23)}=\CR^{(23)}\Phi^{(231), -1} \CR^{(13)}\Phi^{(213)}\CR^{(12)}\Phi^{(123), -1},
\ee
where we omitted the spectral and dynamical parameters, which are inserted at the natural places of this equation. 

In our case, \eqref{eq:cocycle_F_adjusted} implies that
\be
\Phi(z_1,z_2,z_3;\lambda)=F^{(12),-1}(z_1, z_2, \lambda+\hbar \omega^3)F^{(12)}(z_1,z_2;\lambda).
\ee
Observe the independence of \(z_3\). From this we can deduce:
\be
\begin{aligned}
    &\Phi^{(213)}(\lambda) \CR^{(12)}(\lambda) \Phi^{(123), -1}(\lambda)\\& = F^{(21), -1}(\lambda+\hbar \omega^3) F^{(21)}(\lambda)F^{(21), -1}(\lambda)\CR_\gamma^{(12)} F^{(12)} (\lambda)F^{(12),-1}(\lambda) F^{(12)}(\lambda+\hbar \omega^3) \\ &= F^{(21), -1}(z_1, z_2, \lambda+\hbar \omega^3) \CR_\gamma^{(12)}  F^{(12)}(\lambda+\hbar \omega^3)=\CR^{(12)}(\lambda+\hbar \omega^3),
\end{aligned}
\ee
where we omitted the dependence on the spectral parameters \(z_1,z_2,z_3\).
Combining this equality with \eqref{eq:quasiYB}, we obtain \eqref{eq:dynamYB}, which completes the proof.
\end{proof}

\newpage

\appendix

\section{Notation}

\begin{itemize}
    \item \(\Sigma\) is a smooth projective irreducible complex algebraic curve, \(p \in \Sigma\) is a fixed point, \(\Sigma^{\circ} \coloneqq \Sigma \setminus \{p\}\).

    \item \(t\) is local coordinate at \(p \in \Sigma\) and we identify \(\CO \coloneqq \CO_{\Sigma,p} \cong \C[\![t]\!]\), \(\CK \coloneqq \CO_{\Sigma,p}[t^{-1}] \cong \C(\!(t)\!)\).
    
    \item \(G\) complex connected simple algebraic group, \(\fg\) is the Lie algebra of \(G\), and \(\kappa \colon \fg \times \fg \to \C\) is the Killing form of \(\fg\).
    
    \item \(G(\CO), G(\CK)\), and \(G(\Sigma^\circ) \coloneqq G(\CO(\Sigma^\circ))\) are the loop groups of regular functions from the formal disc around \(p\), the punctured formal disc around \(p\), and the affine curve \(\Sigma^\circ\) to \(G\) respectively.

    \item \(\fg(\CO) = \fg \otimes \CO, \fg(\CK) = \fg \otimes \CK\), and \(\fg(\Sigma^\circ) = \fg \otimes \CO_\Sigma(\Sigma^\circ)\) are the Lie algebras of \(G(\CO), G(\CK)\), and \(G(\Sigma^\circ)\) respectively.

    \item \(\Bun(\Sigma) = G(\CO) \setminus G(\CK) \,/\, G(\Sigma^\circ)\) is the moduli stack of \(G\)-bundles on \(\Sigma\) and \(\Bun^0(\Sigma)\) is the coarse moduli space of regularly stable \(G\)-bundles on \(\Sigma\).

    \item \(\CP\) is a fixed regularly stable \(G\)-bundle on \(\Sigma\), \(B\) is the completion of \(\Bun^0(\Sigma)\) at \(P\), and \(\sigma \colon B \to G(\CK)\) is a section of \(G(\CK) \to \Bun\) over \(B\).

    \item In Section \ref{sec:generalities}, \(A\) is a (potentially non-commutative) \(\C\)-algebra and \(R\) is a commutative \(\C\)-algebra. From Section \ref{sec:LieHitchin} onwards, \(R \coloneqq \C[\![\lambda_1,\dots,\lambda_N]\!] \coloneqq \C[\![\lambda]\!]\) for \(\lambda = (\lambda_1,\dots,\lambda_N)\), where \(N = \textnormal{dim}(B)\). We fix \(B \cong \textnormal{Spec}(\C[\![\lambda]\!])\) and write \(\partial_\alpha \coloneqq \partial/\partial\lambda_\alpha\) as well as \(T_R \coloneqq  \bigoplus_{\alpha = 1}^N\C[\![\lambda]\!]\partial_\alpha\) for the Lie algebra of derivations of \(R\) and \(D_R \coloneqq U_R(T_R)\) for the Hopf algebroid of differential operators of \(R\).

    \item In Section \ref{sec:generalities}, \(\fG\) is some Lie algebroid over \(R\). From Section \ref{sec:LieHitchin} onwards, \(\fG_+ = T_R \ltimes \fg(\CO)[\![\lambda]\!] \subseteq T_R \ltimes \fg(\CK)[\![\lambda]\!] \eqqcolon \fG\) and \(\fG^\sigma_-\) is the complement of $\fG_+$ inside $\fG$ defined by $\sigma$ in Section \ref{sec:lie_bialgebroid}. 

    \item In section \ref{sec:generalities}, \(\fD\) is some Courant algebroid. From Section \ref{sec:LieHitchin} onwards, \(\fD\) is the trivial double of \(\fG\), with $\fD_+ = \fG_+ \oplus \fG_+^\bot$, and $\fD_-^\sigma = \fG_-^\sigma \oplus \fG_-^{\sigma,\bot}$.
    Here, \((\cdot)^\bot\) is take with respect to the canonical pairing \(\langle, \rangle \colon \fD \times \fD \to R\). Observe that on \(\fd(\CK)[\![\lambda]\!]\), \(\langle ,\rangle \) is given in formulas by the residue of the \(\CK[\![\lambda]\!]\)-bilinear extension of the canonical pairing \(\fd \times \fd \to \C\).
    
    \item \((\cdot)^*\) denotes the (topological) dual with respect to \(\C\) and \((\cdot)^\dagger\) is the (topological) dual with respect to the base ring \(R\).

    \item \(Y_\hbar(\fd)\) is the Yangian of \(\fd\) and \(\Delta_\hbar^\gamma\) is its coproduct, and \(\CR_\gamma\) is its spectral \(R\)-matrix. Furthermore,
    \(\Upsilon_\hbar^\sigma(\fd)\) is the Hopf algebroid quantizing the Lie bialgebroid structure on $\fD_+$ determined by \(\fD = \fD_+ \oplus \fD_-^\sigma\), \(\Delta_\hbar^\gamma\) is its coproduct, and \(\CR\) is its spectral dynamical \(R\)-matrix. $\Upsilon_\hbar (\fd)$ is the trivial dynamicalification of $Y_\hbar (\fd)$ over $B$. The Hopf algebroids $\wt \Upsilon_\hbar^\sigma(\fd)$ and $\wt\Upsilon_\hbar (\fd)$ are the subalgebroids generated by $\fd (\CO)$ and $R$, i.e.\ the subalgebroids without differential operators. The standard symmetric coproduct of \(U_R(\fD)\), as well as its \(\C[\![\hbar]\!]\)-linear extension to \(U_R(\fD)[\![\hbar]\!]\), is simply denoted by \(\Delta\).

     \item Different tensors are used in this work. The undecorated tensor product $\otimes$ always means tensor product over $\C$, the tensor product of left modules (resp.\ right modules) over a commutative ring $R$ is denoted by $\otimes_R$, and the tensor product of bimodules is denoted by $\rltensor_R$. In the situation of Hopf algebroids, where the bimodule structure is induced by the source map \(s\) and target map \(t\), ${}^s\!\otimes_R^t$ is the tensor product is also used for the tensor product of bimodules (see Definition \ref{def:hopf_algebroid}) and if \(s = t\) we also write ${}^s\!\otimes_R^s$, which is also equal to $\otimes_R$. On the level of elements, all decorations are dropped in order to keep the notation clean, so we recommend the reader to stay aware in which space calculations under consideration take place.
    
\end{itemize}

\newpage

   \bibliographystyle{amsalpha}
   
   \bibliography{Yangian}

\newcommand{\etalchar}[1]{$^{#1}$}
\providecommand{\bysame}{\leavevmode\hbox to3em{\hrulefill}\thinspace}
\providecommand{\MR}{\relax\ifhmode\unskip\space\fi MR }
\providecommand{\MRhref}[2]{%
  \href{http://www.ams.org/mathscinet-getitem?mr=#1}{#2}
}
\providecommand{\href}[2]{#2}
\begin{thebibliography}{ABC{\etalchar{+}}24b}

\bibitem[ABC{\etalchar{+}}24a]{GLC2}
D~Arinkin, D~Beraldo, J~Campbell, L~Chen, J~Faergeman, Gaitsgory D, K~Lin,
  S~Raskin, and N~Rozenblyum, \emph{{Proof of the geometric Langlands
  conjecture II: Kac-Moody localization and the FLE}}, arXiv preprint
  arXiv:2405.03648 (2024).

\bibitem[ABC{\etalchar{+}}24b]{GLC4}
D~Arinkin, D~Beraldo, L~Chen, J~Faergeman, D~Gaitsgory, K~Lin, S~Raskin, and
  N~Rozenblyum, \emph{{Proof of the geometric Langlands conjecture IV:
  ambidexterity}}, arXiv preprint arXiv:2409.08670 (2024).

\bibitem[Abe24]{abedin2024r}
Raschid Abedin, \emph{{The $ r $-matrix structure of Hitchin systems via loop
  group uniformization}}, preprint, arXiv:2401.01327 (2024).

\bibitem[AN24]{abedin2024yangian}
Raschid Abedin and Wenjun Niu, \emph{{Yangian for cotangent Lie algebras and
  spectral $ R $-matrices}}, arXiv preprint arXiv:2405.19906 (2024).

\bibitem[BBB96]{babelon1996quasi}
O~Babelon, Denis Bernard, and E~Billey, \emph{{A quasi-Hopf algebra
  interpretation of quantum 3-j and 6-j symbols and difference equations}},
  Physics Letters B \textbf{375} (1996), no.~1-4, 89--97.

\bibitem[BD91a]{beilinson_drinfeld_quantization}
Alexander Beilinson and Vladimir Drinfeld, \emph{{Quantization Of Hitchin's
  Integrable System And Hecke Eigensheaves}}, 1991.

\bibitem[BD91b]{beilinson1991quantization}
\bysame, \emph{{Quantization of Hitchin’s integrable system and Hecke
  eigensheaves}}, 1991.

\bibitem[BZN16]{Ben-Zvi:2016mrh}
David Ben-Zvi and David Nadler, \emph{{Betti Geometric Langlands}}.

\bibitem[CCF{\etalchar{+}}24]{GLC3}
Justin Campbell, Lin Chen, Joakim Faergeman, Dennis Gaitsgory, Kevin Lin, Sam
  Raskin, and Nick Rozenblyum, \emph{{Proof of the geometric Langlands
  conjecture III: compatibility with parabolic induction}}, arXiv preprint
  arXiv:2409.07051 (2024).

\bibitem[Cos14]{costello2014integrable}
Kevin Costello, \emph{{Integrable lattice models from four-dimensional field
  theories}}, Proc. Symp. Pure Math, vol.~88, 2014, pp.~3--24.

\bibitem[CW19]{cautis2019cluster}
Sabin Cautis and Harold Williams, \emph{{Cluster theory of the coherent Satake
  category}}, Journal of the American Mathematical Society \textbf{32} (2019),
  no.~3, 709--778.

\bibitem[CW23]{cautis2023canonical}
\bysame, \emph{{Canonical bases for Coulomb branches of 4d $\mathcal{N}=2$
  gauge theories}}, arXiv preprint arXiv:2306.03023 (2023).

\bibitem[CWY18a]{costello2018gauge}
Kevin Costello, Edward Witten, and Masahito Yamazaki, \emph{{Gauge Theory And
  Integrability, I}}, Notices of the International Consortium of Chinese
  Mathematicians \textbf{6} (2018), no.~1, 46--119.

\bibitem[CWY18b]{costello2018gauge2}
\bysame, \emph{{Gauge Theory and Integrability, II}}, ICCM Not. \textbf{6}
  (2018), no.~arXiv: 1802.01579, 120--146.

\bibitem[CY19]{costello2019gauge}
Kevin Costello and Masahito Yamazaki, \emph{{Gauge theory and integrability,
  III}}, arXiv preprint arXiv:1908.02289 (2019).

\bibitem[DP06]{Donagi:2006cr}
Ron Donagi and Tony Pantev, \emph{{Langlands duality for Hitchin systems}}.

\bibitem[Dri85]{Drinfeld:1985rx}
V.~G. Drinfeld, \emph{{Hopf algebras and the quantum Yang-Baxter equation}},
  Sov. Math. Dokl. \textbf{32} (1985), 254--258.

\bibitem[Dri86]{drinfeld1986quantum}
Vladimir Drinfeld, \emph{{Quantum groups}}, Zapiski Nauchnykh Seminarov POMI
  \textbf{155} (1986), 18--49.

\bibitem[Dri89]{drinfeld1989quasi}
\bysame, \emph{{Quasi-hopf algebras}}, Algebra i Analiz \textbf{1} (1989),
  no.~6, 114--148.

\bibitem[EV]{etingof_varchenko}
Pavel Etingof and Alexander Varchenko, \emph{{Geometry and Classificatin of
  Solutions of the Classical Dynamical Yang–Baxter Equation}}, Communications
  in Mathematical Physics \textbf{192}, 77--120.

\bibitem[Fel98]{felder_kzb}
Giovanni Felder, \emph{{The KZB equations on Riemann surfaces}}, Sym\'etries
  quantiques (Les Houches, 1995), North-Holland, Amsterdam (1998), 687--725.

\bibitem[FGV02]{frenkel2002geometric}
Edward Frenkel, Dennis Gaitsgory, and Kari Vilonen, \emph{{On the geometric
  Langlands conjecture}}, Journal of the American Mathematical Society
  \textbf{15} (2002), no.~2, 367--417.

\bibitem[Gai13]{gaitsgory2013outline}
Dennis Gaitsgory, \emph{{Outline of the proof of the geometric Langlands
  conjecture for GL (2)}}, arXiv preprint arXiv:1302.2506 (2013).

\bibitem[GR24a]{GLC1}
D~Gaitsgory and S~Raskin, \emph{{Proof of the geometric Langlands conjecture I:
  construction of the functor}}, arXiv preprint arXiv:2405.03599 (2024).

\bibitem[GR24b]{GLC5}
\bysame, \emph{{Proof of the geometric Langlands conjecture V: the multiplicity
  one theorem}}, arXiv preprint arXiv:2409.09856 (2024).

\bibitem[Hit87]{Hitchin:1987mz}
Nigel~J. Hitchin, \emph{{Stable bundles and integrable systems}}, Duke Math. J.
  \textbf{54} (1987), 91--114.

\bibitem[Kap06]{kapustin2006holomorphic}
Anton Kapustin, \emph{{Holomorphic reduction of N= 2 gauge theories, Wilson-'t
  Hooft operators, and S-duality}}, arXiv preprint hep-th/0612119 (2006).

\bibitem[KS02]{karolinsky_Stolin}
E.~Karolinsky and A.~Stolin, \emph{{Classical Dynamical r-Matrices, Poisson
  Homogeneous Spaces, and Lagrangian Subalgebras}}, Letters in Mathematical
  Physics \textbf{60} (2002), 257--274.

\bibitem[Las98]{laszlo_hitchin_wzw}
Yves Laszlo, \emph{{Hitchin's and WZW Connections are the same}}, J.
  Differential Geometry \textbf{49} (1998), 547--576.

\bibitem[Lur09]{Lurie:2009keu}
Jacob Lurie, \emph{{On the Classification of Topological Field Theories}}.

\bibitem[LWX97]{liu_weinstein_xu_manin_triples}
Zhang-Ju Liu, Alan Weinstein, and Ping Xu, \emph{{Manin triples for Lie
  bialgebroids}}, J. Differential Geometry \textbf{45} (1997), 547--574.

\bibitem[Maj90]{majid1990physics}
Shahn Majid, \emph{{Physics for algebraists: Non-commutative and
  non-cocommutative Hopf algebras by a bicrossproduct construction}}, Journal
  of Algebra \textbf{130} (1990), no.~1, 17--64.

\bibitem[MM10]{moerdijk2010universal}
Ieke Moerdijk and Janez Mr{\v{c}}un, \emph{{On the universal enveloping algebra
  of a Lie algebroid}}, Proceedings of the American Mathematical Society
  \textbf{138} (2010), no.~9, 3135--3145.

\bibitem[Niu22]{niu2022local}
Wenjun Niu, \emph{{Local operators of 4d $\mathcal{N}=2$ gauge theories from
  the affine Grassmannian}}, Adv. Theor. Math. Phys. \textbf{26} (2022),
  no.~arXiv: 2112.12164, 3207--3247.

\bibitem[PT24]{Padurariu:2024zff}
Tudor P\u{a}durariu and Yukinobu Toda, \emph{{Quasi-BPS categories for Higgs
  bundles}}.

\bibitem[Xu01]{xu2001quantum}
Ping Xu, \emph{{Quantum groupoids}}, Communications in Mathematical Physics
  \textbf{216} (2001), no.~3, 539--581.

\end{thebibliography}

\end{document}